\documentclass[12pt,reqno]{amsart}
\usepackage{latexsym}
\usepackage{amssymb}
\usepackage{amscd}

\numberwithin{equation}{section}

\newtheorem{theorem}{Theorem}
\newtheorem{proposition}[theorem]{Proposition}
\newtheorem{lemma}[theorem]{Lemma}
\newtheorem{corollary}[theorem]{Corollary}
\newtheorem{conjecture}[theorem]{Conjecture}

\theoremstyle{remark}
\newtheorem*{remark}{Remark}
\newtheorem{remarknu}[theorem]{Remark}

\def\al{\alpha}
\def\be{\beta}

\def\ga{\gamma}

\def\om{\omega}

\def\Z{{\mathbb Z}}

\def\coef#1{\left\langle#1\right\rangle}

\def\fl#1{\left\lfloor#1\right\rfloor}
\def\cl#1{\left\lceil#1\right\rceil}

\begin{document}
\title[Romik's sequence of Taylor
  coefficients of Jacobi's theta function $\theta_3$]{The
  congruence properties of Romik's sequence of Taylor
  coefficients of Jacobi's theta function $\theta_3$}
\author[C. Krattenthaler]{C. Krattenthaler$^\dagger$}

\address{Fakult\"at f\"ur Mathematik, Universit\"at Wien,
Oskar-Morgenstern-Platz~1, A-1090 Vienna, Austria.
WWW: {\tt http://www.mat.univie.ac.at/\lower0.5ex\hbox{\~{}}kratt}.}

\author[T.\,W. M\"uller]{T. W. M\"uller} 

\address{Fakult\"at f\"ur Mathematik, Universit\"at Wien,
Oskar-Morgenstern-Platz~1, A-1090 Vienna, Austria.}

\thanks{$^\dagger$Research partially supported by the Austrian
Science Foundation FWF, grant S50-N15,
in the framework of the Special Research Program
``Algorithmic and Enumerative Combinatorics"}

\subjclass[2020]{Primary 11F37;
Secondary 11B83 14K25}

\keywords{Modular forms of half integer weight,
  Jacobi theta function, Taylor coefficients, congruences}

\begin{abstract}
In [{\em Ramanujan J.} {\bf52} (2020), 275--290], Romik considered the
Taylor expansion of Jacobi's theta function $\theta_3(q)$ at $q=e^{-\pi}$ and encoded it
in an integer sequence $(d(n))_{n\ge0}$ for which he provided a recursive procedure
to compute the terms of the sequence.
He observed intriguing behaviour of $d(n)$ modulo primes and prime powers.
Here we prove (1) that $d(n)$ eventually vanishes modulo any prime power $p^e$
with $p\equiv3$~(mod~4), (2) that $d(n)$ is eventually periodic
modulo any prime power $p^e$
with $p\equiv1$~(mod~4), and (3) that $d(n)$ is purely periodic
modulo any 2-power~$2^e$. Our results also provide more detailed information on
period length, respectively from when on the sequence vanishes or becomes periodic.
The corresponding bounds may not be optimal though, as computer data suggest.
Our approach shows that the above congruence properties
hold at a much finer, polynomial level.
\end{abstract}
\maketitle

\section{Introduction}
\label{sec:1}

The focus of this article is on Jacobi's theta function $\theta_3$
defined by (cf.\ \cite[top of p.~464 with $z=0$]{WhWaAA}) 
$$
\theta_3(\tau)=\sum_{n=-\infty}^\infty q^{n^2}, \quad \text{with $q=e^{i\pi\tau}$}.
$$
In \cite{RomikAA}, Romik considered the Taylor expansion of
$\theta_3(\tau)$ at $\tau=i$ in the form (cf.\ \cite[display below Eq.~(8)]{RomikAA})
\begin{equation} \label{eq:Taylor}
\theta_3\left(i\frac {1+z} {1-z}\right)=
\theta_3(i)(1-z)^{1/2}\sum_{n=0}^\infty \frac {d(n)} {(2n)!}\Phi^{n}z^{2n},
\end{equation}
where $\Phi=\Gamma^8(1/4)/(128\pi^4)$ and, as is well-known
(cf.\ \cite[p.~325, Entry~1(i)]{Bern5}),
$\theta_3(i)=\pi^{1/4}/\Gamma(3/4)$. He showed that the sequence
$\big(d(n)\big)_{n\ge0}$ is an integer sequence, and he provided a
highly non-trivial recursive procedure for computing the
coefficients~$d(n)$ (see Section~\ref{sec:uvr}).
The first few values turn out to be
\begin{multline*}
1, 1, -1, 51, 849, -26199, 1341999, 82018251, 18703396449, \\
-993278479599, -78795859032801,
38711746282537251, -923351332174412751, \dots
\end{multline*}
The signs of these numbers seem very irregular. (In fact, Problem~5 in Section~8 of
\cite{RomikAA} asks for figuring out the pattern of signs. So far,
there has been no progress on that question.) However,
computer experiments led Romik~\cite[Conj.~13]{RomikAA} to conjecture
that $\big(d(n)\big)_{n\ge0}$ is (eventually) periodic when taken
modulo a prime~$p$ with $p\equiv1$~(mod~4), and that the sequence
eventually vanishes when taken
modulo a prime~$p$ with $p\equiv3$~(mod~4). More computer experiments
suggest (cf.\ also \cite[Conjecture~18(3)]{WakhAA})
that analogous assertions hold modulo {\it any} prime {\it
power} (including powers of~2). To be more precise, such experiments suggest that:

\begin{enumerate} 
\item $d(n)$ eventually vanishes modulo any prime power $p^e$
  with $p\equiv3$~(mod~4);
\item $d(n)$ is eventually periodic modulo any prime power $p^e$
  with $p\equiv1$~(mod~4);
\item $d(n)$ is purely periodic modulo any 2-power~$2^e$.
\end{enumerate}

Item~(1) was proved for primes (i.e., for $e=1$) by
Scherer~\cite{ScheAA}. He also obtained partial results on Items~(2)
and~(3) by proving that
$d(n)\equiv(-1)^{n+1}$~(mod~5) for $n\ge1$, and that $d(n)$ is odd for
all~$n$.
Guerzhoy, Mertens and Rolen~\cite{GuMRAA} claim to have proved Item~(2) 
in full, as a special case of a more general result for a whole family of modular forms
of half integer weight.\footnote{The authors of the present article admit that they
are not able to comprehend what the result(s) in~\cite{GuMRAA} say about Romik's
sequence $\big(d(n)\big)_{n\ge0}$.
The first problem is that $\tilde\Omega$ in Theorem~1.3 of~\cite{GuMRAA}
is nowhere defined. One may guess that $\tilde\Omega^2 = \omega$,
with the $\omega$
in the proof of Theorem~1.3 on pages~148/149 of~\cite{GuMRAA}.
We assume this guess in the following.

In the Remark on top of page~150 of~\cite{GuMRAA}, Romik's sequence is
addressed explicitly.
There, $\omega$ is computed as $\Gamma^2(1/4)/(8 \pi)^{1/2}$.

In Theorem 1.3, the coefficients $\partial^n f(\tau_0)$ are considered.
If, in the case of $f(\tau)=\Theta(\tau)=\theta_3(2\tau)$,
we make the comparison of coefficients
in (1-1) and (1-3) of~\cite{GuMRAA}, then we obtain the relation
$$
   (2 \pi)^n \partial^n \Theta(i/2) = \Theta(i/2) \Phi^n D(n),
$$
where we write $D(n)$ for the Taylor coefficients in Romik's series,
that is, $D(2n)=d(n)$ and\break $D(2n+1)=0$.

In order to apply Theorem~1.3 to Romik's sequence, we must choose $k=1$
there. In particular, we should divide the above relation by
$\tilde\Omega^{4n+1}$:
$$
\frac {(2 \pi)^n \partial^n \Theta(i/2)} {\tilde\Omega^{4n+1}}
=\frac {\Theta(i/2) \Phi^n D(n)} {\tilde\Omega^{4n+1}}.
$$
Now we use the assumption that $\tilde\Omega = \omega^{1/2}$ and the relation
$\omega^2 = 2^{1/2} \pi \Phi$ from the Remark on top of page 150 of~\cite{GuMRAA}.
We substitute and get, after some cancellation,
$$
\frac {(2 \pi)^n \partial^n \Theta(i/2)} {\tilde\Omega^{4n+1}}
=\frac {D(n)} {2^{n/2-1/4} \pi^{n+1/2}}.
$$
Here is the first problem: $\pi$ does not drop out.

This could be easily fixed: just
change the normalisation by that power of $\pi$. However, that
would create another, equally serious problem. Then the (re)normalised
sequence equals Romik's sequence (with 0s for odd
indexed coefficients) up to some power of $2^{1/2}$.
Then Theorem~1.3, say with $A=0$,
predicts a period length of the (re)normalised sequence, taken
modulo~$p$, of $2(p-1)$.
(The factor~2 comes from the period length of~$2^{1/2}$ modulo~$p$.)
Translated to Romik's original sequence $\big(d(n)\big)_{n\ge0}$ (i.e.,
without 0s), this implies
a period length of $p-1$. However, this is definitely wrong. For
example, for $p=13$, Romik's sequence, taken modulo~13,
begins
\begin{multline*}
1, 1, 12, 12, 4, 9, 9, 3, 10, 10, 12, 1, 1, 9, 4, 4, 10, 3, 3,\\
   1, 12, 12, 4, 9, 9, 3, 10, 10, 12, 1, 1, 9, 4, 4, 10, 3, 3,
   1, 12, 12, 4, \dots
\end{multline*}
The period length is visibly 18 ($= (p-1)^2/8$); see Conjecture~\ref{conj:4}(1) in
Section~\ref{sec:Z}. Despite correspondence with the authors of~\cite{GuMRAA},
these concerns were not dispelled.}
In~\cite{WakhAA}, Wakhare revisited Item~(2) for primes (i.e., for
$e=1$). He showed that $d(n)$ is (eventually) periodic modulo any
prime number~$p$ with $p\equiv1$~(mod~4) by
proving the refinement
(see \cite[Theorem~2]{WakhAA}\footnote{\label{foot:1}
Wakhare did not simplify $-2^{(p-1)/2}$ modulo~$p$ on the
right-hand side of the displayed congruence in \cite[Theorem~1]{WakhAA}
to~$(-1)^{(p-5)/4}$.})
\begin{multline} \label{eq:Wakh} 
d\left(n+\tfrac {p-1} {2}\right)\equiv
(-1)^{(p-5)/4}
(3\cdot7\cdot11\cdots(2p-3))^2d(n)\pmod p,\\
\text{for
}p\equiv1~(\text{mod }4)\text{ and }n\ge\tfrac {p+1} {2}.
\end{multline}
Since on the right-hand side of this congruence we have a square,
$\frac {p-1} {2}$-fold iteration of this congruence
in combination with Fermat's Little Theorem shows that $d(n)$
is (eventually) periodic modulo~$p$ with (not necessarily minimal)
period length $\frac {(p-1)^2} {4}$.\footnote{Wakhare 
  concludes a period length of only $\frac {(p-1)^2} {2}$. The reason is
  probably the missed simplification pointed out in Footnote~\ref{foot:1}.}

What all these authors did not notice is that periodicity was already
(implicitly) known. Namely,
Rodr\'\i guez Villegas and Zagier showed in~\cite[\S\S~6,~7]{RVZaAA}
that Taylor coefficients of entire modular forms at complex multiplication
points\footnote{$\tau=i$ is such a ``complex multiplication point" for~$\theta_3$}
--- suitably normalised ---
can be computed via a certain recursive (polynomial) scheme, and they indicated
that this scheme can also be adapted to cover half-integral
weights.\footnote{We have worked out such a recursive scheme for producing
the Taylor coefficients~$d(n)$: let the polynomials
$p_n(t)$, $n\ge0$, be given by
$$
p_{n+1}(t)=(\tfrac 16 - 96 t^2) p'_n(t) + 16 (4 n +1) t p_n(t)
   - n (n - \tfrac 12) (256 t^2 + \tfrac 43) p_{n - 1}(t),
$$
with $p_{-1}(t)=0$ and $p_0(t)=1$. Then $p_{2n+1}(0)=0$ and
$$
d(n)=2^{-n}p_{2n}(0)
$$
for all $n\ge0$.}
Moreover, O'Sullivan and Risager proved in \cite[Theorem~6.1]{OSRiAA}
(in a special case, but the argument is completely general) that integral number
sequences that are produced by such a scheme are automatically periodic modulo
any prime power. On the other hand, this argument only leads to very crude,
super-exponential bounds on the period length, it cannot tell from when on
periodicity occurs, and it also cannot predict
whether such a sequence vanishes eventually modulo a prime power (thus
producing a trivial period).

\medskip
The purpose of the present article is to provide full proofs of all
the above items that do include explicit statements about 
period lengths, respectively from when on
periodicity or vanishing modulo a prime power holds. 
The theorem below collects our corresponding findings
(see Theorems~\ref{thm:10}, \ref{thm:13}, and \ref{thm:11}).

\begin{theorem} \label{thm:main}
{\em(1)} Let $p$ be a prime number with $p\equiv3$~{\em(mod~4),} and
let $e$ be an integer with $e\ge2$.
Then $d(n)\equiv0$~{\em(mod~$p^e$)} for $n\ge \cl{\frac {(e-1)p^2} {2}}$.

{\em(2)} Let $p$ be a prime number with $p\equiv1$~{\em(mod~4),} and
let $e$ be a positive integer.
Then the sequence $\big(d(n)\big)_{n\ge e+1}$ is purely periodic
modulo~$p^e$ with {\em(}not necessarily minimal\/{\em)} period
length~$\frac {1} {4}p^{e-1}(p-1)^2$.

{\em(3)} Let $e$ be a positive integer.
The sequence $(d(n))_{n\ge0},$ when taken modulo any fixed $2$-power $2^e$ with $e\ge3,$
is purely periodic with {\em(}not necessarily minimal\/{\em)} 
period length~$2^{e-1}$. Modulo~$4,$ the sequence is 
purely periodic with period length~$4,$ the first few
values of the sequence {\em(}modulo~$4${\em)} being given by 
$$
1,1, 3,  3  , 1 ,\dots.
$$
%
\end{theorem}

Our proof of Item~(2) in fact generalises the
congruence~\eqref{eq:Wakh} to prime powers; see~\eqref{eq:d-u-p}.
We remark that Larson and Smith~\cite{LaSmAA} prove a result
analogous to Theorem~\ref{thm:main}(1) for Taylor expansions
of modular forms of integral weight at complex multiplication
points in some imaginary quadratic number field, for primes
$p\ge5$.

\medskip
The proofs of the congruence properties of~$d(n)$ listed in
Theorem~\ref{thm:main} that we provide here are elementary throughout.
They do however reveal that these congruences hold at a much 
finer, polynomial level; see Remarks~\ref{rem:1}, \ref{rem:2},
\ref{rem:3}, \ref{rem:4}, Theorems~\ref{prop:2}--\ref{thm:4}, and~\ref{thm:12}.

More specifically, our proofs require a careful $p$-adic analysis of the earlier
mentioned recursive procedure of computation of the
numbers~$d(n)$. This procedure involves two further sequences, namely
$\big(u(n)\big)_{n\ge0}$ and $\big(v(n)\big)_{n\ge0}$, and a
lower triangular matrix $\big(r(n,k)\big)_{n,k\ge0}$. We refer the
reader to Section~\ref{sec:uvr} for the corresponding definitions.
In order to accomplish the proofs of the congruence assertions in
Theorem~\ref{thm:main}, we must first analyse $u(n)$, $v(n)$, and
$r(n,k)$ modulo the three families of prime powers that feature in the
theorem, before we can make conclusions about~$d(n)$.

Accordingly, our article is organised as follows.
In the next section, we review Romik's recursive procedure to compute
the Taylor coefficients~$d(n)$. This section contains in fact a
notable novelty that is crucial for our proofs. 
Briefly, Romik defines a sequence $\big(v(n)\big)_{n\ge0}$ and shows
that it equals the product of the matrix $\big(r(n,k)\big)_{n,k\ge0}$ ---
which in its definition involves the sequence
$\big(u(n)\big)_{n\ge0}$ ---
and the sequence $\big(d(n)\big)_{n\ge0}$ (seen as column vector).
At this point, one wants to invert this relation to have direct
access to the numbers~$d(n)$. This requires the computation of
the {\it inverse} of the matrix $\big(r(n,k)\big)_{n,k\ge0}$. In all previous
papers, this point is somehow by-passed. However, as it turns out,
Lagrange inversion permits one to give a compact formula for the entries
of this inverse matrix, again in terms of the sequence
$\big(u(n)\big)_{n\ge0}$; see~\eqref{eq:R^{-1}}. More specifically,
we extend the matrix $\big(r(n,k)\big)_{n,k\ge0}$ to the larger matrix
$\mathbf R=\big(R(n,k)\big)_{n,k\ge0}$ in which the former matrix is
embedded as the submatrix indexed by even  labelled rows and columns;
that is, $r(n,k)=R(2n,2k)$ for all $n$ and~$k$. Our inversion result then applies
to this larger matrix which, in view of the Lagrange inversion
formula, is the one that should be looked at.

In Section~\ref{sec:aux}, we recall the standard results from
elementary number theory concerning the divisibility of factorials and
binomial coefficients by prime powers that we use ubiquitously in our
article.

Subsequently, we embark on the $p$-adic analysis of our
sequences. Section~\ref{sec:2} is devoted to the proof that $u(n)$
vanishes modulo any fixed odd prime prime power~$p^e$ for large  enough~$n$.
It contains two main results:
Theorem~\ref{thm:1} treats the case where $p\equiv3$~(mod~4), while
Theorem~\ref{thm:1A} contains the (significant) improvement for the
case where $p\equiv1$~(mod~4). The corresponding results for $v(n)$
modulo odd prime powers are presented in Section~\ref{sec:3}.
Again there are two main results:
Theorem~\ref{thm:2} treats the ``generic case", while
Theorem~\ref{thm:2A} contains the (significant) improvement for the
case where $p\equiv1$~(mod~4). 
The final preparations for our first main result on $d(n)$ is done in
Section~\ref{sec:4}, where we prove a $p$-divisibility result for the matrix
entries of the {\it inverse} of the matrix~$\mathbf R$
for primes~$p$ with $p\equiv3$~(mod~4);
see Theorem~\ref{thm:9}.
Theorem~\ref{thm:10} in Section~\ref{sec:5} is then our first main
result for the sequence $\big(d(n)\big)_{n\ge0}$.
It says that $d(n)$ is divisible by
prime powers~$p^e$ with $p\equiv3$~(mod~4) for $n\ge \cl{{(e-1)p^2}/
 {2}}$ and $e\ge2$, thus establishing Part~(1) of Theorem~\ref{thm:main}.

The next sections are devoted to the analysis of our sequences and our
matrix~$\mathbf R$ modulo powers of~2. The subject of
Section~\ref{sec:6} is to show that the sequence $(v(n))_{n\ge0}$ is
purely periodic modulo any fixed power of $2$, together with a precise
statement concerning the period length. As a matter of fact,
Theorems~\ref{prop:2} and~\ref{prop:2a} contain a polynomial
refinement of these assertions.
In Section~\ref{sec:7} we establish a periodicity result for the entries
of the {\it inverse} of the matrix $\mathbf R$ modulo powers of $2$; see
Theorem~\ref{thm:4}.
These findings are then put together to obtain our second main result
for the sequence $\big(d(n)\big)_{n\ge0}$. Namely, $\big(d(n)\big)_{n\ge0}$ is purely
periodic modulo~$2^e$ with (not necessarily minimal) period
length~$2^{e-1}$ for $e\ge3$, while modulo~4 it starts with $1,1,3,3$ and
then repeats itself; see Theorem~\ref{thm:11} in Section~\ref{sec:8}.
These results establish Part~(3) of Theorem~\ref{thm:main}.

The purpose of Sections~\ref{sec:9} and \ref{sec:10} is to prove
periodicity of $\big(d(n)\big)_{n\ge0}$ modulo prime powers $p^e$ with
$p\equiv1$ (mod $4$), accompanied by a precise statement on the
period length. The preparatory work in the former section
concerns the analysis of the entries of the {\it inverse} 
of the matrix $\mathbf R$ modulo these prime powers; see Theorem~\ref{thm:12}.
Again, this theorem actually contains a polynomial refinement.
This is then used to prove our third main result on $d(n)$, namely
that $\big(d(n)\big)_{n\ge e+1}$ is purely periodic modulo $p^e$ with
$p\equiv1$~(mod~4) with (not necessarily minimal) period length $\frac
{1} {4}p^{e-1}(p-1)^2$, thus establishing Part~(2) of Theorem~\ref{thm:main};
see Theorem~\ref{thm:13} in Section~\ref{sec:10}.

The final section, Section~\ref{sec:Z}, addresses the question
of whether our results can be improved concerning period length,
respectively point of vanishing. As is reported in that
section, (computer) data suggest that our results are not far from
being optimal, but that it may be possible to improve them by ---
roughly --- a factor of~2. (Our result on $d(n)$ modulo powers of~2 in
Theorem~\ref{thm:main}(3) seems to be optimal, though.)
A proof of such strengthenings would however
require considerable effort involving highly technical considerations.

\section{Romik's recursive procedure for the computation of $d(n)$}
\label{sec:uvr}

In order to prove integrality of the coefficients $d(n)$
in~\eqref{eq:Taylor}, Romik sets up a recursive scheme to compute
these numbers from which it is immediate that it produces integers.
This scheme involves two further sequences, $\big(u(n)\big)_{n\ge0}$
and $\big(v(n)\big)_{n\ge0}$, and a lower triangular matrix
$\big(r(n,k)\big)_{n,k\ge0}$, which we are going to define next.

Let $(u(n))_{n\ge0}$ be the sequence defined by an exponential
generating function $U(t)$ via (cf.\ \cite[Eq.~(9) and Lemma~5]{RomikAA}\footnote{%
The reader should be warned that our definitions of the series $U(t)$ and the later
defined series~$V(t)$ slightly deviate from Romik's definitions; one gets Romik's series
from ours by dividing our $U(t)$ by~$t$, and by then replacing~$t$ by~$t^{1/2}$.
Our convention is crucial for the application of Lagrange inversion in
Proposition~\ref{prop:1}.
})
\begin{equation} \label{eq:u-def} 
  U(t):=
  \sum_{n\ge0}\frac {u(n)} {(2n+1)!}t^{2n+1}=
t\frac {
  {} _{2} F _{1} \!\left [ \begin{matrix} 
  \frac {3} {4},\frac {3} {4}\\\frac {3} {2}\end{matrix} ; 
  {\displaystyle 4t^2}\right ]  }
{
  {} _{2} F _{1} \!\left [ \begin{matrix} 
  \frac {1} {4},\frac {1} {4}\\\frac {1} {2}\end{matrix} ; 
  {\displaystyle 4t^2}\right ]  }.
\end{equation}
Here, $_2F_1[\dots]$ is the usual Gau{\ss} hypergeometric series. The
reader is referred to any standard book on special functions for the
definition, such as
e.g.\ \cite[Eq.~(2.1.2)]{AAR}.
It is not difficult to see (cf.\ \cite[Eq.~(17)]{RomikAA})
that an equivalent definition of
$u(n)$ is by the recurrence
\begin{equation} \label{eq:Rek1} 
u(n)=
\prod _{j=1} ^{n}(4j-1)^2-\sum_{m=0}^{n-1}\binom {2n+1}{2m+1}
\Bigg(\prod _{j=1} ^{n-m}(4j-3)^2\Bigg) u(m),\quad \text{with }u(0)=1.
\end{equation}
The first few values turn out to be
\begin{multline*}
1, 6, 256, 28560, 6071040, 2098483200, 1071889920000, 
758870167910400, \\
711206089850880000, 852336059876720640000, 
1271438437097485762560000, \dots
\end{multline*}

For convenience, we shall later frequently use the short notations
\begin{equation} \label{eq:Pi13} 
\Pi_1(N):=\prod _{j=1} ^{N}(4j-1)^2
\quad \text{and}\quad 
\Pi_3(N):=\prod _{j=1} ^{N}(4j-3)^2.
\end{equation}
Then, using these, the above recurrence can be rewritten as
\begin{equation} \label{eq:Rek1-Pi} 
u(n)=\Pi_1(n)-\sum_{m=0}^{n-1}\binom {2n+1}{2m+1}
\Pi_3(n-m)u(m),\quad \text{with }u(0)=1.
\end{equation}

The sequence $(v(n))_{n\ge0}$ is also defined via an exponential
generating function, namely by (cf.\ \cite[Eq.~(10) and Lemma~6]{RomikAA})
\begin{equation} \label{eq:v-def} 
\sum_{n\ge0}\frac {v(n)} {2^n(2n)!}t^{2n}
=   {} _{2} F _{1} \!\left [ \begin{matrix} 
  \frac {1} {4},\frac {1} {4}\\\frac {1} {2}\end{matrix} ; 
  {\displaystyle 4t^2}\right ] ^{1/2}
=\left(\sum_{j\ge0}\frac {
\prod _{\ell=1} ^{j}(4\ell-3)^2} {(2j)!}t^{2j}\right)^{1/2}.
\end{equation}
Again, it is not difficult to see (cf.\ \cite[Eq.~(20)]{RomikAA})
that an equivalent definition of
$v(n)$ is by the recurrence
\begin{equation} \label{eq:Rek2} 
v(n)=2^{n-1}\Pi_3(n)-\frac {1} {2}\sum_{m=1}^{n-1}\binom {2n}{2m}
v(m)v(n-m),\quad \text{with }v(0)=1.
\end{equation}
The first few values turn out to be
\begin{multline*}
1, 1, 47, 7395, 2453425, 1399055625, 1221037941375, 1513229875486875, \\
2526879997358510625, 5469272714829657020625, 
14892997153152592003359375, \dots
\end{multline*}

The lower triangular matrix $\big(r(n,k)\big)_{n,k\ge0}$ is defined via the generating
function $U(t)$ for the $u(n)$'s in~\eqref{eq:u-def}, namely by
\begin{equation} \label{eq:r2} 
r(n,k):=2^{n-k}\frac {(2n)!} {(2k)!}\coef{t^{2n}}U^{2k}(t).
\end{equation}
The matrix is indeed lower triangular since $r(n,k)=0$ for $n<k$, due
to the fact that $U(t)$ has zero constant term.
For  the convenience of the reader, we display
the table of numbers $r(n,k)$ with $0\le n,k\le 6$:
$$\begin{matrix}
1& 0& 0& 0& 0& 0& 0\\
0& 1& 0& 0& 0& 0& 0\\
0& 48& 1& 0& 0& 0& 0\\
0& 7584& 240& 1& 0& 0& 0\\
0& 2515968& 97664& 672& 1& 0& 0\\
0& 1432498176& 63221760& 560448& 1440& 1& 0\\
0& 1247557386240& 60299053056& 628024320& 2141568& 2640& 1
\end{matrix}
$$

Romik showed that the Taylor coefficients~$d(n)$ in~\eqref{eq:Taylor} are related to 
the numbers~$v(n)$ via the triangular system of equations
\begin{equation} \label{eq:d-def2} 
v(n)=\sum_{k=0}^{n}r(n,k)d(k), \quad 
\text{for all }n\ge0.
\end{equation}
This may be ``turned around" to obtain a recursion for the $d(n)$'s
(cf.\ \cite[Theorem~7]{RomikAA}),
\begin{equation} \label{eq:d-def} 
d(n)=v(n)-\sum_{k=0}^{n-1}r(n,k)d(k), \quad 
\text{and }d(0)=1.
\end{equation}

Without any doubt, Equation~\eqref{eq:d-def} does provide a
recursive way to compute the coefficients~$d(n)$. Indeed,
Scherer~\cite{ScheAA} and Wakhare~\cite{WakhAA} used it for the proof
of their results. However, in our opinion the suitability of~\eqref{eq:d-def}
for the proof of congruence relations satisfied by the
$d(n)$'s using inductive arguments is limited. It is more conceptual
to invert the relation~\eqref{eq:d-def2} and express the $d(n)$'s
entirely in terms of the $v(n)$'s and the {\it inverse} of the matrix 
$\big(r(n,k)\big)_{n,k\ge0}$.

In order to carry out this programme, we need an explicit formula
for the entries of this inverse matrix. It turns out that this can best be achieved if one
defines the (larger) infinite matrix $\mathbf R=(R(n,k))_{n,k\ge 0}$ by
\begin{equation} \label{eq:R} 
R(n,k)=2^{(n-k)/2}\frac {n!} {k!}\coef{t^{n}}U^k(t).
\end{equation}
Since $U(t)$ is a power series in which even powers of $t$ do not
occur, the matrix $\mathbf R$ has a ``checkerboard pattern"; more
precisely, $R(n,k)\ne0$ if, and only if, $n$ and $k$ have the same parity.
The entries $R(n,k)$ are in fact all integers. This follows from this
checkerboard pattern of~$\mathbf  R$  (implying that $2^{(n-k)/2}$ on
 the right-hand side of~\eqref{eq:R} is always an integer) and from the
exponential formula of combinatorics as Romik explains in \cite[Proof
  of Theorem~7]{RomikAA} in a special case (namely in the case where
both~$n$ and~$k$ are even; the general case can be treated in the
same manner).
The matrix $\big(r(n,k)\big)_{n,k\ge0}$ is a submatrix of~$\mathbf R$
since $R(2n,2k)=r(n,k)$ for all $n$ and~$k$ (compare with~\eqref{eq:r2}).

\begin{proposition} \label{prop:1}
The inverse $\mathbf R^{-1}$ of $\mathbf R$ is given by
$(R^{-1}(n,k))_{n,k\ge0}$ with
\begin{equation} \label{eq:R^{-1}} 
R^{-1}(n,k)=2^{(n-k)/2}\frac {(n-1)!} {(k-1)!}\coef{t^{-k}}U^{-n}(t).
\end{equation}
Here, the case $k=0$ has to be interpreted as $R^{-1}(0,0)=1$ and
$R^{-1}(n,0)=0$ for $n\ge1$.
Moreover, $\mathbf R^{-1}$ has integer entries and
$R^{-1}(n,k)\ne0$ only if $n$ and $k$ have the same parity.
\end{proposition}

\begin{proof}
That $\mathbf R^{-1}$ has integer entries is obvious since it is the
inverse of a triangular matrix with integer entries
(cf.\ \cite[Theorem~7]{RomikAA}) and $1$'s on the main diagonal.

The first assertion is a consequence of Lagrange inversion (cf.\ 
\cite[Theorem~5.4.2]{StanBI}): if $F(t)$ is a formal power series with
$F(0)=0$ and $F'(0)\ne0$, and $F^{(-1)}(t)$ is its compositional
inverse, then
$$
\coef{t^n}\big(F^{(-1)}(t)\big)^k=\frac {k} {n}\coef{t^{-k}}F^{-n}(t).
$$
In order to apply this theorem to our situation, it must be observed
that under the above conditions the matrix
$\big(\coef{t^n}F^k(t)\big)_{n,k\ge0}$ is inverse to
$\big(\coef{t^n}\big(F^{(-1)}(t)\big)^k\big)_{n,k\ge0}$.
If one applies this observation with $F(t)=U(t)$, then the assertion
of the proposition follows upon little manipulation.

With the formula \eqref{eq:R^{-1}} established,
the parity condition follows again from the fact that $U(t)$ is a
power series in which only odd powers of~$t$ occur.
\end{proof}

Using the entries of the matrix $\mathbf R$, the system of
equations~\eqref{eq:d-def2} can be rewritten as
$$
v(n)=\sum_{k=0}^{n}R(2n,2k)d(k), \quad 
\text{for all }n\ge0.
$$
If we invert this relation using the inverse matrix~$\mathbf R^{-1}$,
then we obtain
\begin{equation} \label{eq:d-v} 
d(n)=\sum_{k=0}^{n}R^{-1}(2n,2k)v(k).
\end{equation}
We are going to use this formula in the proofs in Sections~\ref{sec:8}
and~\ref{sec:10}.

\section{Classical criteria for divisibility of factorials and binomial
  coefficients by prime powers}
\label{sec:aux}

Here and in the sequel, for a prime number~$p$,
let $v_p(\al)$ denote the $p$-adic valuation of
the integer (or rational number)~$\al$, defined as 
the maximal exponent $e$ such that $\al/p^e$
is an integer (respectively a rational number with numerator and denominator
coprime to~$p$).

There are essentially two formulae for the $p$-adic valuation of a
factorial. The one that we need in this article is Legendre's formula.

\begin{lemma}[{\sc Legendre's formula \cite[p.~12]{LegeAA}}]
\label{lem:4}
Let $N$ be a positive integer and $p$ a prime number.
Then
$$
v_p(N!)=\frac {N-s_p(N)} {p-1},
$$
where $s_p(N)$ denotes the sum of digits in the $p$-adic
representation of~$N$.
\end{lemma}

Similarly, there are essentially two formulae for the $p$-adic
valuation of a binomial coefficient, one in terms of the number of
carries when adding the involved numbers, the other in terms of digit sums; 
we shall need the former one. 

\begin{lemma}[{\sc Kummer's theorem \cite[pp.~115--116]{KummAA}}]
\label{lem:5}
Let $N$ and $K$ be positive integers and $p$ a prime number.
Then
$
v_p\left(\binom NK\right)
$
equals the number of carries when $K$ and $N-K$ are added in terms of
their $p$-adic representations.
\end{lemma}

The following relation forms the link between the previous two lemmas,
and it provides, at the same time, the previously mentioned alternative
in computing the $p$-adic valuation of a binomial coefficient.

\begin{lemma}
\label{lem:4-5}
Let $A$ and $B$ be positive integers and $p$ a prime number.
Then
\begin{multline*}
\frac {1} {p-1}\Big(s_p(A)+s_p(B)-s_p(A+B)\Big)\\
=\#(\text{\em carries when adding $A$ and $B$ in their $p$-adic representations}).
\end{multline*}
\end{lemma}

\section{The sequence $(u(n))_{n\ge0}$ modulo odd prime powers}
\label{sec:2}

In this section, we analyse the numbers $u(n)$ modulo odd prime
powers~$p^e$. The first result is Theorem~\ref{thm:1} which states that,
for primes $p$ congruent to~3 modulo~4 and all integers $e\ge2$,
the number $u(n)$ vanishes modulo~$p^e$ for $n\ge\fl{(e-1)p^2/2}$.
The proof of the theorem is inductive, the
start of the induction being given in Proposition~\ref{prop:u-p^2},
which is itself based on auxiliary results in Lemma~\ref{lem:u-klein}.
In case the prime $p$ should be congruent to~$1$
modulo~4, however, a much stronger result holds; see
Theorem~\ref{thm:1A}. The proof of that theorem requires two
auxiliary results which we state and prove separately; cf.\ 
Lemmas~\ref{lem:u1} and~\ref{lem:u2}.

\medskip
We begin by providing lower bounds on
the  $p$-adic valuations of the products $\Pi_1(N)$ and $\Pi_3(N)$
defined in \eqref{eq:Pi13} that will be ubiquitously used in this and
the next section.

\begin{lemma} \label{lem:Pi13}
For all odd primes $p$ and non-negative integers $N,$ we have
\begin{equation} \label{eq:vp-prod} 
v_p\big(\Pi_1(N)\big)\ge 2\fl{\tfrac {N} {p}}
\quad \text{and}\quad 
v_p\big(\Pi_3(N)\big)\ge 2\fl{\tfrac {N} {p}}.
\end{equation}
If $p\equiv1$~{\em(mod $4$),}  then we even have 
\begin{equation} \label{eq:4j-3+} 
v_p\big(\Pi_3(N)\big)\ge 2\fl{\tfrac {N+\frac {3} {4}(p-1)} {p}}.
\end{equation}
\end{lemma}

Next we state and prove the announced auxiliary results,
which afterwards lead to Proposition~\ref{prop:u-p^2}.

\begin{lemma} \label{lem:u-klein}
Let $p$ be an odd prime number,
and let $(u(n))_{n\ge0}$ be defined by the recurrence \eqref{eq:Rek1}. 
Then we have
\begin{equation} \label{eq:u-klein1} 
u(ap+b)\equiv0~(\text{\em mod }p),\quad  \text{for }1\le a\le
\tfrac {p-1} {2}\text{ and\/ }0\le b\le a-1,
\end{equation}
and
\begin{equation} \label{eq:u-klein2} 
u\left(ap+\tfrac {p-1} {2}+b\right)\equiv0~(\text{\em mod }p),\quad 
\text{for }1\le a\le
\tfrac {p-1} {2}\text{ and\/ }0\le b\le a.
\end{equation}
Moreover, if $p\equiv3$~{\em(mod~$4$)},
the second congruence also holds for $a=0$, that is,\break
$u\left(\frac {p-1} {2}\right)\equiv0$~{\em(mod~$p$)} for primes~$p$
with $p\equiv3$~{\em(mod~$4$)}.
\end{lemma}
\begin{proof}
We prove the assertions by induction on the size of $ap+b$
and of $ap+\frac {p-1} {2}+b$, simultaneously. The induction
is based on the recurrence~\eqref{eq:Rek1}.

\medskip
Now let first $n=ap+b$ with $a$ and $b$ satisfying the conditions
in~\eqref{eq:u-klein1}. We want to prove that $u(n)\equiv0$~(mod~$p$).

It is clear that the first term on the right-hand side of~\eqref{eq:Rek1},
the product $\Pi_1(n)=\Pi_1(ap+b)$, is
divisible by~$p$ because of~\eqref{eq:vp-prod}.

From Lemma~\ref{lem:5} we infer that
the binomial coefficient $\binom {2n+1}{2m+1}
=\binom {2ap+2b+1}{2m+1}$ is divisible by~$p$, except when
$2m=cp+d$ with both $0\le c\le 2a$ and $0\le d\le 2b$.
It is important to note that at this point the upper bounds of
$\frac {p-1} {2}$ for~$a$ and~$b$ enter crucially.

If $c=0$ (and hence $d$ even since $2m=cp+d=d$), then
$\Pi_3(n-m)=\Pi_3\left(ap+b-\frac {d} {2}\right)$ is divisible by~$p$
according to~\eqref{eq:vp-prod}.
From now on we may assume that $c\ge1$.

We distinguish whether $c$ and~$d$ are both even or both odd.
(There are no other cases since $2m=cp+d$ is even.)

If both $c$ and~$d$ are even, then, provided $c>d$, we may use the
induction hypothesis~\eqref{eq:u-klein1} to infer that
$u(m)=u\left(\frac {c} {2}p+\frac 
{d} {2}\right)$ is divisible by~$p$. On the other hand, if $c\le d$,
then we have
$$
n-m=\left(a-\tfrac {c} {2}\right)p+b-\tfrac {d} {2}.
$$
Since, by assumption, we have $a>b$, we have
$$
a-\tfrac {c} {2}>b-\tfrac {d} {2}\ge0.
$$
Consequently, again using~\eqref{eq:vp-prod}, we
obtain that $\Pi_3(n-m)$ is divisible by~$p$.

Now let both $c$ and~$d$ be odd. Here we may rewrite
$m=\frac {c} {2}p+\frac {d} {2}$ as
$$m=\tfrac {c-1} {2}p+\tfrac {p-1} {2}+\tfrac {d+1} {2}.$$
Furthermore, we
may write
$$
n-m=\left(a-\tfrac {c+1} {2}\right)p
+\tfrac {p-1} {2}+b-\tfrac {d-1} {2}.
$$
It should be noted that, if $c\le d$, we have $a-\tfrac {c+1} {2}>
b-\tfrac {d+1} {2}\ge0.$
As a consequence,
by~\eqref{eq:vp-prod},
we have $\Pi_3(n-m)\equiv0$~(mod~$p$).
If, on the other hand, we have $c> d$, then
$u(m)\equiv0$~(mod~$p$), again
by the induction hypothesis~\eqref{eq:u-klein2}.

\medskip
Next we discuss the case where
$n=ap+\frac {p-1} {2}+b$ with $a$ and $b$ satisfying the conditions
in~\eqref{eq:u-klein2}. We want to prove that $u(n)\equiv0$~(mod~$p$).

Again it is clear that the first term on the right-hand side
of~\eqref{eq:Rek1},
the product $\Pi_1(n)=\Pi_1(ap+\frac {p-1} {2}+b)$, is
divisible by~$p$ because of~\eqref{eq:vp-prod}.

In the current case, the binomial coefficient on the right-hand side
of~\eqref{eq:Rek1} becomes $\binom {2n+1}{2m+1}
=\binom {(2a+1)p+2b}{2m+1}$. Here it must be observed that,
because of the upper bounds on~$a$ and~$b$ in~\eqref{eq:u-klein2}, we
have $2a+1\le p$ and $2b\le p-1$.
If $a=\frac {p-1} {2}$, so that $2a+1=p$, then by Lemma~\ref{lem:5}
$\binom {2n+1}{2m+1}=\binom {p^2+2b}{2m+1}$ is divisible by~$p$, except
when $1\le 2m+1\le 2b$. In this exceptional case, we have
$\Pi_3(n-m)\equiv0$~(mod~$p$) by~\eqref{eq:vp-prod}.

We assume from now on that $a<\frac {p-1} {2}$.
From Lemma~\ref{lem:5} we infer that
the binomial coefficient $\binom {2n+1}{2m+1}
=\binom {(2a+1)p+2b}{2m+1}$ is divisible by~$p$, except when
$2m=cp+d$ with both $0\le c\le 2a+1$ and $0\le d\le 2b-1$.

We distinguish whether $c$ and~$d$ are both even or both odd.
(There are no other cases since $2m=cp+d$ is even.)

If both $c$ and~$d$ are even, then, provided $c>d$, we may use the
induction hypothesis~\eqref{eq:u-klein1} to infer that
$u(m)=u\left(\frac {c} {2}p+\frac 
{d} {2}\right)$ is divisible by~$p$. On the other hand, if $c\le d$,
then we have
$$
n-m=\left(a-\tfrac {c} {2}\right)p+\tfrac {p-1} {2}+b-\tfrac {d} {2}.
$$
Since, by assumption, we have $a\ge b$, we have
$$
a-\tfrac {c} {2}\ge b-\tfrac {d} {2}\ge1.
$$
Consequently, again using~\eqref{eq:vp-prod}, we
obtain that $\Pi_3(n-m)$ is divisible by~$p$.

Now let both $c$ and~$d$ be odd. We may again rewrite
$m=\frac {c} {2}p+\frac {d} {2}$ as
$$m=\tfrac {c-1} {2}p+\tfrac {p-1} {2}+\tfrac {d+1} {2}.$$
Furthermore, we
may write
$$
n-m=\left(a-\tfrac {c-1} {2}\right)p+b-\tfrac {d+1} {2}.
$$
Hence, according to the induction hypothesis~\eqref{eq:u-klein2},
we have $u(m)\equiv0$~(mod~$p$) if $c> d$.
Otherwise we have
$$
a-\tfrac {c-1} {2}>b-\tfrac {d+1} {2}\ge0.
$$
Hence, $\Pi_3(n-m)\equiv0$~(mod~$p$) due to~\eqref{eq:vp-prod}.

\medskip
Finally, the congruence for $u\left(\frac {p-1} {2}\right)$
in the case where $p\equiv3$~(mod~$4$) holds because, as is seen by
inspection, the first term on the right-hand side of~\eqref{eq:Rek1},
the product $\Pi_1(n)=\Pi_1(\frac {p-1} {2})$, is divisible
by~$p^2$ (here, the condition $p\equiv3$~(mod~$4$) enters crucially).
\end{proof}

\begin{proposition} \label{prop:u-p^2}
Let $(u(n))_{n\ge0}$ be defined by the recurrence \eqref{eq:Rek1}. 
Then, given a prime $p\equiv3$~{\em(mod~$4$)} and a positive integer~$e,$ 
the number $u(n)$ is divisible by $p^2$ for $n\ge \fl{\frac {p^2} {2}}$.
\end{proposition}

\begin{proof}
Let $n\ge \fl{\frac {p^2} {2}}=\frac {p^2-1} {2}$.
As in the proof of Lemma~\ref{lem:u-klein}, we use again induction on~$n$.
Also here, the induction will be based on~\eqref{eq:Rek1}, and it
proceeds by showing that
each summand on the right-hand side is divisible by~$p^2$.

For the start of the induction, we consider $n=\frac {p^2-1} {2}$,
so that $2n+1=p^2$. In that case, there are two carries when
adding $2m+1$ and $2(n-m)$ for $0\le m\le n-1$, except when
both of $2m+1$ or $2(n-m)$ are divisible by~$p$. In the former case,
Lemma~\ref{lem:5} implies that
the binomial coefficient $\binom {2n+1}{2m+1}$ is divisible by~$p^2$.
In the latter case, there is only one carry and so the binomial
coefficient $\binom {2n+1}{2m+1}$ is only divisible by~$p$.
In its turn, we see that $u(m)$ is divisible by~$p$
due to~\eqref{eq:u-klein2} with $b=0$ and the last assertion in
Lemma~\ref{lem:u-klein}.
In either case,
each summand on the right-hand side
of~\eqref{eq:Rek1} is divisible by~$p^2$.
Furthermore, it is again clear that the first term on the right-hand side,
the product $\Pi_1(n)$, is
divisible by~$p$ for all $n\ge\fl{\frac {p^2} {2}}$
because of~\eqref{eq:vp-prod}.

We will assume $n\ge \frac {p^2+1} {2}$ from now on.

\medskip
We may restrict our attention to $n\le p^2$ because otherwise
either $m\ge \frac {p^2+1} {2}$ or
$n-m\ge \frac {p^2+1} {2}$; whence, either
$u(m)\equiv0$~(mod~$p^2$)
or $\Pi_3(n-m)\equiv0$~(mod~$p^2$). Consequently, in this case
each term on the
right-hand side of~\eqref{eq:Rek1} would be divisible by~$p^2$.

To summarise the discussion so far: we may write $2n$ as
$2n=p^2+ap+b$ for some $a$ and $b$ not of the same parity
and with $0\le a,b\le p-1$.

\medskip
If there are two carries when adding $2m+1$ and $2(n-m)$, then
by Lemma~\ref{lem:5}
the binomial coefficient $\binom {2n+1}{2m+1}$ is divisible by~$p^2$.
Hence the corresponding summand on the right-hand side
of~\eqref{eq:Rek1} is divisible by~$p^2$.

\medskip
Now let us assume that there is only one carry when adding $2m+1$ and
$2(n-m)$. As earlier, if either $m$ or~$n-m$ are larger
than~$\frac {p^2} {2}$, then $u(m)$ or $\Pi_3(n-m)$ are divisible
by~$p^2$, and thus as well the corresponding summand on the right-hand
side of~\eqref{eq:Rek1}. We may therefore assume without loss of
generality that $2m=cp+d$ with $c\le \frac {p-1} {2}$. We may
furthermore assume that $d\le b$ since, otherwise, there would be
two carries when adding $2m+1$ and $2(n-m)$.

We distinguish again whether $c$ and~$d$ are both even or both odd.

Let first $c$ and $d$ be even. We write $m=\frac {c} {2}p+\frac {d} {2}$.
If $c>d$, then by~\eqref{eq:u-klein1},
we infer $u(m)\equiv0$~(mod~$p$). If $c\le d$, then we may write
$$
n-m=\left(\tfrac {p+a-c} {2}\right)p+\tfrac {b-d} {2}
=\left(\tfrac {p+a-c-1} {2}\right)p+\tfrac {p-1} {2}+\tfrac {b-d+1} {2}
$$
and, depending on the parities of $a$ and~$b$, use the alternative
which has an integer coefficient in front of~$p$.
By our assumptions, we have
$$
\tfrac {p+a-c} {2}> \tfrac {b+a-c} {2}
\ge \tfrac {b-c} {2}
\ge \tfrac {b-d} {2}\ge0.
$$
If $a=0$ then $b$ must be odd, implying that the last inequality is
strict so that $\tfrac {p+a-c-1} {2}=\tfrac {p-c-1} {2}$ is also positive.
By \eqref{eq:vp-prod}, regardless whether $a=0$ or not, this implies that
$\Pi_3(n-m)\equiv0$~(mod~$p$).
In total, in both cases this shows that the
corresponding summand in~\eqref{eq:Rek1} is divisible by~$p^2$.

Now let $c$ and $d$ be odd.
Here we write $m=\frac {c-1} {2}p+\frac {p-1} {2}+\frac {d+1} {2}$.
If $c>d$, then by~\eqref{eq:u-klein2},
we infer $u(m)\equiv0$~(mod~$p$). If $c\le d$, then we may again write
$$
n-m=\left(\tfrac {p+a-c} {2}\right)p+\tfrac {b-d} {2}
=\left(\tfrac {p+a-c-1} {2}\right)p+\tfrac {b-d+1} {2}.
$$
Arguing as before, we conclude that
$\Pi_3(n-m)\equiv0$~(mod~$p$).
This shows again that the
corresponding summand in~\eqref{eq:Rek1} is divisible by~$p^2$.

\medskip
If there is no carry when adding $2m+1$ and $2(n-m)$, then
necessarily one of $2m+1$ or $2(n-m)$ is at least $p^2+1$.
So, again, one of $u(m)$ or $\Pi_3(n-m)$ is divisible by~$p^2$
due to the induction hypothesis, respectively due to~\eqref{eq:vp-prod}.

\medskip
This completes the proof of the proposition.
\end{proof}

The following theorem proves Conjecture~18(1) in \cite{WakhAA} for $u(n)$.

\begin{theorem} \label{thm:1}
Let $(u(n))_{n\ge0}$ be defined by the recurrence \eqref{eq:Rek1}. 
Then, given a prime $p\equiv3$~{\em(mod~$4$)} and an integer~$e\ge2,$ 
the number $u(n)$ is divisible by $p^e$ for $n\ge \fl{\frac {(e-1)p^2} {2}}$.
\end{theorem}

\begin{proof}
We proceed by a double induction on~$e$
and~$n$, the outer induction being on~$e$.
For the start of the induction, we use Proposition~\ref{prop:u-p^2}
which proves the assertion of the theorem for $e=2$.
From now on let $e\ge3$.

We assume that $\fl{\frac {(e-1)p^2} {2}}\le n<\fl{\frac {ep^2} {2}}$.
We claim that 
the first term on the right-hand side of~\eqref{eq:Rek1},
namely $\Pi_1(n)$, is always
divisible by~$p^e$ under this assumption. Indeed,
by \eqref{eq:vp-prod} we have
\begin{equation} \label{eq:vpPi1} 
v_p\left(\Pi_1(n)\right)
\ge
2\fl{\tfrac {n} {p}}
\ge
2\fl{\tfrac {1} {p}\fl{\tfrac {(e-1)p^2} {2}}}
\ge
2\fl{\tfrac {(e-1)p} {2}-\tfrac {1} {2p}}
\ge
2\left(\tfrac {3(e-1)} {2}-1\right)
\ge e
\end{equation} 
for $e\ge3$. The conclusion $v_p\left(\Pi_1(n)\right)\ge e$ holds
however as well for $e=2$ as shown by a modification of the final
part of the estimation:
\begin{equation} \label{eq:vpPi2} 
v_p\left(\Pi_1(n)\right)
\ge
2\fl{\tfrac {n} {p}}
\ge\dots\ge
2\fl{\tfrac {p} {2}-\tfrac {1} {2p}}
=2\left(\tfrac {p} {2}-\tfrac {1} {2}\right)
\ge 2.
\end{equation}

Next we consider the summand on the right-hand side
of~\eqref{eq:Rek1} for $\fl{\frac {(f-1)p^2} {2}}\le m<\fl{\frac {fp^2}
  {2}}$ with $2\le f\le e-2$.
By the induction hypothesis applied to $u(m)$, we have $v_p(u(m))\ge f$.
Furthermore we have
$$
n-m> \fl{\tfrac {(e-1)p^2} {2}}-\fl{\tfrac {fp^2} {2}}
\ge \fl{\tfrac {(e-f-1)p^2} {2}}.
$$
Therefore, by replacing $n$ by $n-m$ and $e$ by $e-f$
in~\eqref{eq:vpPi1} and~\eqref{eq:vpPi2}, we get
\begin{equation} \label{eq:Pi_1>} 
v_p\left(\Pi_3(n-m)\right)
\ge e-f
\end{equation}
since $e-f\ge 2$.
We infer that $v_p\big(u(m)\Pi_3(n-m)\big)\ge f+(e-f)=e$, and hence the
corresponding summand in \eqref{eq:Rek1} is divisible by~$p^e$.

Let now $0\le m< \fl{\tfrac {p^2} {2}}$.
In that case, the previous argument yields that
$n-m> \fl{\tfrac {(e-2)p^2} {2}}$, with the consequence that we only
have $v_p(\Pi_3(n-m))\ge e-1$.
It may be that actually $n-m\ge \fl{\tfrac {(e-1)p^2} {2}}$.
Then the estimation~\eqref{eq:Pi_1>} implies that
$v_p(\Pi_3(n-m))\ge e$, so that 
the corresponding summand in \eqref{eq:Rek1} is divisible by~$p^e$.
On the other hand, if
$\fl{\tfrac {(e-2)p^2} {2}}\le n-m< \fl{\tfrac {(e-1)p^2} {2}}$,
then we may write
$2n=(e-1)p^2+ap+b$ and $2(n-m)=(e-2)p^2+cp+d$ for some $a,b,c,d$
with $0\le a,b,c,d\le p-1$. Since $2m+1<p^2$,
there is (at least) one carry
when adding $2(n-m)$ and $2m+1$. Therefore, by Lemma~\ref{lem:5},
the binomial coefficient
$\binom {2n+1}{2m+1}$ is divisible by~$p$. 
Together with
the previously observed fact that $v_p(\Pi_3(n-m))\ge e-1$ this shows that the
corresponding summand in \eqref{eq:Rek1} is divisible by~$p^e$.

Finally, let
$\fl{\tfrac {(e-2)p^2} {2}}\le m< \fl{\tfrac {(e-1)p^2} {2}}$.
The induction hypothesis (in~$n$) implies that
in this case we have $v_p(u(m))\ge e-1$.
We write
$2n+1=(e-1)p^2+ap+b$ and $2m+1=(e-2)p^2+cp+d$ for some $a,b,c,d$
with $0\le a,b,c,d\le p-1$.
If $n-m\ge \fl{\frac {p^2} {2}}$, then by~\eqref{eq:vpPi2}
we obtain $v_p(\Pi_3(n-m))\ge1$.
If, on the other hand, we have $n-m< \fl{\frac {p^2} {2}}$, then there
is (at least) one carry when adding $2m+1$ and $2(n-m)$ in their
$p$-adic representations.
Therefore, by Lemma~\ref{lem:5}, the binomial coefficient
$\binom {2n+1}{2m+1}$ is divisible by~$p$. 
In either case, together with
the previously observed fact that $v_p(u(m))\ge e-1$ this shows that the
corresponding summand in \eqref{eq:Rek1} is divisible by~$p^e$.

\medskip
This concludes the induction step, and, thus, the proof of the theorem.
\end{proof}

\begin{remarknu} \label{rem:1}
An examination of the above arguments reveals that the products
$\Pi_1(n)$ and $\Pi_3(n)$ in the definition~\eqref{eq:Rek1-Pi}
could have been replaced with any functions $f(n)$ and $g(n)$ that
satisfy the $p$-divisibility properties in~\eqref{eq:vp-prod}.
\end{remarknu}

\begin{theorem} \label{thm:1A}
Let $(u(n))_{n\ge0}$ be defined by the recurrence \eqref{eq:Rek1}. 
Then, given a prime\linebreak $p\equiv1$~{\em (mod~$4$)} and a positive integer~$e,$ 
the number $u(n)$ is divisible by $p^e$ for $n\ge \cl{\frac {ep}2}$.
Moreover, the number $u\big(\frac {p-1} {2}\big)$ is a quadratic
residue modulo~$p$ and relatively prime to~$p$.
\end{theorem}

\begin{proof}
We begin with the second assertion. 
Setting $n=\frac {p-1} {2}$ in~\eqref{eq:Rek1}/\eqref{eq:Rek1-Pi}, we obtain
$$
u\left(\tfrac {p-1} {2}\right)=
\Pi_1\left(\tfrac {p-1} {2}\right)-\sum_{m=0}^{(p-3)/2}\binom {p}{2m+1}
\Pi_3\left(\tfrac {p-1} {2}\right) u(m).
$$
Due to the range of the sum on the right-hand side, the binomial
coefficient is always divisible by~$p$. Therefore, we have
\begin{equation} \label{eq:u-Pi1} 
u\left(\tfrac {p-1} {2}\right)\equiv\Pi_1\left(\tfrac {p-1} {2}\right)
\pmod p.
\end{equation}
The assertion now follows by observing that $\Pi_1\left(\tfrac {p-1}
{2}\right)$ is a square not divisible by~$p$.

\medskip
For the proof of the first assertion,
we proceed again by a double induction on~$e$
and~$n$, the outer induction being on~$e$.
The theorem is trivial for $e=0$, which serves as the start
of the induction.

We divide the proof into subtasks. At many places, the argument
depends on the parity of~$e$ and/or on the distance of $n$ from its
lower bound~$\cl{\frac {ep} {2}}$.

\medskip
{\sc Task 1: The first term on the left-hand side of \eqref{eq:Rek1-Pi}
vanishes modulo $p^e$ for even~$e$}.
By \eqref{eq:vp-prod}, we have
$$
v_p\big(\Pi_1(n)\big)\ge 2\fl{\tfrac {n} {p}}
\ge 2\fl{\tfrac {\cl{\frac{ep}2}}{p}}\ge 2\fl{\tfrac {e} {2}}=e,
$$
as desired, since $e$ is even.

\medskip
{\sc Task 2: The assertion holds for $n\ge\cl{\frac {(e+1)p} {2}}$.}
Under this assumption we have
$$
v_p\big(\Pi_1(n)\big)\ge 2\fl{\tfrac {n} {p}}
\ge 2\fl{\tfrac {\cl{\frac{(e+1)p}2}}{p}}\ge 2\fl{\tfrac {e+1} {2}}\ge e,
$$
which shows that the first term on the right-hand side of
\eqref{eq:Rek1-Pi} vanishes modulo~$p^e$.

Next we consider the summand on the right-hand side of
\eqref{eq:Rek1-Pi}. For $m\ge\cl{\frac {ep} {2}}$, the term $u(m)$ is
divisible by $p^e$ by the inductive hypothesis (with respect to~$n$),
hence the  corresponding summand vanishes modulo~$p^e$.
Now let $\cl{\frac {(f-1)p} {2}}\le m<\cl{\frac {fp}
  {2}}$ with $1\le f\le e$.
Our conditions imply that 
\[
n-m\ge \cl{\frac{(e+1)p}2}-\cl{\frac{fp}2}+1\ge
\cl{\frac{(e+1-f)p}2}
.
\]
Therefore, by using \eqref{eq:4j-3+} we obtain
$$
v_p\left(\big(\Pi_3(n-m)\big) u(m)\right)
\ge 2\fl{\tfrac {\cl{\frac{(e+1-f)p}2}+\frac {3} {4}(p-1)} {p}}+f-1\ge( e-f+1)+f-1=e,
$$
which shows that also in this case the corresponding summand
vanishes modulo~$p^e$.

\medskip
From now on we assume that $n<\cl{\frac {(e+1)p} {2}}$ or, more precisely,
integrating the overall assumption, that $\cl{\frac {ep} {2}}\le n<\cl{\frac {(e+1)p} {2}}$.

\medskip
{\sc Task 3: $p$-adic analysis of the summand in \eqref{eq:Rek1-Pi} for
$\cl{\frac {ep} {2}}\le n<\cl{\frac {(e+1)p} {2}}$.}
As before, by the induction hypothesis of the induction on
$n$, the term $u(m)$ in the sum in \eqref{eq:Rek1-Pi} is divisible
by~$p^e$ for $m\ge\cl{\frac {ep} {2}}$, and thus is the corresponding summand.

Next we consider the summand on the right-hand side of
\eqref{eq:Rek1-Pi} for $\cl{\frac{(f-1)p}2}\le m<\cl{\frac{fp}2}$ with
$1\le f\le e$.
Our conditions imply $n-m\ge \cl{\frac{ep}2}-\cl{\frac{fp}2}+1\ge
\cl{\frac{(e-f)p}2}$.
Therefore, by using \eqref{eq:4j-3+} we obtain
$$
v_p\left(\big(\Pi_3(n-m)\big) u(m)\right)
\ge 2\fl{\tfrac {\cl{\frac{(e-f)p}2}+\frac {3} {4}(p-1)} {p}}+f-1\ge( e-f)+f-1=e-1.
$$

Our goal is to show that the
summand is divisible by~$p^e$. 
We claim that this is indeed the case as long as $e$ is even, or $e$
is odd and  $m\ne \frac {fp-1} {2}$ for odd~$f$.
We must therefore prove
that the binomial coefficient in \eqref{eq:Rek1-Pi} is divisible by~$p$
in these cases.

To see this, we consider the $p$-adic representations
of $2m+1$ and $2n-2m$. Since $m\ne \frac {fp-1} {2}$,
(this condition is essential for the first line below), these are
\begin{alignat}2
\notag
(2m+1)_p&={}&(f-1)_p\ *\\
\big(2n-2m\big)_p&={}&(e-f)_p\ *
\label{eq:e-f3}
\end{alignat}
Here, $(\al)_p$ denotes the $p$-adic representation of the
integer~$\al$, and the stars on the right-hand sides indicate 
the right-most digits of $(2m+1)_p$ and $(2n-2m)_p$ whose precise
values are irrelevant. The sum of $2m+1$ and $2n-2m$ is $2n+1$
whose $p$-adic representation has the form $(e)_p\ *$. 
Hence, when adding the two numbers on the right-hand sides
of~\eqref{eq:e-f3}, at least one carry must occur --- namely one from the
$p^0$-digit to the $p^1$-digit. By Kummer's theorem in Lemma~\ref{lem:5},
the consequence is that the binomial
coefficient $\binom {2n+1}{2m+1}$ is divisible by~$p$. Therefore
the corresponding summand is divisible by~$p^e$.

\medskip
In order to summarise our findings so far:

\begin{enumerate} 
\item 
By Task~2 the
assertion of the theorem holds for $n\ge\cl{\frac{(e+1)p}2}$. 
\item For even~$e$ and $\cl{\frac {ep} {2}}\le n<\cl{\frac{(e+1)p}2}$ we have shown that
all summands of the sum on the right-hand side of \eqref{eq:Rek1-Pi}
are divisible by $p^e$. Since in Task~1 we have proved
that also the first term on the right-hand side of \eqref{eq:Rek1-Pi} is
divisible by~$p^e$, this establishes the assertion of the theorem for
even~$e$.
\item For odd~$e$ and $\cl{\frac {ep} {2}}\le
n<\cl{\frac{(e+1)p}2}$ 
we have shown that all summands of the
sum on the right-hand side of \eqref{eq:Rek1-Pi} vanish modulo~$p^e$
except for those where $m$ is of the form $\frac {p-1} {2}+ip$ for
some non-negative integer~$i$.
\end{enumerate}

Hence, it remains to discuss the case where $e$ is odd and 
$\cl{\frac {ep} {2}}\le
n<\cl{\frac{(e+1)p}2}$.

\medskip
{\sc Task 4: $e$ is odd and 
$\cl{\frac {ep} {2}}\le
n<\cl{\frac{(e+1)p}2}$.}
By Item~(3) above, 
the relation \eqref{eq:Rek1-Pi} reduces modulo~$p^e$ to
\begin{equation} \label{eq:u10} 
u(n)\equiv\Pi_1(n)
-\sum_{i=0}^{(e-1)/2}\binom {2n+1}{(2i+1)p}\Pi_3\big(n-\tfrac {p-1} {2}-ip\big)
u\big(\tfrac {p-1} {2}+ip\big)
\pmod{p^e}.
\end{equation}

By the induction hypothesis, we know that 
\begin{equation} \label{eq:u-val}
v_p\left(  u\left(\tfrac {p-1} {2}+ip\right)\right)\ge 2i.
\end{equation}
Furthermore, by \eqref{eq:vp-prod} we have
\begin{equation} \label{eq:Pi3-val}
v_p\left(\Pi_3\big(n-\tfrac {p-1} {2}-ip\big)\right)
\ge 2\fl{\tfrac {1} {p}\left(n-\tfrac {p-1} {2}-ip\right)}
\ge 2\fl{\tfrac {1} {p}\left(\tfrac {(e-2i-1)p} {2}\right)}
\ge e-2i-1,
\end{equation}
since $e$ is odd. In particular, the inequalities \eqref{eq:u-val}
and \eqref{eq:Pi3-val} together imply that
$$
v_p\left(\Pi_3\big(n-\tfrac {p-1} {2}-ip\big)
u\big(\tfrac {p-1} {2}+ip\big)\right)\ge e-1.
$$
Consequently, we may consider the binomial coefficient in
\eqref{eq:u10} modulo~$p$.
Now we use the elementary congruence
$$
\binom {ap+r}{bp}\equiv\binom {ap}{bp}\pmod{p}
$$
for positive integers $a,b,r$ with $1\le r<p$.
Since the assumptions of the case in which we are imply
that 
$ep+2\le 2n+1\le (e+1)p-1$, we infer that 
$$
\binom {2n+1}{(2i+1)p}\equiv\binom {ep}{(2i+1)p}\pmod{p}.
$$
Thus, we see that the congruence \eqref{eq:u10} is equivalent with
\begin{equation} \label{eq:u12} 
u(n)\equiv\Pi_1(n)
-\sum_{i=0}^{(e-1)/2}\binom {ep}{(2i+1)p}\Pi_3\big(n-\tfrac {p-1} {2}-ip\big)
u\big(\tfrac {p-1} {2}+ip\big)
\pmod{p^e}.
\end{equation}

On the other hand, the inequality \eqref{eq:Pi3-val} also implies that
we may consider the term $u\big(\tfrac {p-1} {2}+ip\big)$ in \eqref{eq:u12}
modulo~$p^{2i+1}$. This allows us to use Lemma~\ref{lem:u2} with $e$
replaced by $2i+1$. The corresponding substitution in \eqref{eq:u12} gives
\begin{multline} \label{eq:u11} 
u(n)\equiv\Pi_1(n)
-\sum_{i=0}^{(e-1)/2}\binom {ep}{(2i+1)p}\Pi_3\big(n-\tfrac {p-1}
{2}-ip\big)\\
\cdot
  \sum_{k\ge0}\sum_{i=\al_0>\al_1>\dots>\al_k\ge0}
  (-1)^k\Pi_1\left(\tfrac {p-1} {2}+\al_kp\right)\\
  \cdot
\prod _{j=0} ^{k-1}\binom {(2\al_j+1)p}
      {(2\al_{j+1}+1)p}\Pi_3\big((\al_j-\al_{j+1})p\big) 
\pmod{p^e}.
\end{multline}
We have
\begin{equation} \label{eq:Aufsp} 
\Pi_3\big(n-\tfrac {p-1}
{2}-ip\big)=\Pi_3\big(\tfrac {ep-2i-1} {2}\big)
\prod _{j=\frac {(e-2i-1)p} {2}+1} ^{n-\frac {p-1} {2}-ip}(4j-3)^2.
\end{equation}
Moreover, by \eqref{eq:vp-prod}, and remembering that $i=\al_0$, we get
\begin{multline*}
  v_p\left(\Pi_3\big(\tfrac {(e-2i-1)p} {2}\big)
  \Pi_1\left(\tfrac {p-1} {2}+\al_kp\right)
\prod _{j=0} ^{k-1}\Pi_3\big((\al_j-\al_{j+1})p\big) 
\right)\\
\ge (e-2i-1)+2\al_k+\sum _{j=0} ^{k-1}2(\al_j-\al_{j+1})=e-1.
\end{multline*}
This allows us to consider the product on the right-hand side of
\eqref{eq:Aufsp} modulo~$p$, once that equation is substituted in
\eqref{eq:u11}. The result then is
\begin{multline} \label{eq:u13} 
u(n)\equiv\Pi_1(n)
-\bigg(\prod _{j=1} ^{n-\frac {ep-1} {2}}(4j-3)^2\bigg)
\sum_{i=0}^{(e-1)/2}\binom {ep}{(2i+1)p}\Pi_3\big(\tfrac {(e-2i-1)p} {2}\big)\\
\cdot
  \sum_{k\ge0}\sum_{i=\al_0>\al_1>\dots>\al_k\ge0}
  (-1)^k\Pi_1\left(\tfrac {p-1} {2}+\al_kp\right)\\
  \cdot
\prod _{j=0} ^{k-1}\binom {(2\al_j+1)p}
      {(2\al_{j+1}+1)p}\Pi_3\big((\al_j-\al_{j+1})p\big) 
\pmod{p^e}.
\end{multline}
Here we see that, in the sum, for $k\ge1$ the term for
$$
i=\tfrac {e-1} {2}=\al_0>\al_1>\dots>\al_k\ge0
$$
cancels with the term for
$$
\tfrac {e-1} {2}>\al'_0>\al'_1>\dots>\al'_{k-1}\ge0,
$$
where $\al'_j=\al_{j+1}$ for $j=0,1,\dots,k-1$.
After this cancellation, the only remaining term in the sum is the one
for $k=0$ and $\al_0=\frac {e-1} {2}$.
In this regard, we have
\begin{align*}
\Bigg(\prod _{j=1} ^{n-\frac {ep-1} {2}}(4j-3)^2\Bigg)
\Pi_1\left(\tfrac {p-1} {2}+\al_0p\right)
&\equiv
\Bigg(\prod _{j=1} ^{n-\frac {ep-1} {2}}(4j+2p+4\al_0p-3)^2\Bigg)
\Pi_1\left(\tfrac {p-1} {2}+\al_0p\right)\\
&\equiv
\Pi_1(n)\pmod{p^e}
\end{align*}
since, again using \eqref{eq:vp-prod},
$$
v_p\left(\Pi_1\left(\tfrac {p-1} {2}+\al_0p\right)\right)
=
v_p\left(\Pi_1\left(\tfrac {p-1} {2}+\tfrac {e-1} {2}p\right)\right)
\ge e-1.
$$
Consequently, this term cancels with the first term on the right-hand
side of \eqref{eq:u13}.

\medskip
This concludes the induction step, and hence the proof of the theorem is now complete.
\end{proof}

\begin{remarknu} \label{rem:2}
An examination of the above arguments (including the proofs of
Lemmas~\ref{lem:u1} and~\ref{lem:u2} which are used in the above
proof) reveals that the products
$\Pi_1(n)$ and $\Pi_3(n)$ in the definition~\eqref{eq:Rek1-Pi}
could have been replaced with any functions $f(n)$ and $g(n)$ that
satisfy the $p$-divisibility properties in~\eqref{eq:vp-prod}
and~\eqref{eq:4j-3+}, and where $f(n)$
is a quadratic residue modulo~$p$ not divisible by~$p$.
\end{remarknu}

\begin{lemma} \label{lem:u1}
We assume that, given a prime~$p\equiv1$~{\em (mod~$4$)} and an odd
positive integer~$e,$ 
the number $u(n)$ is divisible by $p^h$ for $n\ge \cl{\frac {hp}2}$
and $h<e$. Then
\begin{equation} \label{eq:Rek-p^e1}
u\left(\tfrac {ep-1} {2}\right)\equiv\Pi_1\left(\tfrac {ep-1} {2}\right)
-\sum_{i=1}^{(e-1)/2}\binom {ep}{2ip}
\Pi_3(ip)\,u\!\left(\tfrac {ep-1} {2}-ip\right)
\pmod{p^e}.
\end{equation}
\end{lemma}

\begin{proof}
We put $n=\frac {ep-1} {2}$ in \eqref{eq:Rek1-Pi} and
consider the summand indexed by~$m$ on the right-hand side
for $\cl{\frac{(f-1)p}2}\le m<\cl{\frac{fp}2}$ with $1\le f\le e$.  
By our assumption on the divisibility of $u(m)$ by powers of~$p$
and by \eqref{eq:4j-3+}, we have
$$
v_p\left(\big(\Pi_3\big(\tfrac {ep-1} {2}-m\big)\big) u(m)\right)
\ge 2\fl{\tfrac {\frac {ep-1} {2}-m+\frac {3} {4}(p-1)} {p}}+f-1.
$$
Our conditions imply $\frac {ep-1} {2}-m\ge {\frac{ep-1}2}-\cl{\frac{fp}2}+1\ge
\cl{\frac{(e-f)p}2}$.
Therefore, we obtain
$$
v_p\left(\big(\Pi_3\big(\tfrac {ep-1} {2}-m\big)\big) u(m)\right)
\ge 2\fl{\tfrac {\cl{\frac{(e-f)p}2}+\frac {3} {4}(p-1)} {p}}+f-1\ge( e-f)+f-1=e-1.
$$
In order to show that the
summand is indeed divisible by~$p^e$ for $m$ not of the form
$\frac {ep-1} {2}-ip=\frac {p-1} {2}+\frac {e-2i-1} {2}p$ for some~$i$,
we must therefore prove
that the binomial coefficient in \eqref{eq:Rek1-Pi} is divisible by~$p$
in this case. To see this, we consider the $p$-adic representations
of $2m+1$ and $ep-1-2m$. Since $m$ is not of the form $\frac {p-1}
{2}+sp$ (this condition is essential for the first line below), these are
\begin{alignat}2
\notag
(2m+1)_p&={}&(f-1)_p\ *\\
\big(ep-1-2m\big)_p&={}&(e-f)_p\ *
\label{eq:e-f4}
\end{alignat}
As before, $(\al)_p$ denotes the $p$-adic representation of the
integer~$\al$, and the stars on the right-hand sides indicate 
the right-most digits of $(2m+1)_p$ and $(ep-1-2m)_p$ whose precise
values are irrelevant. The sum of $2m+1$ and $ep-1-2m$ is $ep$
whose $p$-adic representation has the form $(e)_p\ *$. 
Hence, when adding the two numbers on the right-hand sides of
\eqref{eq:e-f4}, at least one carry must occur --- namely one from the
$p^0$-digit to the $p^1$-digit. By Kummer's theorem in
Lemma~\ref{lem:5},
the consequence is that the binomial
coefficient $\binom {ep}{2m+1}$ is divisible by~$p$. Therefore
the corresponding summand is divisible by~$p^e$.

This proves that, under our assumptions, the only summands which
``survive" on the right-hand side of \eqref{eq:Rek1-Pi} modulo $p^e$
are those for which $m$ is of the form $\frac {p-1} {2}+sp$ for
some~$s$. This leads directly to~\eqref{eq:Rek-p^e1}.
\end{proof}

By iteration of the recurrence in \eqref{eq:Rek-p^e1}, we obtain the
following non-recursive congruence.

\begin{lemma} \label{lem:u2}
Under the assumptions of Lemma~{\em\ref{lem:u1},} we have
\begin{multline} \label{eq:Rek-p^e2}
  u\left(\tfrac {ep-1} {2}\right)\equiv
  \sum_{k\ge0}\sum_{\frac {e-1} {2}=\al_0>\al_1>\dots>\al_k\ge0}
  (-1)^k\Pi_1\left(\tfrac {p-1} {2}+\al_kp\right)\\
  \cdot
\prod _{j=0} ^{k-1}\binom {(2\al_j+1)p}
      {(2\al_{j+1}+1)p}\Pi_3\big((\al_j-\al_{j+1})p\big) 
\pmod{p^e}.  
\end{multline}
\end{lemma}

\begin{proof}
We show the claim by an induction on~$e$. 
For the start of the induction, we observe that for $e=1$ the
sum on the right-hand side of \eqref{eq:Rek-p^e2} reduces to a single
term for $k=0$ and $\al_0=0$, which turns out to equal $u(\frac {p-1}
{2})$, in agreement with the claim.

For the induction step, we observe that, by \eqref{eq:vp-prod},
we have $v_p\big(\Pi_3(ip)\big)\ge2i$.
Hence we may consider the term $u\left(\tfrac {ep-1} {2}-ip\right)$
in \eqref{eq:Rek-p^e1} modulo $p^{e-2i}$. This allows us to
substitute the right-hand side of \eqref{eq:Rek-p^e2} for $e$ replaced
by $e-2i$ in \eqref{eq:Rek-p^e1}. If, at the same time, we replace
$\al_j$ by $\al_{j+1}$ for all~$j$, then this leads us to
\begin{align*}
u\left(\tfrac {ep-1} {2}\right)&\equiv\Pi_1\left(\tfrac {ep-1} {2}\right)\\
&\kern.5cm
-\sum_{i=1}^{(e-1)/2}\binom {ep}{(e-2i)p}
\Pi_3(ip)
  \sum_{k\ge0}\sum_{\frac {e-1} {2}-i=\al_1>\al_2>\dots>\al_{k+1}\ge0}
  (-1)^k\Pi_1\left(\tfrac {p-1} {2}+\al_{k+1}p\right)\\
  &\kern5cm
  \cdot
\prod _{j=1} ^{k}\binom {(2\al_j+1)p}
      {(2\al_{j+1}+1)p}\Pi_3\big((\al_j-\al_{j+1})p\big) 
\pmod{p^e}.
\end{align*}
Setting $\al_0=\frac {e-1} {2}$, we realise that the right-hand side
can be put together in the form of the expression on the right-hand
side of~\eqref{eq:Rek-p^e2} with $k$ replaced by $k+1$. This concludes
the induction step and, hence, the induction argument.
\end{proof}

\section{The sequence $(v(n))_{n\ge0}$ modulo odd prime powers}
\label{sec:3}

Here, we analyse the numbers $v(n)$ modulo odd prime
powers~$p^e$. The first result is Theorem~\ref{thm:2} which says that $v(n)$
vanishes modulo~$p^e$ for $n\ge\cl{(e-1)p^2/2}$, for all odd primes and
all integers~$e\ge2$.
Also here, the proof of the theorem is inductive, the
start of the induction being given in Proposition~\ref{prop:v-p^2},
which is itself based on auxiliary results in Lemma~\ref{lem:v-klein}.
Again, should the prime $p$ be congruent to~$1$
modulo~4, Theorem~\ref{thm:2} can be significantly improved; see
Theorem~\ref{thm:2A}.

We begin with the announced auxiliary results,
which afterwards lead to Proposition~\ref{prop:v-p^2}.

\begin{lemma} \label{lem:v-klein}
Let $p$ be an odd prime number,
and let $(v(n))_{n\ge0}$ be defined by the recurrence \eqref{eq:Rek2}. 
Then we have
\begin{equation} \label{eq:v-klein1} 
v(ap+b)\equiv0~(\text{\em mod }p),\quad  \text{for }1\le a\le
\tfrac {p-1} {2}\text{ and\/ }0\le b\le a-1,
\end{equation}
and
\begin{equation} \label{eq:v-klein2} 
v\left(ap+\tfrac {p+1} {2}+b\right)\equiv0~(\text{\em mod }p),\quad 
\text{for }1\le a\le
\tfrac {p-1} {2}\text{ and\/ }0\le b\le a-1.
\end{equation}
\end{lemma}
\begin{proof}
We prove the assertions by induction on the size of $ap+b$
and of $ap+\frac {p+1} {2}+b$, simultaneously. The induction
is based on the recurrence~\eqref{eq:Rek2}.
It is clear that the first term on the right-hand side,
the product $2^{n-1}\Pi_3(n)=2^{n-1}\Pi_3(ap+b)$, is
divisible by~$p$ because of~\eqref{eq:vp-prod}.

\medskip
Now let first $n=ap+b$ with $a$ and $b$ satisfying the conditions
in~\eqref{eq:v-klein1}. We want to prove that $v(n)\equiv0$~(mod~$p$).

From Lemma~\ref{lem:5} we infer that
the binomial coefficient $\binom {2n}{2m}
=\binom {2ap+2b}{2m}$ is divisible by~$p$, except when
$2m=cp+d$ with both $0\le c\le 2a$ and $0\le d\le 2b$.
It is important to note that at this point the upper bounds of
$\frac {p-1} {2}$ for~$a$ and~$b$ enter crucially.

If $c=0$ (and hence $d$ even since $2m=cp+d=d$), then
$v(n-m)=v\left(ap+b-\frac {d} {2}\right)$ is divisible by~$p$
according to the induction hypothesis.
From now on we may assume that $c\ge1$.

We distinguish whether $c$ and~$d$ are both even or both odd.
(There are no other cases since $2m=cp+d$ is even.)

If both $c$ and~$d$ are even, then, provided $c>d$, we may use the
induction hypothesis~\eqref{eq:v-klein1} to infer that
$v(m)=v\left(\frac {c} {2}p+\frac 
{d} {2}\right)$ is divisible by~$p$. On the other hand, if $c\le d$,
then we have
$$
n-m=\left(a-\tfrac {c} {2}\right)p+b-\tfrac {d} {2}.
$$
Since, by assumption, we have $a>b$, we have
$$
a-\tfrac {c} {2}>b-\tfrac {d} {2}\ge0.
$$
Consequently, again using the induction hypothesis~\eqref{eq:v-klein1}, we
obtain that $v(n-m)$ is divisible by~$p$.

Now let both $c$ and~$d$ be odd. Here we may rewrite
$m=\frac {c} {2}p+\frac {d} {2}$ as
$$m=\tfrac {c-1} {2}p+\tfrac {p+1} {2}+\tfrac {d-1} {2}.$$
Furthermore, we
may write
$$
n-m=\left(a-\tfrac {c+1} {2}\right)p
+\tfrac {p+1} {2}+b-\tfrac {d+1} {2}.
$$
It should be noted that, if $c\le d$, we have $a-\tfrac {c+1} {2}>
b-\tfrac {d+1} {2}\ge0.$
As a consequence,
by the induction hypothesis~\eqref{eq:v-klein2},
we have $v(n-m)\equiv0$~(mod~$p$).
If, on the other hand, we have $c\ge d$, then
$v(m)\equiv0$~(mod~$p$), again
by the induction hypothesis~\eqref{eq:v-klein2}.

\medskip
Next we discuss the case where
$n=ap+\frac {p+1} {2}+b$ with $a$ and $b$ satisfying the conditions
in~\eqref{eq:v-klein2}. We want to prove that $v(n)\equiv0$~(mod~$p$).

In the current case, the binomial coefficient on the right-hand side
of~\eqref{eq:Rek2} becomes $\binom {2n}{2m}
=\binom {(2a+1)p+2b+1}{2m}$. Here it must be observed that,
because of the upper bounds on~$a$ and~$b$ in~\eqref{eq:v-klein2}, we
have $2a+1\le p$ and $2b+1\le p-1$.
If $a=\frac {p-1} {2}$, so that $2a+1=p$, then, by Lemma~\ref{lem:5},
$\binom {2n}{2m}=\binom {p^2+2b+1}{2m}$ is divisible by~$p$, except
when $2\le 2m\le 2b$. In this exceptional case, we have
$v(n-m)\equiv0$~(mod~$p$) by induction hypothesis. 

We assume from now on that $a<\frac {p-1} {2}$.
From Lemma~\ref{lem:5} we infer that
the binomial coefficient $\binom {2n}{2m}
=\binom {(2a+1)p+2b+1}{2m}$ is divisible by~$p$, except when
$2m=cp+d$ with both $0\le c\le 2a+1$ and $0\le d\le 2b+1$.

We distinguish whether $c$ and~$d$ are both even or both odd.
(There are no other cases since $2m=cp+d$ is even.)

If both $c$ and~$d$ are even, then, provided $c>d$, we may use the
induction hypothesis~\eqref{eq:v-klein1} to infer that
$v(m)=v\left(\frac {c} {2}p+\frac 
{d} {2}\right)$ is divisible by~$p$. On the other hand, if $c\le d$,
then we have
$$
n-m=\left(a-\tfrac {c} {2}\right)p+\tfrac {p+1} {2}+b-\tfrac {d} {2}.
$$
Since, by assumption, we have $a>b$, we have
$$
a-\tfrac {c} {2}>b-\tfrac {d} {2}\ge0.
$$
Consequently, using the induction hypothesis~\eqref{eq:v-klein2}, we
obtain that $v(n-m)$ is divisible by~$p$.

Now let both $c$ and~$d$ be odd. We may again rewrite
$m=\frac {c} {2}p+\frac {d} {2}$ as
$$m=\tfrac {c-1} {2}p+\tfrac {p+1} {2}+\tfrac {d-1} {2}.$$
Furthermore, we
may write
$$
n-m=\left(a-\tfrac {c-1} {2}\right)p+b-\tfrac {d-1} {2}.
$$
Hence, according to the induction hypothesis~\eqref{eq:v-klein2},
we have $v(m)\equiv0$~(mod~$p$) if $c>d$, and otherwise we have
$v(n-m)\equiv0$~(mod~$p$) according to the induction
hypothesis~\eqref{eq:v-klein1}.
\end{proof}

\begin{proposition} \label{prop:v-p^2}
Let $(v(n))_{n\ge0}$ be defined by the recurrence \eqref{eq:Rek2}. 
Then, given an odd prime~$p,$
the number $v(n)$ is divisible by $p^2$ for $n\ge \cl{\frac {p^2} {2}}$.
\end{proposition}

\begin{proof}
Let $n\ge \cl{\frac {p^2} {2}}=\frac {p^2+1} {2}$.
Again we use induction on~$n$.
The induction will be based on~\eqref{eq:Rek2},
and it proceeds by showing that
each summand on the right-hand side is divisible by~$p^2$.

It is again clear that the first term on the right-hand side,
the product $\Pi_3(n)=\Pi_3(ap+b)$, is
divisible by~$p^2$ because of~\eqref{eq:vp-prod}.

\medskip
In the sequel,
we may restrict our attention to $n< p^2$ because otherwise
either $m\ge \frac {p^2+1} {2}$ or
$n-m\ge \frac {p^2+1} {2}$. The induction hypothesis then implies
that either\break
$v(m)\equiv0$~(mod~$p^2$)
or $v(n-m)\equiv0$~(mod~$p^2$), and therefore each term on the
right-hand side of~\eqref{eq:Rek2} would be divisible by~$p^2$.

To summarise the discussion so far: we may write $2n$ as
$2n=p^2+ap+b$ for some $a$ and $b$ not of the same parity
and with $0\le a,b\le p-1$.

\medskip
If there are two carries when adding $2m$ and $2(n-m)$ in their
$p$-adic representations, then,
by Lemma~\ref{lem:5},
the binomial coefficient $\binom {2n}{2m}$ is divisible by~$p^2$.
Hence the corresponding summand on the right-hand side
of~\eqref{eq:Rek2} is divisible by~$p^2$.

\medskip
Now let us assume that there is only one carry when adding
$2m$ and $2(n-m)$ in their
$p$-adic representations. As earlier, if either $m$ or~$n-m$ is larger
than~$\frac {p^2} {2}$, then $v(m)$ or $v(n-m)$ is divisible
by~$p^2$,
and thus as well the corresponding summand on the right-hand
side of~\eqref{eq:Rek2}. We may therefore assume without loss of
generality that $2m=cp+d$ with $c\le \frac {p-1} {2}$.
Furthermore,
we have $d\le b$ since, otherwise, there would be two carries when
adding $2m$ and~$2n-2m$ in their
$p$-adic representations.

We distinguish again whether $c$ and~$d$ are both even or both odd.

Let first $c$ and $d$ be even. We write $m=\frac {c} {2}p+\frac {d} {2}$.
If $c>d$, then, by~\eqref{eq:v-klein1},
we infer $v(m)\equiv0$~(mod~$p$). If $c\le d$, then we may write
$$
n-m=\left(\tfrac {p+a-c} {2}\right)p+\tfrac {b-d} {2}
=\left(\tfrac {p+a-c-1} {2}\right)p+\tfrac {p+1} {2}+\tfrac {b-d-1} {2}.
$$
We use one of the two expressions, depending on the parities of $a$
and~$b$ (the reader should recall that $a$ and~$b$ have different
parities), so that the fractions produce integers.
By our assumptions, we have
$$
\tfrac {p+a-c} {2}> \tfrac {b+a-c} {2}
\ge \tfrac {b-c} {2}
\ge \tfrac {b-d} {2}\ge0.
$$
By Lemma~\ref{lem:v-klein}, this implies that
$v(n-m)\equiv0$~(mod~$p$).
In total, in both cases this shows that the
corresponding summand in~\eqref{eq:Rek2} is divisible by~$p^2$.

Now let $c$ and $d$ be odd.
Here we write $m=\frac {c-1} {2}p+\frac {p+1} {2}+\frac {d-1} {2}$.
If $c>d$, then by~\eqref{eq:v-klein2},
we infer $v(m)\equiv0$~(mod~$p$). If $c\le d$, then we may again write
$$
n-m=\left(\tfrac {p+a-c} {2}\right)p+\tfrac {b-d} {2}
=\left(\tfrac {p+a-c-1} {2}\right)p+\tfrac {p+1} {2}+\tfrac {b-d-1} {2}.
$$
Arguing as before, we conclude that
$v(n-m)\equiv0$~(mod~$p$).
This shows again that the
corresponding summand in~\eqref{eq:Rek2} is divisible by~$p^2$.

\medskip
If there is no carry when adding
$2m$ and $2(n-m)$ in their
$p$-adic representations, then
necessarily one of $2m$ or $2(n-m)$ is at least $p^2+1$.
So, again, one of $v(m)$ or $v(n-m)$ is divisible by~$p^2$
due to the induction hypothesis.

\medskip
This completes the proof of the proposition.
\end{proof}

The following theorem proves Conjecture~18(1) in \cite{WakhAA} for $v(n)$.

\begin{theorem} \label{thm:2}
Let $(v(n))_{n\ge0}$ be defined by the recurrence \eqref{eq:Rek2}. 
Then, given an odd prime~$p$ and a positive integer~$e\ge2,$ 
the number $v(n)$ is divisible by $p^e$ for $n\ge \cl{\frac {(e-1)p^2} {2}}$.
\end{theorem}

\begin{proof}
In analogy with the proof of Theorem~\ref{thm:1}, we proceed
by a double induction on~$e$
and~$n$, the outer induction being on~$e$.
For the start of the induction, we use Proposition~\ref{prop:v-p^2}
which proves the assertion of the theorem for $e=2$.
From now on let $e\ge3$.

We assume that $\cl{\frac {(e-1)p^2} {2}}\le n<\cl{\frac {ep^2} {2}}$.
We claim that 
the first term on the right-hand side of~\eqref{eq:Rek2},
namely $2^{n-1}\Pi_3(n)$, is always
divisible by~$p^e$ under this assumption. Indeed,
by \eqref{eq:vp-prod} we have
$$
v_p\left(2^{n-1}\Pi_3(n)\right)
\ge
2\fl{\tfrac {n} {p}}
\ge
2\fl{\tfrac {1} {p}\cl{\tfrac {(e-1)p^2} {2}}}
\ge
2\fl{\tfrac {(e-1)p} {2}}
\ge 2(e-1)\ge e
$$
for $e\ge3$.

Next we consider the summand on the right-hand side
of~\eqref{eq:Rek2} for $\cl{\frac {(f-1)p^2} {2}}\le m<\cl{\frac {fp^2}
  {2}}$ with $2\le f\le e-2$.
By the induction hypothesis applied to $v(m)$, we have $v_p(v(m))\ge f$.
Furthermore we have
$$
n-m> \cl{\tfrac {(e-1)p^2} {2}}-\cl{\tfrac {fp^2} {2}}
= \cl{\tfrac {(e-f-1)p^2} {2}}-\chi(e,f\text{ odd}),
$$
where $\chi(\mathcal A)=1$ if $\mathcal A$ is true and
$\chi(\mathcal A)=0$ otherwise.
The above inequality implies
$$
n-m\ge \cl{\tfrac {(e-f-1)p^2} {2}}.
$$
Therefore, if we apply the induction hypothesis to $v(n-m)$,
then we obtain\break $v_p(v(n-m))\ge e-f$.
%
We infer that $v_p\big(v(m)v(n-m)\big)\ge f+(e-f)=e$, and hence the
corresponding summand in \eqref{eq:Rek2} is divisible by~$p^e$.

Let now $0\le m< \cl{\tfrac {p^2} {2}}$.
In that case, the previous argument only shows that
$n-m\ge \cl{\tfrac {(e-2)p^2} {2}}$, with the consequence that we only
have $v_p(v(n-m))\ge e-1$.
It might be that actually $n-m\ge \cl{\tfrac {(e-1)p^2} {2}}$.
Then the induction hypothesis (in~$n$) implies that
$v_p(v(n-m))\ge e$, so that 
the corresponding summand in \eqref{eq:Rek2} is divisible by~$p^e$.
On the other hand, if
$\cl{\tfrac {(e-2)p^2} {2}}\le n-m< \cl{\tfrac {(e-1)p^2} {2}}$,
then we may write
$2n=(e-1)p^2+ap+b$ and $2(n-m)=(e-2)p^2+cp+d$ for some $a,b,c,d$
with $0\le a,b,c,d\le p-1$. Then there is (at least) one carry
when adding $2(n-m)$ and $2m$ in their
$p$-adic representations. Therefore,
by Lemma~\ref{lem:5}, the binomial coefficient
$\binom {2n}{2m}$ is divisible by~$p$, and together with
the previously observed fact that $v_p(v(n-m))\ge e-1$ this shows that the
corresponding summand in \eqref{eq:Rek2} is divisible by~$p^e$.

Finally, let
$\cl{\tfrac {(e-2)p^2} {2}}\le m< \cl{\tfrac {(e-1)p^2} {2}}$.
The induction hypothesis (in~$n$) implies that
in this case we have $v_p(v(m))\ge e-1$.
If $n-m\ge \cl{\tfrac {p^2} {2}}$,
then the induction hypothesis (in~$n$) implies that
$v(n-m)$ is divisible by~$p^2$, so that 
the corresponding summand in \eqref{eq:Rek2} is divisible by~$p^e$.
On the other hand, if
$0\le n-m< \cl{\tfrac {p^2} {2}}$,
then we may write
$2n=(e-1)p^2+ap+b$ and $2m=(e-2)p^2+cp+d$ for some $a,b,c,d$
with $0\le a,b,c,d\le p-1$. Then there is (at least) one carry
when adding $2(n-m)$ and $2m$ in their
$p$-adic representations. Therefore, by Lemma~\ref{lem:5},
the binomial coefficient
$\binom {2n}{2m}$ is divisible by~$p$, and together with
the previously observed fact that $v_p(v(m))\ge e-1$ this shows that the
corresponding summand in \eqref{eq:Rek2} is divisible by~$p^e$.

\medskip
This concludes the induction step, and, thus, the proof of the theorem.
\end{proof}

\begin{remarknu} \label{rem:3}
An examination of the above arguments reveals that the product
$\Pi_3(n)$ in the definition~\eqref{eq:Rek2}
could have been replaced with any function $f(n)$ that
satisfies the $p$-divisibility property in~\eqref{eq:vp-prod}.
\end{remarknu}

\begin{theorem} \label{thm:2A}
Let $(v(n))_{n\ge0}$ be defined by the recurrence \eqref{eq:Rek2}. 
Then, given a prime\linebreak $p\equiv1$~{\em (mod~$4$)} and a positive integer~$e,$ 
the number $v(n)$ is divisible by $p^e$ for $n\ge \cl{\frac{ep}2}$.
\end{theorem}

\begin{proof}
In analogy with the proof of Theorem~\ref{thm:1}, we proceed
by a double induction on~$e$
and~$n$, the outer induction being on~$e$.
Clearly, the theorem is trivial for $e=0$, which serves as the start
of the induction.

Now let $n\ge \cl{\frac{ep}2}$.
Hence, by \eqref{eq:4j-3+} for $n\ge \cl{ep/2}$ we have
\[
v_p\big(\Pi_3(n)\big)\ge
2\fl{\frac {\cl{\frac {ep} {2}}+\frac {3} {4}(p-1)} {p}}
\ge 2\fl{{\frac {e} {2}+\frac {3} {4}-\frac {3} {4p}}}\ge e.
\]
Thus,
the first term on the right-hand side of \eqref{eq:Rek2} is divisible
by~$p^e$.

We consider the summand on the right-hand side of
\eqref{eq:Rek2} for $\cl{\frac{fp}2}\le m<\cl{\frac{(f+1)p}2}$ with $0\le f<e$.
By the induction hypothesis applied to $v(m)$, we have $v_p(v(m))=f$.
On the other hand, if we apply the induction hypothesis to $v(n-m)$,
then we obtain $v_p(v(n-m))=e-f-1$ or $v_p(v(n-m))=e-f$, depending on
whether or not $n-m<\cl{\frac{(e-f)p}2}$. In the latter case,
we infer that $v_p\big(v(m)v(n-m)\big)=f+(e-f)=e$, and hence the
corresponding summand in \eqref{eq:Rek2} is divisible by~$p^e$.

On the other hand, in the former case, the conclusion is that 
$v_p\big(v(m)v(n-m)\big)= e-1$. In order to show that the
corresponding summand is divisible by~$p^e$, we must therefore prove
that the binomial coefficient in \eqref{eq:Rek2} is divisible by~$p$
in this case. To see this, we consider the $p$-adic representations
of $2m$ and $2(n-m)$. These are
\begin{alignat}2
\notag
(2m)_p&={}&(f)_p\ *\\
\big(2(n-m)\big)_p&={}&(e-f-1)_p\ *
\label{eq:e-fA}
\end{alignat}
Here as before, $(\al)_p$ denotes the $p$-adic representation of the
integer~$\al$, and the stars on the right-hand sides indicate 
the right-most digits of $(2m)_p$ and $(2(n-m))_p$ whose precise
values are irrelevant. The sum of $2m$ and $2(n-m)$ is $2n$
whose $p$-adic representation has the form $(e)_p\ *$. 
Hence, when adding the two numbers on the right-hand sides of
\eqref{eq:e-fA}, at least one carry must occur --- namely one from the
$p^0$-digit to the $p^1$-digit. By Kummer's theorem in
Lemma~\ref{lem:5},
the consequence is that the binomial
coefficient $\binom {2n}{2m}$ is divisible by~$p$. Therefore, again,
the corresponding summand is divisible by~$p^e$.

Finally, if $\cl{\frac {ep^2} {2}}\le m<n$, by the induction
hypothesis (for~$n$), the number $v(m)$ is divisible by~$p^e$, which of
course implies divisibility of the corresponding summand by~$p^e$.

\medskip
This concludes the induction step, and, thus, the proof of the theorem.
\end{proof}

\begin{remarknu} \label{rem:4}
An examination of the above arguments reveals that the product
$\Pi_3(n)$ in the definition~\eqref{eq:Rek2}
could have been replaced with any function $f(n)$ that
satisfies the $p$-divisibility property in~\eqref{eq:4j-3+}.
\end{remarknu}

\section{The inverse of the matrix $\mathbf R$
  modulo prime powers $p^e$ with $p\equiv3~(\text{mod }4)$}
\label{sec:4}

The goal of this section is to derive a divisibility result for the
matrix entries $R^{-1}(n,k)$ by prime powers~$p^e$ with $p\equiv3$~(mod~4)
that would allow us to carry through an inductive proof of
Theorem~\ref{thm:main}(1) via the use of the relation~\eqref{eq:d-v}.
This result is presented in Theorem~\ref{thm:9}. It is
based on Proposition~\ref{prop:n-k-m} which is stated and proved before.

\medskip
We start by deriving an explicit expression for $R^{-1}(2n,2k)$,
the other entries of the matrix $\mathbf
R^{-1}=(R^{-1}(n,k))_{n,k\ge0}$ being zero due to
the checkerboard pattern of the matrix; cf.\ Proposition~\ref{prop:1}.
By the definition \eqref{eq:R^{-1}}, we have
\begin{align}
\notag
R^{-1}&(2n,2k)=2^{n-k}\frac {(2n-1)!} {(2k-1)!}\coef{t^{2n-2k}}
\big(U(t)/t\big)^{-2n}\\[1mm]
\notag
&=2^{n-k}\frac {(2n-1)!} {(2k-1)!}\coef{t^{2n-2k}}
\sum_{m\ge0}\binom {-2n}m\big(U(t)/t-1\big)^m\\[1mm]
\notag
&=2^{n-k}\frac {(2n-1)!} {(2k-1)!}\coef{t^{2n-2k}}
\sum_{m\ge0}(-1)^m\binom {2n+m-1}m
\Bigg(\sum_{j\ge1}\frac {u(j)} {(2j+1)!}t^{2j}\Bigg)^m\\[1mm]
\notag
&=2^{n-k}\frac {(2n-1)!} {(2k-1)!}
\sum_{m\ge0}(-1)^m\binom {2n+m-1}m\\[1mm]
\notag
&\kern1cm
\cdot
\underset{c_1=0}{\sum_{(c_i)\in\mathcal P^o_{2n-2k+m,m}}}
\frac {m!} {3!^{c_3}c_3!\,5!^{c_5}c_5!\cdots (2n-2k+1)!^{c_{2n-2k+1}}c_{2n-2k+1}!}
\prod _{i=1} ^{2n-2k+1}
{u^{c_i}\big(\tfrac {i-1} {2}\big)}\\[1mm]
&=
\sum_{m\ge0}\
\underset{c_1=0}{\sum_{(c_i)\in\mathcal P^o_{2n-2k+m,m}}}
(-1)^m2^{n-k}
\frac {(2n+m-1)!}
{(2k-1)!\,
\prod _{i=1} ^{2n-2k+1}i!^{c_i}c_i!}
\prod _{i=1} ^{2n-2k+1}
{u^{c_i}\big(\tfrac {i-1} {2}\big)}.
\label{eq:R^(-1)p}
\end{align}
Here, the symbol
$\mathcal P^{o}_{N,K}$ stands for the set of all
tuples $(c_1,c_2,\dots,c_N)$ of
non-negative integers~$c_i$ for which $c_{2j}=0$ for all~$j$, and
which satisfy
\begin{align} \label{eq:ci1}
c_1+c_3+\dots+c_{2n-1}&=K,\\[1mm]
c_1+3c_3+\dots+(2n-1)c_{2n-1}&=N,
\label{eq:ci2}
\end{align}
with $n=\cl{N/2}$.
It should be noted that, if $N$ and $K$ do not have the same parity, then
the set $\mathcal P^{o}_{N,K}$ is empty. 

The next proposition provides a lower bound on the $p$-divisibility
of (the essential part of) the summand on the right-hand side
of~\eqref{eq:R^(-1)p}.

\begin{proposition} \label{prop:n-k-m}
Let $n,k,m,e,f$ be positive integers with $n\ge k,$ and let
$p$ be a prime number with $p\equiv3$~{\em(mod $4$)}.
If $2n\ge ep^2$, $(f-1)p^2\le 2k-1<fp^2,$
and $(c_i)\in \mathcal P^{o}_{2n-2k+m,m}$ 
with $c_1=0$, then the expression
\begin{equation} \label{} 
F(n,k,m,(c_i)):=\frac {(2n+m-1)!} {(2k-1)!\,
\prod _{i=1} ^{2n-2k+1}i!^{c_i}\,c_i!}\prod _{i=1} ^{2n-2k+1}
{u^{c_i}\big(\tfrac {i-1} {2}\big)}
\end{equation}
is divisible by~$p^{e-f+1}$.
Furthermore, if $f=1$, then the above expression multiplied by~$v(k)$
is divisible by~$p^{e+1}$.
\end{proposition}

\begin{proof}
As a consequence of Kummer's theorem in Lemma~\ref{lem:5}
(the multinomial coefficient below can be written as a product
of binomial coefficients), we have
\begin{multline} \label{eq:multinom-carries}
v_p\left(\frac {(2n+m-1)!} {(2k-1)!
\prod _{i=1} ^{2n-2k+1}i!^{c_i}}\right)\\
=\#\big(\text{carries when performing the addition
$(2k-1)+\textstyle\sum_{i=1}^{2n-2k+1}c_i\cdot i$}\big).
\end{multline}
Here, ``performing the addition $(2k-1)+\sum_{i=1}^{2n-2k+1}c_i\cdot i$" means
that we start with $(2k-1)_p$ (the $p$-adic representation of~$2k-1$),
then add $(1)_p$ to it $c_1$ many times,
then add $(3)_p$ to it $c_3$ many times (the reader should recall that
even-indexed $c_i$'s are zero by definition of~$\mathcal
P^{o}_{2n-2k+m,m}$), etc. Carries are recorded at each single addition.
Moreover, we have
$$
v_p\big({u\big(\tfrac {i-1} {2}\big)}\big)
\ge\begin{cases}
\fl{\tfrac {i} {p^2}}+1,&\text{for }i\ge p^2,\\
1,&\text{for }i=ap+b\text{ with $1\le a\le p-1$ and $0\le b<a$,}
\end{cases}
$$
the first alternative being due to Theorem~\ref{thm:1},
and the
second alternative being due to Lemma~\ref{lem:u-klein}.

Combining the last inequality with~\eqref{eq:multinom-carries}, we
obtain 
\begin{multline} \label{eq:F(n,k,m,c)1}
v_p\big(F(n,k,m,(c_i))\big)\\\ge
\#\big(\text{carries when performing the addition
  $(2k-1)+\textstyle\sum_{i=1}^{2n-2k+1}c_i\cdot i$}\big)\\
-\sum_{i=1}^{2n-2k+1}v_p(c_i!)
+\sum_{i=p^2}^{2n-2k+1}c_i\cdot\left(\fl{\tfrac {i} {p^2}}+1\right)
+{\sum}{}^{\displaystyle\prime}c_i,
\end{multline}
where $\sum\!{}^\prime$ is over all odd~$i$ with $1\le i<p^2$ and
$i=ap+b$ for some $a$ and~$b$ with $1\le a\le p-1$ and $0\le b<a$.

\medskip
We begin with the negative term in~\eqref{eq:F(n,k,m,c)1},
namely $-\sum_{i=1}^{2n-2k+1}v_p(c_i!)$. By Legendre's formula
in Lemma~\ref{lem:4}, we have $v_p(c_i!)\le \frac {c_i} {p-1}$, so
that 
$$
-\sum_{i=1}^{2n-2k+1}v_p(c_i!)\ge -\sum_{i=1}^{2n-2k+1}\frac {c_i} {p-1}
=-\frac {m} {p-1},
$$
the last step being due to the fact that the tuple $(c_i)$ is in
$\mathcal P^{o}_{2n-2k+m,m}$.
To compensate this term, we consider the $p^0$-digits of $i_1,i_2,\dots,i_m$,
the carries that they cause in the addition $i_1+i_2+\dots+i_m$,
and contributions from $\sum\!{}^\prime c_i$
on the right-hand side of~\eqref{eq:F(n,k,m,c)1}. Let
us assume that exactly $t$ of the $i_\ell$'s are $\equiv0,1$~(mod~$p$).
Then we get
\begin{multline} \label{eq:p0-1}
\#\big(\text{carries from the $p^0$-digit to the
$p^1$-digit when doing the addition
  $\sum_{i=1}^{2n-2k+1}c_i\cdot i$}\big)\\
+{\sum}{}^{\displaystyle\prime}c_i
-\sum_{i=1}^{2n-2k+1}v_p(c_i!)
\ge \fl{\tfrac {2(m-t)} {p}}+t-\tfrac {m} {p-1}
\ge \fl{\tfrac {2m} {p}}-\tfrac {m} {p-1}\ge 0
\end{multline}
as long as $m\ge p$. However, if $m<p$, then the sum
$\sum_{i=1}^{2n-2k+1}v_p(c_i!)$ equals zero, so that the
left-hand side in~\eqref{eq:p0-1} is also non-negative
in this case. 

\medskip
For the next step, in order to ease notation, let
$$
(i_1,i_2,\dots,i_m)=(3^{c_3},5^{c_5},\dots),
$$
where $i^{c_i}$ stands for the sequence $i,i,\dots,i$,
with $i$ repeated $c_i$~times. We assume that $j_\ell p^2\le
i_\ell<(j_\ell +1)p^2$ for all~$\ell$. Then, when concentrating
on the $p^2$-digit in the $p$-adic representations of $2n+m-1$, $2k-1$,
$i_1$, $i_2$, \dots,~$i_m$ while performing
the addition described on the right-hand side
of~\eqref{eq:multinom-carries}, namely
\begin{equation} \label{eq:addition} 
(2k-1)+i_1+i_2+\dots+i_m
\end{equation}
(resulting in $2n+m-1$),
we may extract the inequality
\begin{equation} \label{eq:C-e} 
(f-1)+j_1+j_2+\dots+j_m+C
\ge e,
\end{equation}
where $C$ is the number of carries from the $p^1$- to the
$p^2$-digit when performing the addition~\eqref{eq:addition}.

\medskip
If we now use \eqref{eq:p0-1} and~\eqref{eq:C-e} 
in~\eqref{eq:F(n,k,m,c)1}, then we obtain
\begin{multline} \label{eq:F(n,k,n,c)2}
v_p\big(F(n,k,m,(c_i))\big)\ge
C+\sum_{i=p^2}^{2n-2k+1}c_i\cdot\left(\fl{\tfrac {i} {p^2}}+1\right)\\
\ge \big(e-(f-1)-j_1-j_2-\dots- j_m\big)
+j_1+j_2+\dots +j_m=e-f+1.
\end{multline}

\medskip
Finally we address the case where $f=1$. The previous arguments show
that
$$v_p\big(F(n,k,m,(c_i))\big)\ge e.$$
Hence, the task is to
find --- so-to-speak --- an additional $+1$ somewhere, which may also come
from $v_p(v(k))$. 

Inspection of \eqref{eq:F(n,k,n,c)2} shows that, using the earlier
notation, we gain such an additional $+1$ whenever one of the
$i_\ell$'s is at least~$p^2$ since in the
inequality~\eqref{eq:F(n,k,n,c)2} we dropped the $+1$ in
$\left(\fl{\tfrac {i} {p^2}}+1\right)$ on the right-hand side.
Hence, from now on, we may assume that $i_\ell<p^2$ for all~$\ell$.

We may furthermore assume $m<p$ because for $m\ge p$ we have actually
strict inequality in~\eqref{eq:p0-1} thereby having again found an
additional~$+1$. In particular, as already remarked earlier, the
restriction $m<p$ implies that the sum $\sum_{i=1}^{2n-2k+1}v_p(c_i!)$
equals zero so that the argument in~\eqref{eq:p0-1} is not required,
and consequently no carries from the $p^0$-digit to the $p^1$-digit
need to be considered at this point.

Since in the current case $2k-1<p^2$, from~\eqref{eq:C-e} we 
deduce that the number $C$ of carries from before is at least~$e$.
If there is an $i_\ell$ of the form $i_\ell=ap+b$ with
$1\le a\le  p-1$ and $0\le b<a$,
then we get an additional $+1$ from $\sum\!{}^\prime c_i$
in~\eqref{eq:F(n,k,m,c)1}. If $k=ap+b$ or $k=ap+\frac {p+1} {2}+b$
with $1\le a\le \frac {p-1} {2}$ and $1\le b<a$ so that $2k-1=Ap+B$
with $2\le A\le p-1$ and $0\le B\le A-3$ then $v_p(v(k))\ge1$ by
Lemma~\ref{lem:v-klein}. Therefore the remaining case to discuss is
when $i_\ell=a_\ell p+b_\ell$ with $a_\ell\le b_\ell$, $1\le \ell\le
m$, and $2k-1=Ap+B$ 
with $A\le B+1$. Since the $i_\ell$'s are odd, we have actually 
$a_\ell<b_\ell$ for all~$\ell$. 
Since $m\ge 1$, the above conditions imply that
\begin{equation} \label{eq:a-b-sum} 
A+a_1+a_2+\dots+a_m\le (B+1)+(b_1-1)+b_2+\dots+b_m.
\end{equation}
We know that the number of carries $C$ from the $p^1$-digit to the
$p^2$-digit when adding $(2k-1)+i_1+i_2+\dots+i_m$ is at least
$e\ge1$. Hence, either 
$B+b_1+b_2+\dots+b_m\ge p$, thus creating a carry from the
$p^0$-digit to the $p^1$-digit when adding $(2k-1)+i_1+i_2+\dots+i_m$,
or $A+a_1+a_2+\dots+a_m\ge p$. But then \eqref{eq:a-b-sum}
implies again $B+b_1+b_2+\dots+b_m\ge p$ with the consequence of an
additional carry. Consequently, we have found the additional $+1$ in
all cases.

\medskip
This completes the proof of the proposition.
\end{proof}

\begin{theorem} \label{thm:9}
Let $n,k,e,f$ be non-negative integers with $n\ge k,$ and let
$p$ be a prime number with $p\equiv3$~{\em(mod $4$)}.
If $n\ge\cl{\frac {ep^2} {2}}$ and $k<\cl{\frac {fp^2} {2}},$
then $R^{-1}(2n,2k)$ is divisible by~$p^{e-f+1}$.
Moreover, if $f=1$ then $R^{-1}(2n,2k)v(k)$ is divisible by~$p^{e+1}$.
\end{theorem}

\begin{proof}
The result follows immediately when using Proposition~\ref{prop:n-k-m}
in~\eqref{eq:R^(-1)p}.
\end{proof}

\section{The sequence $(d(n))_{n\ge0}$ modulo prime powers $p^e$ with
$p\equiv3~(\text{mod }4)$}
\label{sec:5}

We are now in a position to prove our first main result.
The following theorem proves Conjecture~18(3) in \cite{WakhAA}, in the
stronger form given in Theorem~\ref{thm:main}(1).

\begin{theorem} \label{thm:10}
Given a prime~$p$ with $p\equiv3$~{\em(mod~$4$)} and an integer~$e\ge2,$ 
the number $d(n)$ is divisible by $p^e$ for $n\ge \cl{\frac {(e-1)p^2} {2}}$.
\end{theorem}

\begin{proof}We prove the assertion by induction on~$n$.
Let $n\ge\cl{\frac {(e-1)p^2} {2}}$.
We use the relation~\eqref{eq:d-v}. We consider some $k$
with $(f-1)p^2\le k<fp^2$.
If $f\ge2$, then Theoremxx~\ref{thm:9} and~\ref{thm:2}
imply
\begin{equation} \label{eq:vpR-1} 
v_p\big(R^{-1}(2n,2k)v(k)\big)\ge (e-f)+f=e.
\end{equation}
On the other hand, if $f=1$ then the supplement in Theorem~\ref{thm:9}
says that~\eqref{eq:vpR-1} also holds in this case.
In other words, each term on the right-hand side of 
\eqref{eq:d-v} is divisible by~$p^e$, hence so is $d(n)$, as desired.
\end{proof}

\section{The sequence $(v(n))_{n\ge0}$ modulo powers of $2$}
\label{sec:6}

The purpose of this section is to show periodicity of the sequence
$\big(v(n)\big)_{n\ge0}$ modulo powers of~2, together with a precise
statement on the period length. The corresponding
analysis is much more involved than the preceding ones modulo odd prime
powers. Our starting point is an expansion that is very similar in
spirit to the one in Section~\ref{sec:4},
namely~\eqref{eq:v(n)Expr}. As a matter of fact, the main results of this
section concern a polynomial refinement of $v(n)$ in which the product
$\Pi_3(j)=
\prod _{\ell=1} ^{j}(4\ell-3)^2$ gets replaced by a variable~$x(j)$, $j=0,1,\dots$, with
the only restriction that $x(0)=1$ and that both $x(1)$ and $x(2)$ are odd.

The main results in Theorems~\ref{prop:2} and~\ref{prop:2a} need
substantial preparations. These are the contents of
Lemma~\ref{lem:6}, Corollaries~\ref{lem:7} and \ref{lem:8}, and
Lemmas~\ref{lem:9}--\ref{lem:11}, which provide formulae and bounds
for the 2-adic valuation of the ratios of factorials that appear in the
expansion~\eqref{eq:v(n)Expr}.

\begin{lemma} \label{lem:6}
For all positive integers $n$ and $c_2$ with $n\ge 2c_2,$ we have
\begin{equation} \label{eq:c_2} 
v_2\left(\frac {(2n)!} {2!^{n-2c_2}(n-2c_2)!\,4!^{c_2}c_2!}\right)=
\#(\text{\em carries when adding $(n-2c_2)_2$ and $(2c_2)_2$}),
\end{equation}
where, as before, $(\al)_2$ denotes the $2$-adic representation of the
integer~$\al$.
\end{lemma}

\begin{proof}
By Legendre's formula in Lemma~\ref{lem:4}, we obtain
\begin{align*}
v_2&\left(\frac {(2n)!} {2!^{n-2c_2}(n-2c_2)!\,4!^{c_2}c_2!}\right)\\[1mm]
&\kern1.5cm
=
2n-s_2(2n)-(n-2c_2)-\big(n-2c_2-s_2(n-2c_2)\big)-3c_2-(c_2-s_2(c_2))\\[1mm]
&\kern1.5cm
=-s_2(n)+s_2(n-2c_2)+s_2(2c_2).
\end{align*}
By Lemma~\ref{lem:4-5},
this is indeed the number of carries when
adding $n-2c_2$ and $2c_2$ in their $2$-adic representations. 
\end{proof}

\begin{corollary} \label{lem:7}
For all positive integers $n,$ we have
$$
v_2\left(\frac {(2n)!} {2!^{n-2}(n-2)!\,4!}\right)=\begin{cases} 
v_2(n)-1,&\text{if $n$ is even,}\\[1mm]
v_2(n-1)-1,&\text{if $n$ is odd}.
\end{cases}
$$
\end{corollary}

\begin{proof}
We appeal to Lemma~\ref{lem:6} with $c_2=1$.
If $n$ is even, say $n=2^{v_2(n)}n_0$ for an odd integer~$n_0$ and $v_2(n)\ge1$,
then the number of carries in \eqref{eq:c_2} is $v_2(n)-1$. On the other hand,
if $n$ is odd, say $n=2^{v_2(n-1)}n_0+1$ for an odd integer~$n_0$,
then this number of carries is $v_2(n-1)-1$. 
\end{proof}

\begin{corollary} \label{lem:8}
For all positive integers $n$, we have
$$
v_2\left(\frac {(2n)!} {2!^{n-4}(n-4)!\,4!^22!}\right)=\begin{cases} 
v_2(n)-2,&\text{if $n\equiv0$ {\em(mod $4$)},}\\[1mm]
v_2(n-1)-2,&\text{if $n\equiv1$ {\em(mod $4$)},}\\[1mm]
v_2(n-2)-2,&\text{if $n\equiv2$ {\em(mod $4$)},}\\[1mm]
v_2(n-3)-2,&\text{if $n\equiv3$ {\em(mod $4$)}.}
\end{cases}
$$
\end{corollary}

\begin{proof}
Again, we appeal to Lemma~\ref{lem:6}, here with $c_2=2$.
If $n\equiv0$~(mod~$4$), say $n=2^{v_2(n)}n_0$ for an odd integer~$n_0$
and $v_2(n)\ge2$,
then the number of carries in \eqref{eq:c_2} is $v_2(n)-2$. 
The other cases are treated similarly. We leave the details to the reader.
\end{proof}

\begin{lemma} \label{lem:9}
For all positive integers $n$ and $c_2$ with $n\ge2c_2\ge6,$ we have
\begin{equation} \label{eq:max2} 
c_2+v_2\left(\frac {(2n)!} {2!^{n-2c_2}(n-2c_2)!\,4!^{c_2}c_2!}\right)\ge
1+\max\limits_{n-2c_2<i\le n}v_2(i).
\end{equation}
\end{lemma}

\begin{proof}
Choose $\al$ and $\be$ such that $n=n_02^\al+n_1$, with $n_0$ odd and
 $n_1<2^{\be+2}$, and
$2^\be\le c_2<2^{\be+1}$. These conditions imply that the $2$-adic
representations can be schematically indicated as
\begin{alignat}2
\notag
&&\hphantom{\dots\,}\underset{\downarrow}{\scriptstyle \al}\kern24pt
\underset{\downarrow}{\scriptstyle \be+2}\kern-5pt
&\\
\notag
(n)_2&={}&\dots\,10\,\dots\,0&*\,\dots\\
(2c_2)_2&=&&\kern3pt1\,\dots
\label{eq:n-c_2}
\end{alignat}
with the meaning that the $1$ shown in the $2$-adic representation
of~$n$ is the $\al$-th digit counted from right (the counting starting with~$0$),
and that the left-most (non-zero) digit of the $2$-adic representation of~$2c_2$
is the $(\be+1)$-st digit, that is, $2c_2<2^{\be+2}$.

First, we assume that $\be\ge2$. The case of $\be=1$, which implies that
$2c_2=6$, will be disposed of at the end of this proof.

We distinguish two cases, depending on the relative sizes
of~$n_1$ and $2c_2$.

\medskip
{\sc Case 1: $n_1<2c_2$}. 
Here, by inspection of \eqref{eq:n-c_2}, we see that 
\begin{equation} \label{eq:max1} 
\max\limits_{n-2c_2<i\le n}v_2(i)= \al.
\end{equation}
On the other hand, by Lemma~\ref{lem:6}, 
the term $v_2(\,.\,)$ on the left-hand side of \eqref{eq:max2} equals
the number of carries when
adding $n-2c_2$ and $2c_2$ in their $2$-adic representations. 
Equivalently, this is the number of carries when performing the
subtraction of $(2c_2)_2$ from $(n)_2$.
Here, inspection of \eqref{eq:n-c_2} yields that this is at least
$\al-(\be+2)+1$. Thus, together with \eqref{eq:max1} and the elementary inequality
\begin{equation} \label{eq:2^h} 
2^\be\ge \be+2\quad \text{for }\be\ge2,
\end{equation}
we get 
\begin{multline*}
c_2+v_2\left(\frac {(2n)!}
{2!^{n-2c_2}(n-2c_2)!\,4!^{c_2}c_2!}\right)
\ge c_2+\al-\be-1\ge 2^\be+\al-\be-1\\
\ge \be+2+\al-\be-1
=\al+1\ge 1+\max\limits_{n-2c_2<i\le n}v_2(i).
\end{multline*}

\medskip
{\sc Case 2: $n_1\ge 2c_2$}. 
Now, by inspection of \eqref{eq:n-c_2}, we see that 
\begin{equation*} 
\max\limits_{n-2c_2<i\le n}v_2(i)\le \be+1.
\end{equation*}
Hence, using again \eqref{eq:2^h}, we get
$$
c_2+v_2\left(\frac {(2n)!}
{2!^{n-2c_2}(n-2c_2)!\,4!^{c_2}c_2!}\right)
\ge c_2\ge 2^\be
\ge \be+2
\ge 1+\max\limits_{n-2c_2<i\le n}v_2(i).
$$

\medskip
Finally we address the case where $\be=1$, and thus $2c_2=6$. In this
case, the schematic representation \eqref{eq:n-c_2} becomes
\begin{alignat*}2
\notag
&&\underset{\downarrow}{\scriptstyle \al}\kern29pt
\underset{\downarrow}{\scriptstyle 3}\kern1pt
&\\
\notag
(n)_2&={}&\dots\,10\,\dots\,0&*\!**\\
(6)_2&=&&\kern3pt110
\end{alignat*}
Here also, we must distinguish two cases depending on whether or not 
$n_1$ is greater than~6.

If $n_1<6$, then \eqref{eq:max1} holds. By Lemma~\ref{lem:6}, we then get
$$
c_2+v_2\left(\frac {(2n)!}
{2!^{n-2c_2}(n-2c_2)!\,4!^{c_2}c_2!}\right)
\ge 3+(\al-2)
= \al+1
\ge 1+\max\limits_{n-2c_2<i\le n}v_2(i).
$$

On the other hand, if $n_1\ge6$ then
$$
\max\limits_{n-2c_2<i\le n}v_2(i)\le 2,
$$
and we obtain
\begin{equation*} 
c_2+v_2\left(\frac {(2n)!}
{2!^{n-2c_2}(n-2c_2)!\,4!^{c_2}c_2!}\right)
\ge 3
\ge 1+\max\limits_{n-2c_2<i\le n}v_2(i).
\end{equation*}

\medskip
This completes the proof of the lemma.
\end{proof}

\begin{lemma} \label{lem:10}
For all positive integers $n$ and $c_2$ with $n\ge2c_2+3,$ we have
\begin{equation} \label{eq:max3} 
c_2+2+v_2\left(\frac {(2n)!} {2!^{n-2c_2-3}(n-2c_2-3)!\,4!^{c_2}c_2!\,6!}\right)\ge
1+\max\limits_{n-2c_2-3<i\le n}v_2(i).
\end{equation}
\end{lemma}

\begin{proof}
By Legendre's formula in Lemma~\ref{lem:4}, we obtain
\begin{align*}
v_2&\left(\frac {(2n)!} {2!^{n-2c_2-3}(n-2c_2-3)!\,4!^{c_2}c_2!\,6!}\right)\\[1mm]
&\kern1.5cm
=
2n-s_2(2n)-(n-2c_2-3)-\big(n-2c_2-3-s_2(n-2c_2-3)\big)\\[1mm]
&\kern3cm
-3c_2-(c_2-s_2(c_2))-4\\[1mm]
&\kern1.5cm
=-s_2(n)+s_2(n-2c_2-3)+s_2(2c_2)+s_2(3)\\[1mm]
&\kern1.5cm
\ge-s_2(n)+s_2(n-2c_2-3)+s_2(2c_2+3)\\[1mm]
&\kern1.5cm
=\#(\text{carries when adding $(n-2c_2-3)_2$ and $(2c_2+3)_2$}).
\end{align*}

From here on we follow the setup and the arguments in the proof of
Lemma~\ref{lem:9}, with $2c_2+3$ in place of $2c_2$; see, in particular, 
the schematic representation in \eqref{eq:n-c_2} with that
replacement.
There is one notable difference, though: here the lower bound
on $c_2$ is zero (as opposed to~$3$ in the proof of
Lemma~\ref{lem:9}). This entails $2c_2+3\ge3$, and hence the lower bound
on the parameter~$\be$ is $\be\ge0$. Thus, instead of \eqref{eq:2^h}, 
we can only use $2^\be\ge \be+1$ here, a bound that is by $1$ weaker. 
This is balanced by the additional summand $+2$
on the left-hand side of \eqref{eq:max3}
when compared to the left-hand side of \eqref{eq:max2}.
\end{proof}

In order to have a convenient notation for the following
considerations, inspired by~\cite{ScheAA}, we write
$\mathcal P_{N,K}$ for the set of all tuples $(c_1,c_2,\dots,c_N)$
of non-negative integers with
\begin{align} \label{eq:ci1alle}
c_1+c_2+\dots+c_N&=K,\\[1mm]
c_1+2c_2+\dots+Nc_{N}&=N.
\label{eq:ci2alle}
\end{align}

\begin{lemma} \label{lem:11}
For all positive integers $n$ and tuples $(c_1,c_2,\dots,c_n)$ in
$$\mathcal P_{n,k}\backslash\{(n,0,\dots,0),(n-2,1,0,\dots,0),
(n-4,2,0,\dots,0)\},$$ 
we have
\begin{equation} \label{eq:max4} 
n-k+v_2\left(\frac {(2n)!} 
{2!^{c_1}c_1!\,4!^{c_2}c_2!\cdots (2n)!^{c_n}c_n!} \right)\ge
1+\max\limits_{c_1<i\le n}v_2(i).
\end{equation}
\end{lemma}

\begin{proof}
By Legendre's formula in Lemma~\ref{lem:4},
the $2$-adic valuation on the left-hand side of \eqref{eq:max4} equals
\begin{align*}
v_2\left(\frac {(2n)!} 
{2!^{c_1}c_1!\,4!^{c_2}c_2!\cdots (2n)!^{c_n}c_n!} \right)
&=2n-s_2(2n)-\sum_{i=1}^n \big(c_i(2i-s_2(2i))+c_i-s_2(c_i)\big)\\[1mm]
&=\sum_{i=1}^n c_is_2(i)
-\sum_{i=1}^n \big(c_i-s_2(c_i)\big)
-s_2(n).
\end{align*}

Define the quantity $G(n,k)$ by
$$
G(n,k):=n-k+\sum_{i=1}^n c_is_2(i)
-\sum_{i=1}^n \big(c_i-s_2(c_i)\big)
-s_2(n).
$$
We must show that $G(n,k)$ is at least as large as the right-hand
side of \eqref{eq:max4}. We achieve this by demonstrating that, among all
elements $(c_1,c_2,\dots,c_n)\in \mathcal P_{n,k}$ with fixed~$c_1$,
the ones with $c_i=0$ for $i\ge4$ and $c_3=0$ or $c_3=1$, depending on
the parity of $n-c_1$, attain the smallest values of $G(n,k)$, 
and then appealing to Lemmas~\ref{lem:9} and~\ref{lem:10}.

First consider a tuple $(c_1,c_2,\dots,c_n)$ in which $c_{2j}>0$ for 
some~$j\ge2$. We claim that the value of $G(n,k)$ decreases if 
instead we consider the tuple $(c_1,c_2+jc_{2j},\dots,0,\dots,c_n)$,
with $0$ appearing in position~$2j$. This new tuple is an element of
$\mathcal P_{n,k+(j-1)c_{2j}}$. In order to prove the claim, we
compare the contributions of the tuples to $G(n,k)$,
ignoring the terms that are the same for both tuples.
In particular, we may ignore appearances of~$n$ (such as in $s_2(n)$)
since $n$ is the same for both tuples (as opposed to~$k$, which
becomes $k+(j-1)c_{2j}$ for the modified tuple).
The --- in this sense ---
relevant contribution of $(c_1,c_2,\dots,c_n)$ to
$G(n,k)$ is
\begin{multline} \label{eq:contr1} 
    c_2s_2(2)+c_{2j}s_2(2j)-(c_2-s_2(c_2))-(c_{2j}-s_2(c_{2j}))\\[1mm]
=c_{2j}\big(s_2(j)-1\big)+s_2(c_2)+s_2(c_{2j}).
\end{multline}
On the other hand, 
the (relevant) contribution of $(c_1,c_2+jc_{2j},\dots,0,\dots,c_n)$ to\linebreak
$G(n,k+(j-1)c_{2j})$ is
\begin{equation} \label{eq:contr2} 
-(j-1)c_{2j}+(c_2+jc_{2j})s_2(2)-(c_2+jc_{2j}-s_2(c_2+jc_{2j}))
=-(j-1)c_{2j}+s_2(c_2+jc_{2j}).
\end{equation}
The difference of \eqref{eq:contr1} and \eqref{eq:contr2} is
\begin{align*}
c_{2j}\big(s_2(j)+j-2\big)+s_2(c_2)+s_2(c_{2j})
&-s_2(c_2+jc_{2j})\\[1mm]
&\ge
c_{2j}s_2(j)+(j-2)c_{2j}+s_2(c_{2j})
-s_2(c_{2j}j)\\[1mm]
&\ge
(j-2)c_{2j}+s_2(c_{2j})\ge s_2(c_{2j}),
\end{align*}
which is positive, thus proving our claim.

Now consider a tuple $(c_1,c_2,\dots,c_n)$ in which $c_{2j+1}>0$ for 
some~$j\ge2$. Again,\break we claim that the value of $G(n,k)$
(weakly) decreases if 
instead we consider the tuple\break 
$(c_1,c_2+(j-1)c_{2j+1},c_3+c_{2j+1},\dots,0,\dots,c_n)$,
with $0$ appearing in position~$2j+1$. This new tuple is an element of
$\mathcal P_{n,k+(j-1)c_{2j+1}}$. In order to prove the claim, we
compare the contributions of the tuples to $G(n,k)$,
ignoring the terms that are the same for both tuples. 
The (relevant) contribution of $(c_1,c_2,\dots,c_n)$ to
$G(n,k)$ is
\begin{multline} \label{eq:contr3} 
    c_2s_2(2)+c_3s_2(3)+c_{2j+1}s_2(2j+1)
-(c_2-s_2(c_2))-(c_3-s_2(c_3))-(c_{2j+1}-s_2(c_{2j+1}))\\[1mm]
=c_3+c_{2j+1}s_2(j)
+s_2(c_2)+s_2(c_3)+s_2(c_{2j+1}).\kern2cm
\end{multline}
On the other hand, 
the (relevant) contribution of
$(c_1,c_2+(j-1)c_{2j+1},c_3+c_{2j+1},\dots,0,\break\dots,c_n)$ to
$G(n,k+(j-1)c_{2j+1})$ is
\begin{multline} \label{eq:contr4} 
-(j-1)c_{2j+1}
+(c_2+(j-1)c_{2j+1})s_2(2)
+(c_3+c_{2j+1})s_2(3)\\[1mm]
-(c_2+(j-1)c_{2j+1}-s_2(c_2+(j-1)c_{2j+1}))
-(c_3+c_{2j+1}-s_2(c_3+c_{2j+1}))\\[1mm]
=
c_3-(j-2)c_{2j+1}+s_2(c_2+(j-1)c_{2j+1})
+s_2(c_3+c_{2j+1}).
\end{multline}
The difference of \eqref{eq:contr3} and \eqref{eq:contr4} is
\begin{align*}
c_{2j+1}\big(s_2(j)+j-2\big)&
+s_2(c_2)+s_2(c_3)+s_2(c_{2j+1})\\
&\kern2cm
-s_2(c_2+(j-1)c_{2j+1})
-s_2(c_3+c_{2j+1})\\[1mm]
&\ge c_{2j+1}\big(s_2(j)+j-2\big)
+s_2(c_2)
-s_2(c_2+(j-1)c_{2j+1})\\[1mm]
&\ge c_{2j+1}\big(s_2(j)-1\big)
+(j-1)c_{2j+1}
-s_2((j-1)c_{2j+1})\\[1mm]
&\ge c_{2j+1}\big(s_2(j)-1\big),
\end{align*}
which is non-negative, as claimed.

So far, the above arguments show that we may restrict our attention to
tuples $(c_1,c_2,\dots,c_n)$ in $\mathcal P_{n,k}$ with $c_i=0$ for $i\ge4$.
Finally, we argue that ``$3$'s can be traded for $2$'s", that is,
we may decrease $c_3$ at the cost of increasing~$c_2$.
There are two cases to be considered. First let $c_3$ be even,
$c_3=2c_3'$ say, with $c_3'\ge1$. We claim that the value of $G(n,k)$ decreases if 
instead of $(c_1,c_2,c_3,\dots,c_n)$
we consider the tuple $(c_1,c_2+3c_3',0,\dots,c_n)$,
This new tuple is an element of
$\mathcal P_{n,k+c_3'}$. In order to prove the claim, we
compare the contributions of the tuples to $G(n,k)$,
ignoring the terms that are the same for both tuples. 
The (relevant) contribution of $(c_1,c_2,c_3,\dots,c_n)$ to
$G(n,k)$ is
\begin{equation} \label{eq:contr5} 
c_2s_2(2)+c_3s_2(3)-(c_2-s_2(c_2))-(c_3-s_2(c_3))
=c_3+s_2(c_2)+s_2(c_3').
\end{equation}
On the other hand, 
the (relevant) contribution of $(c_1,c_2+3c_3',0,\dots,c_n)$ to
$G(n,k+c_3')$ is
\begin{equation} \label{eq:contr6} 
-c_3'+(c_2+3c_3')s_2(2)-(c_2+3c_3'-s_2(c_2+3c_3'))
=-c_3'+s_2(c_2+3c_3').
\end{equation}
The difference of \eqref{eq:contr5} and \eqref{eq:contr6} is
\begin{align}
\notag
3c_3'+s_2(c_2)+s_2(c_3')
-s_2(c_2+3c_3')
&\ge
3c_3'+s_2(c_2+c_3')
-s_2(c_2+3c_3')\\[1mm]
&\ge
3c_3'-s_2(2c_3')\ge c_3',
\label{eq:diff1}
\end{align}
which is positive, proving our claim.

Finally, let $c_3$ be odd,
$c_3=2c_3'+1$ say, with $c_3'\ge1$. We claim that the value of $G(n,k)$ decreases if 
instead of $(c_1,c_2,c_3,\dots,c_n)$
we consider the tuple $(c_1,c_2+3c_3',1,\dots,c_n)$,
This new tuple is an element of
$\mathcal P_{n,k+c_3'}$. In order to prove the claim, we
compare the contributions of the tuples to $G(n,k)$,
ignoring the terms that are the same for both tuples. 
The (relevant) contribution of $(c_1,c_2,c_3,\dots,c_n)$ to
$G(n,k)$ is 
\begin{equation} \label{eq:contr7} 
c_2s_2(2)+c_3s_2(3)-(c_2-s_2(c_2))-(c_3-s_2(c_3))
=c_3+s_2(c_2)+s_2(c_3')+1.
\end{equation}
On the other hand, 
the (relevant) contribution of $(c_1,c_2+3c_3',1,\dots,c_n)$ to
$G(n,k+c_3')$ is
\begin{multline} \label{eq:contr8} 
-c_3'+(c_2+3c_3')s_2(2)+s_2(3)
-(c_2+3c_3'-s_2(c_2+3c_3'))\\[1mm]
=-c_3'+s_2(c_2+3c_3')+2.
\end{multline}
The difference of \eqref{eq:contr7} and \eqref{eq:contr8} is
the same as \eqref{eq:diff1}, of which we already know that it  is positive.

\medskip
We are now in the position to finish the proof of the lemma.
Our tuple $(c_1,c_2,\dots,c_n)$ may be one where $c_i=0$ for $i\ge3$.
Then $c_1=n-2c_2$, $n-k=(c_1+2c_2)-(c_1+c_2)=c_2$, 
and $c_2\ge3$ by assumption. The assertion of the lemma then
follows from Lemma~\ref{lem:9}.
On the other hand, 
our tuple $(c_1,c_2,\dots,c_n)$ may be one where $c_3=1$ and $c_i=0$ for $i\ge4$.
Then $c_1=n-2c_2-3$, and $n-k=(c_1+2c_2+3)-(c_1+c_2+1)=c_2+2$.
The assertion of the lemma then
follows from Lemma~\ref{lem:10}.
If we are not in one of these two cases, then we apply the above
described reductions repeatedly until we arrive at a tuple which
belongs to one of these two cases. Since the value of $G(n,k)$ never
increases while the value of $c_1$
remains invariant when doing these reductions, Lemmas~\ref{lem:9}
and~\ref{lem:10} again establish the assertion of the lemma.

\medskip
The proof is now complete.
\end{proof}

Before we are able to state and prove the main result of this
section,  we need an auxiliary integrality assertion on a
certain product/quotient of factorials.

\begin{lemma}[{\cite[Theorem 13.2]{AndrAF}}]
\label{lem:3}
Let $N$ and $K$ be positive integers such that $N\ge K$.
For all tuples $(c_i)_{1\le i\le N}$ of integers
in $\mathcal P_{N,K}$ --- that is, satisfying~\eqref{eq:ci1}
and~\eqref{eq:ci2} --- the quantity
$$
 \frac {N!} {
\prod _{i=1} ^{N}i!^{c_i}c_i!}
$$
is an integer, and it equals the number of partitions of the set
$\{1,2,\dots,N\}$ into $c_i$ blocks of size~$i,$ for $i=1,2,\dots,N$.
\end{lemma}

We are now in the position to embark on the proof of periodicity of
$v(n)$ as defined in~\eqref{eq:v-def} modulo powers of~2. The next
theorem treats the ``generic" case where $e\ge3$, whereas
Theorem~\ref{prop:2a} handles the remaining cases where $e=1$ or $e=2$.
Both theorems provide in fact polynomial refinements.

\begin{theorem} \label{prop:2}
Let $x(j),$ $j=0,1,2,\dots,$ be a sequence of integers with $x(0)=1,$ $x(1)$ and
$x(2)$ odd. Then
the coefficients $v_{\mathbf x}(n)$ in the expansion
\begin{equation} \label{eq:vxdef} 
\sum_{n\ge0}\frac {v_{\mathbf x}(n)} {2^n\,(2n)!}t^n
=\Bigg(1+\sum_{j\ge1}\frac {x(j)} {(2j)!}t^{2j}\Bigg)^{1/2}
\end{equation}
are integers. Moreover, for all integers $e\ge3,$
the sequence $(v_{\mathbf x}(n))_{n\ge0}$ is purely
periodic modulo $2^e$ with {\em(}not necessarily minimal\/{\em)}
period length~$2^{e-1}$.
\end{theorem}

\begin{remark}
Computer experiments suggest that, in fact, the period length
of~$2^{e-1}$ {\it is} minimal. 
\end{remark}

\begin{proof}[Proof of Theorem~{\em \ref{prop:2}}]
We compute the defining expansion for the coefficients $v_{\mathbf x}(n)$:
\begin{align*}
\sum_{n\ge0}\frac {v_{\mathbf x}(n)} {2^n(2n)!}t^n
&=\Bigg(1+\sum_{j\ge1}\frac {x(j)} {(2j)!}t^j\Bigg) ^{1/2}\\[1.5mm]
&=\sum_{k\ge0}\binom {1/2}k\Bigg(\sum_{j\ge1}\frac {x(j)}
{(2j)!}t^j\Bigg)^k\\[1.5mm]
&=1+\sum_{k\ge1} (-1)^{k-1}\frac {
\prod _{j=1} ^{k-1}(2j-1)} {2^k\,k!}
{\sum_{c_1+c_2+\dots=k}}\frac {k!} {c_1!\,c_2!\cdots } 
\prod _{i\ge1} ^{}\frac {x^{c_i}(i)} {(2i)!^{c_i}}t^{ic_i}.
\end{align*}
By comparing coefficients of~$t^n$, we obtain
\begin{equation} 
v_{\mathbf x}(n)
=\sum_{k\ge1} 
\sum_{(c_i)\in\mathcal P_{n,k}}
(-1)^{k-1}\Bigg(\prod _{j=1} ^{k-1}(2j-1)\Bigg)
2^{n-k}\frac {(2n)!} 
{2!^{c_1}c_1!\,4!^{c_2}c_2!\cdots (2n)!^{c_n}c_n!} 
\prod _{i=1} ^{n}{x^{c_i}(i)} ,
\label{eq:v(n)Expr}
\end{equation}
where $\mathcal P_{n,k}$ is defined around
 \eqref{eq:ci1alle}--\eqref{eq:ci2alle}.
For convenience, we denote the summand in the above double sum by
$S(n,k,(c_i))$, that is, we set 
\begin{equation} \label{eq:Sdef} 
S(n,k,(c_i)):=
(-1)^{k-1}\Bigg(\prod _{j=1} ^{k-1}(2j-1)\Bigg)
2^{n-k}\frac {(2n)!} 
{2!^{c_1}c_1!\,4!^{c_2}c_2!\cdots (2n)!^{c_n}c_n!} 
\prod _{i=1} ^{n}{x^{c_i}(i)} 
\end{equation}
for some positive integer $k$ and $(c_i)\in\mathcal P_{n,k}$.
The fraction in the expression in \eqref{eq:v(n)Expr} (and in~\eqref{eq:Sdef}) is an integer 
since, by Lemma~\ref{lem:3}, it is
the number of all partitions of~$\{1,2,\dots,2n\}$ 
into $c_i$ blocks of size $2i$, $i=1,2,\dots,n$.
If this property is used in \eqref{eq:v(n)Expr}, then this establishes
the integrality assertion of the theorem.

\medskip
Now, let $n<2^{e-1}$. We want to prove that
\begin{equation} \label{eq:period} 
v_{\mathbf  x}(n+a2^{e-1})\equiv v_{\mathbf  x}(n)\pmod{2^e},
\end{equation}
for all positive integers~$a$. Using our short notation for the
summand in \eqref{eq:v(n)Expr}, we have
\begin{equation} \label{eq:v-S} 
v_{\mathbf x}\big(n+a2^{e-1}\big)=\sum_{k\ge1}
\sum_{(\tilde c_i)\in\mathcal P_{n+a2^{e-1},k}}
S(n+a2^{e-1},k,(\tilde c_i)).
\end{equation}
We are going to prove the finer congruences
\begin{align} 
\notag
S(n+a2^{e-1},k,(\tilde c_i))&\equiv0\pmod{2^e},\quad \text{if }\tilde c_1<a2^{e-1},\\
\label{eq:cong1}
&\kern-1cm \text{$(\tilde c_i)$ not in cases \eqref{eq:exc1}--\eqref{eq:exc3}},\\[1.5mm]
\label{eq:cong2}
S(n+a2^{e-1},n+a2^{e-1},(n+a2^{e-1},0,\dots,0))&\equiv
S(n,n,(n,0,\dots,0))\pmod{2^e},\\[1.5mm]
\notag
S(n+a2^{e-1},k+a2^{e-1},(\tilde c_i))&\equiv 
S(n,k,(c_1,\tilde c_2,\dots,\tilde c_n))
\pmod{2^e}\\
\label{eq:cong3}
&\kern-4.9cm \text{with $\tilde c_1=c_1+a2^{e-1}$, and
  $(\tilde c_i)$ not in cases \eqref{eq:exc1}--\eqref{eq:exc3}},\\[1.5mm]
\notag
S(n+a2^{e-1},n+a2^{e-1}-1,(n+a2^{e-1}-2,1,0,\dots,0))\kern-1cm&\\
\notag
+S(n+a2^{e-1},n
+a2^{e-1}-2,(n+a2^{e-1}-4,2,0,\dots,0))\kern-3cm&\\[1.5mm]
\notag
&\kern-5cm \equiv 
S(n,n-1,(n-2,1,0,\dots,0))\\
&\kern-3cm
+S(n,n-2,(n-4,2,0,\dots,0))\pmod{2^e}.
\label{eq:cong4}
\end{align}
In the last congruence, undefined terms (terms $S(n,k,(\tilde c_i))$
with negative $\tilde c_1$) have to be understood as being zero.

We claim that \eqref{eq:cong1}--\eqref{eq:cong4} together
imply~\eqref{eq:period}, and thus the theorem.  
In order to prove this, we consider the summand
$S(n+a2^{e-1},k,(\tilde c_i))$ on the right-hand side
of~\eqref{eq:v-S}, when taken modulo~$2^e$,
for the various choices of~$k$ and tuples~$(\tilde c_i)$.
If $\tilde c_1<a2^{e-1}$, then,
by~\eqref{eq:cong1}, the corresponding summand vanishes modulo~$2^e$.
On the other hand, if $\tilde c_1\ge a2^{e-1}$, then, since $k\ge \tilde
c_1$, we must also have $k\ge a2^{e-1}$. Consequently, we may
apply~\eqref{eq:cong2} and~\eqref{eq:cong3} to conclude that 
\begin{equation} \label{eq:S-red} 
S(n+a2^{e-1},k,(\tilde c_i))\equiv S(n,k-a2^{e-1},(c_1,\tilde
c_2,\dots,\tilde c_n))\pmod{2^e},
\end{equation}
except if we are in cases \eqref{eq:exc2} or~\eqref{eq:exc3}.
However, for those cases the congruence~\eqref{eq:cong4} applies. It
shows that, even though the congruence~\eqref{eq:S-red} may not hold for
an individual summand in case~\eqref{eq:exc2} or~\eqref{eq:exc3}, if
they are combined then the corresponding reduction in~\eqref{eq:S-red}
is allowed. Altogether, we arrive at
\begin{align*}
v_{\mathbf x}\big(n+a2^{e-1}\big)&=
\sum_{k\ge1}
\sum_{(\tilde c_i)\in\mathcal P_{n+a2^{e-1},k}}
S(n+a2^{e-1},k,(\tilde c_i))\\[1mm]
&\equiv
\sum_{k\ge a2^{e-1}+1}
\sum_{(c_i)\in\mathcal P_{n,k-a2^{e-1}}}
S(n+a2^{e-1},k-a2^{e-1},(c_i))\pmod{2^e}\\[1mm]
&\equiv v_{\mathbf x}(n)\pmod{2^e},
\end{align*}
thus confirming the claim.

\medskip
Now we provide the proofs of the crucial congruences
\eqref{eq:cong1}--\eqref{eq:cong4}.

\medskip
{\sc Proof of \eqref{eq:cong1}.}
If we use Lemma~\ref{lem:11} with $n$ replaced by $n+a2^{e-1}$ and
$c_i=\tilde c_i$ for all~$i$, then
we see that $S(n+ap^{e-1},k,(\tilde c_i))$ is divisible by 
$$
2^{1+\max\limits_{\tilde c_1< i\le n+a2^{e-1}}v_2(i)},
$$
with three exceptions, namely
\begin{align} 
\label{eq:exc1}
(\tilde c_1,\tilde c_2,\dots,\tilde c_n)&=(n+a2^{e-1},0,\dots,0),\\[1mm]
\label{eq:exc2}
(\tilde c_1,\tilde c_2,\dots,\tilde c_n)&=(n+a2^{e-1}-2,1,0,\dots,0),\\[1mm]
\label{eq:exc3}
(\tilde c_1,\tilde c_2,\dots,\tilde c_n)&=(n+a2^{e-1}-4,2,0,\dots,0).
\end{align}
In particular, if $\tilde c_1<a2^{e-1}$, and if
\eqref{eq:exc1}--\eqref{eq:exc3} do not apply, then in the range\linebreak 
$\tilde c_1<i\le n+a2^{e-1}$ we will find that $i=a2^{e-1}$, and,
consequently, $S(n+a2^{e-1},k,(\tilde c_i))$ 
is divisible by~$2^e$, as desired.

\medskip
{\sc Proof of \eqref{eq:cong2}.}
By definition, we have
\begin{align*}
S(n+a2^{e-1},n+a&2^{e-1},(n+a2^{e-1},0,\dots,0))
\\[1mm]
&=
(-1)^{n+a2^{e-1}-1}\bigg(
\prod _{j=1} ^{n+a2^{e-1}-1}(2j-1)\bigg)(2n+a2^e-1)!!\,x^{n+a2^{e-1}}(1)\\[1mm]
&=
(-1)^{n-1}
(2n+a2^e-3)!!\,
(2n+a2^e-1)!!\,x^{n+a2^{e-1}}(1)
\end{align*}
and
$$
S(n,n,(n,0,\dots,0))=
(-1)^{n-1}(2n-3)!!\,(2n-1)!!\,x^{n}(1).
$$
By assumption, the number $x(1)$ is odd. Hence, Euler's theorem and the fact that 
$\varphi(2^e)=2^{e-1}$ imply that
$$
x^{2^{e-1}}(1)\equiv 1\pmod{2^e}.
$$
Furthermore, we have
\begin{align} \notag
(2n+a2^e-1)!!&=(2n-1)!!
\prod _{j=0} ^{a-1}(2n+j2^{e}+1)(2n+j2^e+3)\cdots(2n+(j+1)2^{e}-1)\\[1mm]
&\equiv(2n-1)!!\pmod{2^e},
\label{eq:(-1)^a}
\end{align}
since, for each~$j$, the factors inside the product in the first line
form a complete set of
representatives of the multiplicative group $(\Z/2^e\Z)^\times$. 
An analogous relation holds for $(2n+a2^e-3)!!$ and $(2n-3)!!$.
Altogether, this establishes the congruence~\eqref{eq:cong2}.

\medskip
{\sc Proof of \eqref{eq:cong3}.}
Replace $n$ by $n+a2^{e-1}$, and $c_1$ by $c_1+a2^{e-1}$ in \eqref{eq:Sdef}.
Since $(c_1+a2^{e-1},c_2,\dots)\in P_{n+a2^{e-1},k}$, 
the $c_i$'s must vanish for $i>n$.
Furthermore, we have $k\ge c_1+a2^{e-1}$, which allows us to
write $k=k'+a2^{2-1}$ for some non-negative
integer~$k'$. Hence,
under these substitutions and simplifications, and writing $k$ instead
of~$k'$ in abuse of notation, we obtain
\begin{multline}
S(n+a2^{e-1},k+a2^{e-1},(\tilde c_i))=
(-1)^{k+a2^{e-1}-1}
\Bigg(\prod _{j=1} ^{k+a2^{e-1}-1}(2j-1)\Bigg)\\[1mm]
\times
2^{n-k}\frac {(2n+a2^e)!} 
{2!^{c_1+a2^{e-1}}(c_1+a2^{e-1})!\,4!^{c_2}c_2!\cdots n!^{c_n}c_n!} 
x^{a2^{e-1}}(1)\prod _{i=1} ^{n}{x^{c_i}(i)} ,
\label{eq:v(n)Expr2}
\end{multline}
with $(c_i)\in\mathcal P_{n,k}$.
We are interested in the residue class of this expression modulo
$2^e$. As in~\eqref{eq:(-1)^a}, elementary number theory tells us that
$$
\prod _{j=1} ^{k+a2^{e-1}-1}(2j-1)\equiv
\prod _{j=1} ^{k-1}(2j-1)\quad (\text{mod }2^e).
$$
Furthermore, we know that $x(1)$ is odd by assumption. Therefore the expression
\eqref{eq:v(n)Expr2} simplifies to
\begin{multline}
S(n+a2^{e-1},k+a2^{e-1},(\tilde c_i))\equiv
(-1)^{k-1}
\Bigg(\prod _{j=1} ^{k-1}(2j-1)\Bigg)\\[1mm]
\times
2^{n-k}\frac {(2n+a2^e)!\,c_1!} 
{(2n)!\,2^{a2^{e-1}}(c_1+a2^{e-1})!}
\frac {(2n)!} 
{2!^{c_1}c_1!\,4!^{c_2}c_2!\cdots n!^ {c_n}c_n!} 
\prod _{i=1} ^{n}{x^{c_i}(i)} 
\quad (\text{mod }2^e).
\label{eq:v(n)Expr3}
\end{multline}
We have
$$
\frac {(2n+a2^e)!\,c_1!} 
{(2n)!\,2^{a2^{e-1}}(c_1+a2^{e-1})!}
=
\frac {(n+a2^{e-1})!\,(2n+a2^e-1)!!\,c_1!} 
{n!\,(2n-1)!!\,(c_1+a2^{e-1})!}.
$$
Since, again by elementary number theory, 
$$
\frac {(2n+a2^e-1)!!} {(2n-1)!!}\equiv1\quad (\text{mod }2^e),
$$
we may conclude that 
\begin{align*}
\frac {(2n+a2^e)!\,c_1!} 
{(2n)!\,2^{a2^{e-1}}(c_1+a2^{e-1})!}
&\equiv
\frac {(n+a2^{e-1})!\,c_1!} 
{n!\,(c_1+a2^{e-1})!}
\quad (\text{mod }2^e)\\[1mm]
&\equiv
\prod _{i=c_1+1} ^{n}\frac {i+a2^{e-1}} {i}
\quad (\text{mod }2^e)\\[1mm]
&\equiv
\prod _{i=c_1+1} ^{n}\frac
      {(i\cdot2^{-v_2(i)})+a2^{e-v_2(i)-1}} 
{i\cdot2^{-v_2(i)}}
\quad (\text{mod }2^e).
\end{align*}
If this is substituted back in \eqref{eq:v(n)Expr3}, then we see that
\begin{multline}
  S(n+a2^{e-1},k+a2^{e-1},(\tilde c_i))\\[1mm]
  \equiv
(-1)^{k-1}
\Bigg(\prod _{j=1} ^{k-1}(2j-1)\Bigg)
\Bigg(\prod _{i=c_1+1} ^{n}\frac
      {(i\cdot2^{-v_2(i)})+a2^{e-v_2(i)-1}} 
{i\cdot2^{-v_2(i)}}\Bigg)\\[1mm]
\cdot
2^{n-k}
\frac {(2n)!} 
{2!^{c_1}c_1!\,4!^{c_2}c_2!\cdots n!^ {c_n}c_n!} 
\prod _{i=1} ^{n}{x^{c_i}(i)} 
\quad (\text{mod }2^e).
\label{eq:v(n)Expr4}
\end{multline}
The reader should note that we wrote the first product over~$i$ in
this particular form in order to make certain that the expressions
$i\cdot 2^{-v_2(i)}$ in the denominator are odd numbers. 

At this point, we would like to simplify the terms
$(i\cdot2^{-v_2(i)})+a2^{e-v_2(i)-1}$ to\linebreak $i\cdot2^{-v_2(i)}$.
For, assuming the validity of this simplification, the first product
over~$i$ on the right-hand side of~\eqref{eq:v(n)Expr4} would simplify to~1,
and the remaining terms exactly equal $S\big(n,k,(c_i)\big)$.

We claim that this simplification is indeed allowed. For, by Lemma~\ref{lem:11}, we
know that the second line in \eqref{eq:v(n)Expr4} is an integer which 
is divisible by
$$
2^{1+\max\limits_{c_1< i\le n}v_2(i)}.
$$
Hence, instead of computing modulo~$2^e$, we may
reduce the first line of \eqref{eq:v(n)Expr4} modulo
$$
2^{e-1-\max\limits_{c_1< i\le n}v_2(i)}.
$$
(It should be observed that the exponent is non-negative since we
assumed from\linebreak the beginning that $n< 2^{e-1}$.)
This is exactly what we need to perform the desired\linebreak simplification, and
the first product over~$i$ drops out. What remains is exactly\linebreak 
$S(n,k,(c_1,\tilde c_2,\dots,\tilde c_n))$.

\medskip
{\sc Proof of \eqref{eq:cong4}.}
Here, instead of Lemma~\ref{lem:11}, we can only use the slightly
weaker Corollaries~\ref{lem:7} and~\ref{lem:8}.

We discuss first the ``generic case", namely when $n\ge4$, where all
terms in \eqref{eq:cong4} are defined. 
Let us begin with the term 
$$S(n+a2^{e-1},n+a2^{e-1}-1,(n+a2^{e-1}-2,1,0,\dots,0)).$$
The corresponding expression on the right-hand side of
\eqref{eq:v(n)Expr4} contains the product
\begin{multline}
\prod _{i=(n-2)+1} ^{n}\frac
      {(i\cdot2^{-v_2(i)})+a2^{e-v_2(i)-1}} 
{i\cdot2^{-v_2(i)}}\\[1mm]
=\frac
      {\big(((n-1)\cdot2^{-v_2(n-1)})+a2^{e-v_2(n-1)-1}\big)
\big((n\cdot2^{-v_2(n)})+a2^{e-v_2(n)-1}\big)} 
{(n-1)\cdot2^{-v_2(n-1)}\,n\cdot2^{-v_2(n)}}.
\label{eq:exc-even}
\end{multline}
If $n\equiv0$~(mod~$4$), then Corollary~\ref{lem:7} says that the last line in
\eqref{eq:v(n)Expr4} (which in particular contains the prefactor
$2^{n-k}$) is (exactly) divisible by $2\cdot 2^{v_2(n)-1}=2^{v_2(n)}$, and hence we may
consider the numerator in \eqref{eq:exc-even} modulo~$2^{e-v_2(n)}$.
Since
$$
((n-1)\cdot2^{-v_2(n-1)})+a2^{e-v_2(n-1)-1}
\equiv
((n-1)\cdot2^{-v_2(n-1)})\pmod{2^{e-1}}
$$
(with $v_2(n-1)=0$),
the product \eqref{eq:exc-even} reduces to
\begin{align*}
\prod _{i=n-1} ^{n}\frac
      {(i\cdot2^{-v_2(i)})+a2^{e-v_2(i)-1}} 
{i\cdot2^{-v_2(i)}}
&\equiv\frac
      {(n\cdot2^{-v_2(n)})+a2^{e-v_2(n)-1}} 
{n\cdot2^{-v_2(n)}}
\pmod{2^e}\\[1mm]
&\equiv 1+
\frac
      {a2^{e-v_2(n)-1}} 
{n\cdot2^{-v_2(n)}}
\pmod{2^e}.
\end{align*}
When substituted back in \eqref{eq:v(n)Expr4}, this shows that
\begin{multline*}
S(n+a2^{e-1},n+a2^{e-1}-1,(n+a2^{e-1}-2,1,0,\dots,0))\\[1mm]
\equiv
S(n,n-1,(n-2,1,0,\dots,0))+u_1\frac
      {a2^{e-v_2(n)-1}} 
{n\cdot2^{-v_2(n)}}
\pmod{2^e},
\end{multline*}
where $u_1$ is an odd integer. Here it is important that $x(1)$ and~$x(2)$
are odd, which they are by assumption.

An analogous discussion for the other term on the left-hand side of
\eqref{eq:cong4}, using Corollary~\ref{lem:8}, leads to
\begin{multline*}
S(n+a2^{e-1},n+a2^{e-1}-2,(n+a2^{e-1}-4,2,0,\dots,0))\\[1mm]
\equiv
S(n,n-2,(n-4,2,0,\dots,0))+u_2
\frac
      {a2^{e-v_2(n)-1}} 
{n\cdot2^{-v_2(n)}}
\pmod{2^e},
\end{multline*}
where $u_2$ is an odd integer.
Together, the last two congruences establish \eqref{eq:cong4}.

The discussions for the other congruence classes of $n$ modulo~$4$ are
just minor variations of the above arguments and are therefore
omitted here.

Finally, we address the remaining cases of ``small" $n$, namely
$0\le n<4$, when one or
both terms on the right-hand side of \eqref{eq:cong4} contain negative
parameter values, in which case they are declared to be zero by
definition. We exemplify what has to be done in these cases by going
through the case where $n=1$. All other cases may be handled in a similar manner.

Corollary~\ref{lem:7} with $n$ replaced by $1+a2^{e-1}$, together with
the assumption that all $x(i)$'s are odd, implies that
$$
S(n+a2^{e-1},n+a2^{e-1}-1,(n+a2^{e-1}-2,1,0,\dots,0))\equiv
\begin{cases} 2^{e-1}~(\text{mod }2^e),&\text{if $a$ is odd},\\[1mm]
0~(\text{mod }2^e),&\text{if $a$ is even}.
\end{cases}
$$
Likewise,
Corollary~\ref{lem:8} with $n$ replaced by $1+a2^{e-1}$, together with
the assumption that all $x(i)$'s are odd, implies that
$$
S(n+a2^{e-1},n+a2^{e-1}-2,(n+a2^{e-1}-4,2,0,\dots,0))\equiv
\begin{cases} 2^{e-1}~(\text{mod }2^e),&\text{if $a$ is odd},\\[1mm]
0~(\text{mod }2^e),&\text{if $a$ is even}.
\end{cases}
$$
Together, this establishes \eqref{eq:cong4} in the considered case.

\medskip
This finishes the proof of the theorem.
\end{proof}

Our next result generalises \cite[Theorem~1(i)]{ScheAA}.

\begin{theorem} \label{prop:2a}
Let $x(j),$ $j=0,1,2,\dots,$ be integers with $x(0)=1$ and $x(1)$ 
odd. Then
the sequence $(v_{\mathbf x}(n))_{n\ge0},$ with the coefficients
$v_{\mathbf x}(n)$ being defined in \eqref{eq:vxdef}{\em ,} when taken
modulo~$2,$ is the all-$1$-sequence. Modulo~$4,$ the sequence is 
purely periodic with period length~$4,$ the first few
values of the sequence {\em(}modulo~$4${\em)} being given by 
\begin{equation} \label{eq:0-4} 
1,\, x(1),\, 3,\,  x(1)+2,\, 1 ,\, \dots.
\end{equation}
\end{theorem}

\begin{proof}
It suffices to treat the case of modulus~$4$ since, because
of the assumption that $x(1)$ is odd, it implies
the assertion for the modulus~$2$.

We prove the claim in the statement of the theorem by induction on~$n$.

We use the recurrence \eqref{eq:Rek2} for the induction step.
We may assume that $n\ge4$.
We rewrite the recurrence \eqref{eq:Rek2} modulo~$4$ in the form
$$
0\equiv-\sum_{m=0}^{\cl{n/2}-1}\binom {2n}{2m}
v(m)v(n-m)-\frac {1} {2}\chi(n\text{ even})\binom {2n}n v^2(n/2)
\quad (\text{mod }4),
$$
where $\chi(\mathcal A)=1$ if $\mathcal A$ is true and
$\chi(\mathcal A)=0$ otherwise. Since the central binomial coefficient
is divisible by~$2$ for $n\ge1$,
we see that, for the induction, we may simply substitute the
claimed mod-4 values of $v(m)$ for $m=0,1,\dots,n$ in
\begin{equation} \label{eq:Rek2a}
\frac {1} {2}\sum_{m=0}^{n}\binom {2n}{2m}
v(m)v(n-m),
\end{equation}
and check that the resulting expression is divisible by~$4$.

We substitute the values from \eqref{eq:0-4}
in the expression \eqref{eq:Rek2a}, to get
\begin{multline}
\frac {1} {2}\sum_{m\equiv0~(\text{mod }4)}^{}\binom {2n}{2m}
+\frac {x(1)} {2}\sum_{m\equiv1~(\text{mod }4)}^{}\binom {2n}{2m}\\[1mm]
+\frac {3} {2}\sum_{m\equiv2~(\text{mod }4)}^{}\binom {2n}{2m}
+\frac {x(1)+2} {2}\sum_{m\equiv3~(\text{mod }4)}^{}\binom {2n}{2m}.
\label{eq:binsum}
\end{multline}
Next we recall the well-known (and easily derived) fact that
$$
\sum_{m\equiv r~\text{(mod 8)}}\binom {N}m=
\frac {1} {8}\sum_{j=0}^7\om^{-rj}(1+\om^j)^ {N},
$$
where $\om=\frac {1+\mathbf i} {\sqrt2}$ is a primitive eighth root of
unity (with $\mathbf i=\sqrt{-1}$).
Using this in \eqref{eq:binsum}, we obtain
\begin{multline*}
\frac {1} {16}\sum_{j=0}^7(1+\om^j)^ {2n}
+\frac {x(1)} {16}\sum_{j=0}^7(-\mathbf i)^{j}(1+\om^j)^ {2n}\\[1mm]
+\frac {3} {16}\sum_{j=0}^7(-1)^{j}(1+\om^j)^ {2n}
+\frac {x(1)+2} {16}\sum_{j=0}^7\mathbf i^{j}(1+\om^j)^ {2n}.
\end{multline*}
Now we multiply this expression by $z^n$ and sum the result over all
$n\ge0$. After evaluating all the appearing geometric series and
simplifying, we obtain
\begin{equation} \label{eq:GF} 
\frac{(x(1)+3) \left(
   1 - 7 z + 21 z^2 - 35 z^3 + 34 z^4 - 42 z^5 - 28 z^6 - 8 z^7
\right)}{1 - 8 z + 28 z^2 - 56 z^3 + 68 z^4 - 112 z^5 - 112 z^6 - 64 z^7}.
\end{equation}
Here, we should first observe that the factor $x(1)+3$ is even,
since $x(1)$ is odd by assumption.
Furthermore, 
the denominator has the form $1+2f(z)$, where $f(z)$ is a
polynomial with integer coefficients, and the coefficients of
$z^4$, $z^5$, $z^6$, and $z^7$ in the numerator are all even. It is
not difficult to see that, together, this implies that in the series
expansion of \eqref{eq:GF} the coefficients of $z^n$ is divisible
by~$4$ for $n\ge4$.
\end{proof}

\section{The inverse of the matrix $\mathbf R$ modulo powers of $2$}
\label{sec:7}

In this section, we prove periodicity of the matrix entries
$R^{-1}(k+2n,k)$ (given by Proposition~\ref{prop:1})
modulo powers of~2, when considered as a sequence
indexed by~$k$. Moreover, this result comes again with a precise
statement on period lengths.
Our starting point is a modified version of the
expansion~\eqref{eq:R^(-1)p},
namely~\eqref{eq:R^(-1)}. Also here, the main result of this
section concerns a polynomial refinement. More precisely,
in this refinement, the number~$u(j)$
gets replaced by a variable~$y(j)$, $j=0,1,\dots$, with
the only restriction that $y(0)=1$.

\begin{theorem} \label{thm:4}
Let $y(j),$ $j=0,1,2,\dots,$ be integers with $y(0)=1$. 
Furthermore, define coefficients $R^{-1}_y(n,k)$ by
\begin{equation} \label{eq:Ry^{-1}} 
R^{-1}_y(n,k)=2^{(n-k)/2}\frac {(n-1)!} {(k-1)!}\coef{t^{-k}}U^{-n}_y(t),
\end{equation}
where
$$
U_y(t):=\sum_{j\ge0}\frac {y(j)} {(2j+1)!}t^{2j+1}.
$$
Here again,
the case $k=0$ has to be interpreted as $R^{-1}_y(0,0)=1$ and
$R^{-1}_y(n,0)=0$ for $n\ge1$.
Then, for fixed~$n$ and any $2$-power $2^e,$ 
the sequence $\big(R^{-1}_y(k+2n,k)\big)_{k\ge0}$ is purely periodic modulo~$2^e$
with {\em(}not necessarily minimal\/{\em)} period length~$2^{e}$.
\end{theorem}

\begin{proof}
By the definition \eqref{eq:Ry^{-1}} and the computation
in~\eqref{eq:R^(-1)p} with $u$~replaced by $y$, $n$~replaced by~$n+k$,
and $k$ replaced by $k/2$, in this order, we obtain
\begin{multline}
R^{-1}_y(2n+{}k,k)=
\sum_{m\ge0}
\underset{c_1=0}{\sum_{(c_i)\in\mathcal P^o_{2n+m,m}}}
(-1)^m2^{n}\binom {2n+m+k-1} {2n+m}\\
\cdot
\frac {(2n+m)!} {3!^{c_3}c_3!\,5!^{c_5}c_5!\cdots (2n+1)!^{c_{2n+1}}c_{2n+1}!}
\prod _{i=1} ^{2n+1}
{y^{c_i}\big(\tfrac {i-1} {2}\big)},
\label{eq:R^(-1)}
\end{multline}
with $\mathcal P^o_{N,K}$ having been defined below~\eqref{eq:R^(-1)p}.
Here, we observe that $\mathcal P^o_{2n+m,m}$ is empty for $m>n$.
Indeed, for $(c_i)\in \mathcal P^o_{2n+m,m}$ with $c_1=0$, we have
\begin{equation} \label{eq:m<n} 
3m=3\sum_{i=1}^{2n+1}c_i\le \sum_{i=1}^{2n+1}ic_i=2n+m,
\end{equation}
and hence $m\le n$. Consequently, we may restrict the sum over~$m$ in
\eqref{eq:R^(-1)} to $m=0,1,\dots,n$.

We want to consider $R^{-1}_y
(2n+k,k)$ for fixed~$n$ as a function
in~$k$. Inspection of the last expression in \eqref{eq:R^(-1)} reveals
that, since the sum over~$m$ is finite, $R^{-1}_y(2n+k,k)$ is a
polynomial in~$k$ with rational coefficients. Hence, it is certainly
periodic modulo~$2^e$ (and, more generally, modulo any modulus~$M$).
In order to find a bound on the period length, we must perform a
$2$-adic analysis of the expression in \eqref{eq:R^(-1)}.

We may rewrite \eqref{eq:R^(-1)} as follows:
\begin{align}
\notag
R^{-1}_y(2n+k,k)
&=
\sum_{m\ge0}
\underset{c_1=0}{\sum_{(c_i)\in\mathcal P^o_{2n+m,m}}}
(-1)^m2^{n}\frac {(2n+m+k-1)(2n+m+k-2)\cdots k} {(2n)!}\\[1mm]
\notag
&\kern3cm
\cdot
\frac {(2n)!} {2!^{c_3}c_3!\,4!^{c_5}c_5!\cdots (2n)!^{c_{2n+1}}c_{2n+1}!}
\prod _{i=1} ^{2n+1}
\left(\frac {y\big(\tfrac {i-1} {2}\big)}i\right)^{c_i}\\[3mm]
\notag
&=
\sum_{m\ge0}
\underset{c_1=0}{\sum_{(c_i)\in\mathcal P^o_{2n+m,m}}}
(-1)^m\frac {(2n+m+k-1)(2n+m+k-2)\cdots k} {n!\,(2n-1)!!}\\
&\kern2.5cm
\cdot
\frac {(2n)!} {2!^{c_3}c_3!\,4!^{c_5}c_5!\cdots (2n)!^{c_{2n+1}}c_{2n+1}!}
\prod _{i=1} ^{2n+1}
\left(\frac {y\big(\tfrac {i-1} {2}\big)}i\right)^{c_i}.
\label{eq:R^(-1)a}
\end{align}

We want to consider \eqref{eq:R^(-1)a} modulo~$2^e$, and our goal is
to show that, modulo~$2^e$, this expression is periodic with period
length~$2^e$ as a (polynomial) function in~$k$.

Now, the first term in the last line of \eqref{eq:R^(-1)a} is an
integer due to the fact that
$$
\sum_{i=1}^{2n+1}(2i)c_i=
\sum_{i=1}^{2n+1}(2i+1)c_i-\sum_{i=1}^{2n+1}c_i=
(2n+m)-m=2n,
$$
which allows the application of Lemma~\ref{lem:3} with $(c_i)$
replaced by 
$$(0,c_3,0,c_5,0,\dots,0,c_{2n+1}).$$
The product over~$i$ in the last line of \eqref{eq:R^(-1)a} is a
rational number with an odd denominator, since $c_i=0$ for all even~$i$. Thus,
the last line of \eqref{eq:R^(-1)a} represents a rational number with
an odd denominator, which has a non-negative $2$-adic valuation.
Likewise, the term $(2n-1)!!$ in the denominator of the expression in
the first line of \eqref{eq:R^(-1)a} is an odd number, and
consequently also does not influence the 2-adic analysis of this
expression.

We substitute $k+2^e$ in place of~$k$
in the remaining terms in the first line of \eqref{eq:R^(-1)a}
(ignoring the sign $(-1)^m$), and compute
\begin{align*}
&\frac {(2n+m+k+2^e-1)(2n+m+k+2^e-2)\cdots (k+2^e)} {n!}
\\[2mm]
&\kern.5cm
=
\binom {2n+m+k+2^e-1}n
{(n+m+k+2^e-1)(n+m+k+2^e-2)\cdots (k+2^e)} \\[2mm]
&\kern.5cm
\equiv
\binom {2n+m+k+2^e-1}n
{(n+m+k-1)(n+m+k-2)\cdots k} \pmod {2^e}\\[2mm]
&\kern.5cm
=
(2n+m+k+2^e-1)(2n+m+k+2^e-2)\cdots(n+m+k+2^e)\\
&\kern2cm
\times
{(n+m+k-1)(n+m+k-2)\cdots (n+k) }
\binom {n+k-1}n\pmod {2^e}\\[2mm]
&\kern.5cm
\equiv 
\frac {(2n+m+k-1)(2n+m+k-2)\cdots k} {n!}\pmod {2^e}.
\end{align*}
If we use this congruence in \eqref{eq:R^(-1)a}, then the periodicity
of this expression in~$k$ with period length~$2^e$ becomes apparent.

\medskip
This finishes the proof of the theorem.
\end{proof}

\section{The sequence $(d(n))_{n\ge0}$ modulo powers of $2$}
\label{sec:8}

Now we are in the position to prove our second main result, namely Part~3 of 
Theorem~\ref{thm:main}.

\begin{theorem} \label{thm:11}
The sequence $(d(n))_{n\ge0},$ when taken modulo any fixed $2$-power $2^e$ with $e\ge3,$
is purely periodic with {\em(}not necessarily minimal\/{\em)} 
period length~$2^{e-1}$. Modulo~$4,$ the sequence is 
purely periodic with period length~$4,$ the first few
values of the sequence {\em(}modulo~$4${\em)} being given by 
$$
1,1, 3,  3  , 1 ,\dots.
$$
\end{theorem}

\begin{proof}
Suppose first that $e\ge3$.
By \eqref{eq:R^(-1)} with $y(i)=u(i)$ for all~$i$, we have\linebreak 
$R^{-1}(2n,2k)\equiv 0$~(mod~$2^e$) for
$n\ge k+e$, due to the factor $2^{n-k}$ that arises under the
corresponding substitution.
Thus, the relation \eqref{eq:d-v} becomes
\begin{align} \notag
d(n)&\equiv\sum_{k=\max\{0,n-e+1\}}^{n}R^{-1}(2n,2k)v(k)\quad (\text{mod }2^e)\\[1mm]
&\equiv\sum_{k=0}^{\min\{n,e-1\}}R^{-1}(2n,2n-2k)v(n-k)\quad (\text{mod }2^e).
\label{eq:d-v-2^e} 
\end{align}
By Theorem~\ref{prop:2} with $x(j)=
\prod _{\ell=1} ^{j}(4\ell-3)^2$ for all~$j$,
the sequence $(v(n))_{n\ge0}$ is (purely) periodic
modulo~$2^e$ with period length~$2^{e-1}$, and,
by Theorem~\ref{thm:4} with $y(j)=u(j)$ for all~$j$, 
$k$ and $n$ interchanged, and subsequently $n$ replaced by~$2n-2k$, 
the sequence
$\big(R^{-1}(2n,2n-2k))_{n\ge0}$ is
(purely) periodic modulo~$2^e$ with period length~$2^{e-1}$. 
There is a little detail that needs to be addressed here:
what is the meaning of $R^{-1}(2n,2n-2k)$ if $n<k$?
So far we have not given a meaning to this in that case.
However, the expression for $R^{-1}(2n+k,k)$ provided by
\eqref{eq:R^(-1)} makes it easy to define the appropriate extension.
Namely, as we said earlier, this expression is a polynomial in~$k$ and
therefore it gives a meaning to $R^{-1}(2n+k,k)$ for all --- positive
or negative --- integers~$k$. Moreover, as long as $-2n\le k<0$,
the term $R^{-1}(2n+k,k)$ vanishes because of the appearance of the
binomial coefficient $\binom {2n+m+k-1}{2n+m}$ in \eqref{eq:R^(-1)}.
Furthermore, the periodicity argument still applies to the extended
sequence. Altogether, this allows us to relax the upper bound on the
summation index~$k$ in \eqref{eq:d-v-2^e} to
$$
d(n)
=\sum_{k=0}^{e-1}R^{-1}(2n,2n-2k)v(n-k)\quad (\text{mod }2^e).
$$
The above arguments show in particular that 
each summand on the right-hand side 
is (purely) periodic modulo~$2^e$ with (not necessarily minimal) 
period length~$2^{e-1}$. Since these are finitely many summands,
the same must hold for~$d(n)$.

\medskip
The assertion of the theorem concerning the behaviour of the sequence
modulo~$4$ is a direct consequence of Theorem~\ref{prop:2a} with
$x(1)=1$ and Theorem~\ref{thm:4} with $y(j)=u(j)$ for all~$j$.
\end{proof}

\section{The inverse of the matrix $\mathbf R$ modulo prime powers $p^e$ with
$p\equiv1~(\text{mod }4)$}
\label{sec:9}

This section prepares for the proof of our third main result, given in
the subsequent section. The goal is to establish periodicity of
$R^{-1}(2n,2k)$ modulo powers of a prime~$p$ with $p\equiv1$~(mod~4),
when considered as a sequence in~$n$. This will eventually be achieved
in Theorem~\ref{thm:12}. Similar to Section~\ref{sec:7}, we actually prove a polynomial
refinement in which the number $u(j)$ that appears in the definition
of the matrix entries $R^{-1}(2n,2k)$ (cf.\ \eqref{eq:R^{-1}}) is
replaced by an integer $y(j)$, $j=0,1,\dots$, with the restriction
that $y(j)$ is divisible by $p^e$ for $j\ge\cl{ep/2}$.
The corresponding analysis is the most demanding one in this
article. Already the final periodicity result in~\eqref{eq:R-1per}
(with ``supporting" powers $p^{\fl{2k/p}}$ on both sides) indicates
that matters are much more delicate here. While our starting point 
is again the expansion~\eqref{eq:R^(-1)} (with the
appropriate substitutions; see \eqref{eq:TSumme}),
the analysis of the summand modulo powers
of~$p$ is much more intricate here. The corresponding auxiliary
results are the subject of Lemmas~\ref{lem:16}--\ref{lem:13}.

We start with a lower bound for the $p$-adic valuation of the summand 
in~\eqref{eq:TSumme} in the ``generic" case.

\begin{lemma} \label{lem:16}
Let $p$ be a prime with $p\equiv1$~{\em(mod $4$)} and let 
$\big(y(n)\big)_{n\ge0}$ be an integer sequence with the property that
$v_p\big(y(n)\big)\ge e$ for $n\ge\cl{\frac {ep} {2}}$.
For all positive integers $n$ and $k,$ and non-negative integers $m$
and $c_i,$ $1\le i\le 2n-2k,$
with $\sum_{i=1}^{2n-2k}ic_i=2n-2k+m,$ $\sum_{i=1}^{2n-2k}c_i=m,$ and
$c_i=0$ for all even~$i,$ we have
\begin{multline} \label{eq:max11} 
v_p\left(\frac {(2n+m-1)!} {(2k-1)!
  \prod _{i=1} ^{2n-2k}i!^{c_i}\,c_i!}\prod _{i=1} ^{2n-2k}
{y^{c_i}\big(\tfrac {i-1} {2}\big)}\right)\\[1mm]
\ge
\#\left(\text{\em carries when adding $\left(2n+m-1-pc_p\right)_p$ and
  $(pc_p)_p$}\right)\kern2cm\\[1mm]
+\#\left(\text{\em carries when adding $\left(2n-2k+m-pc_p\right)_p$ and $(2k-1)_p$}\right)\\[1mm]
+\tfrac {1} {2(p-1)}(2n-2k+m-pc_p)
-\tfrac {1} {p-1}s_p(2n-2k+m-pc_p).
\end{multline}
\end{lemma}

\begin{remarknu} \label{rem:-1/2}
For non-negative integers $N$, we have 
$$
\tfrac {1} {2(p-1)}N-\tfrac {1} {p-1}s_p(N)\ge-\tfrac {1} {2}.
$$
Thus, although the last line in \eqref{eq:max11} can be negative, the
sum of the two numbers of carries on the right-hand side
of~\eqref{eq:max11} is still a lower bound for the
$p$-adic valuation on the left-hand side since the latter is an integer.
\end{remarknu}

\begin{proof}[Proof of Lemma~{\em \ref{lem:16}}]
By Legendre's formula in Lemma~\ref{lem:4},
the $p$-adic valuation on the left-hand side of \eqref{eq:max11} equals
\begin{multline} \label{eq:aux10}
v_p\left(\frac {(2n+m-1)!} {(2k-1)!
  \prod _{i=1} ^{2n-2k}i!^{c_i}\,c_i!}\prod _{i=1} ^{2n-2k}
{y^{c_i}\big(\tfrac {i-1} {2}\big)}\right)\\[1mm]
=\frac {1}
{p-1}\left(\vphantom{\sum_p^i}2n-2k+m-s_p(2n+m-1)+s_p(2k-1)\right.
\kern3cm\\[1mm]
\left.
-\sum_{i=1} ^{2n-2k}\Big(c_ii-c_is_p(i)+c_i-s_p(c_i)\Big)
\right)+\sum _{i=1} ^{2n-2k}
c_iv_p\Big(y\big(\tfrac {i-1} {2}\big)\Big).
\end{multline}
We consider the terms involving $c_i$ on the right-hand side of
\eqref{eq:aux10} separately. Concerning these, we claim that the
following lower bound holds for all odd~$i\ge3$ with~$i\ne p$:
\begin{equation} \label{eq:aux12} 
\frac {1} {p-1}\big(-(i+1)c_i+c_is_p(i)+s_p(c_i)
\big)+
c_iv_p\Big(y\big(\tfrac {i-1} {2}\big)\Big)
\ge -\frac {ic_i} {2(p-1)}.
\end{equation}
Multiplying both sides by $2(p-1)$ and using our divisibility
assumption for $y\big(\frac {i-1} {2}\big)$, we see that the above
claim will follow from the estimate 
\begin{equation*} 
-(i+2)c_i+2c_is_p(i)+2s_p(c_i)+
2(p-1)c_i\fl{\tfrac {i-1} {p}}
\ge 0.
\end{equation*}
This inequality can now indeed easily be verified for $3\le i<p$,
for $p< i<2p$, for $i=2p+1$, and finally for $2p+3\le i$ making use
of the simple inequality $\fl{\tfrac {i-1} {p}}\ge \tfrac {i-p} {p}$.
This proves our claim in \eqref{eq:aux12}.

We use \eqref{eq:aux12} in \eqref{eq:aux10} for $i\ne p$, keeping the
terms for $i=p$ unchanged. Thereby, we obtain
\begin{align}
\notag
v_p&\left(\frac {(2n+m-1)!} {(2k-1)!
\prod _{i=1} ^{2n-2k}i!^{c_i}\,c_i!}\prod _{i=1} ^{2n-2k}
{y^{c_i}\big(\tfrac {i-1} {2}\big)}\right)\\[1mm]
\notag&\ge\frac {1} {p-1}\left(\vphantom{\underset f{\sum_p^i}}
2n-2k+m-pc_p+s_p(c_p)-s_p(2n+m-1)+s_p(2k-1)
-\frac {1} {2}\underset {i\ne p}{\sum_{i=1} ^{2n-2k}}ic_i\right)\\[1mm]
\notag&=\frac {1} {2(p-1)}(2n-2k+m-pc_p)
+\frac {1} {p-1}\big(s_p(pc_p)-s_p(2n+m-1)+s_p(2k-1)
\big)\\[1mm]
\notag
&=\frac {1} {2(p-1)}(2n-2k+m-pc_p)-\frac {1} {p-1}s_p(2n-2k+m-pc_p)\\
\notag
&\kern2cm
+\frac {1} {p-1}\big(s_p(pc_p)+s_p(2n+m-1-pc_p)-s_p(2n+m-1)\big)\\
&\kern2cm
+\frac {1} {p-1}
\big(s_p(2k-1)+s_p(2n-2k+m-pc_p)-s_p(2n+m-1-pc_p)
\big).
\label{eq:aux13}
\end{align}
By Lemma~\ref{lem:4-5}, the next-to-last line in
\eqref{eq:aux13} equals the first number of carries that appears on the
right-hand side of~\eqref{eq:max11}, while
the last line in~\eqref{eq:aux13} equals the second number of carries
on the right-hand side of~\eqref{eq:max11}.
Thus, we have established~\eqref{eq:max11}.
\end{proof}

In a special case, the bound from the previous lemma can be slightly improved.

\begin{lemma} \label{lem:17}
To the assumptions of Lemma~{\em\ref{lem:16}} we add the conditions
that $c_i=0$ for $i\ge2p$ and that $c_i<p$ for $1\le i<p$ and for $p<i<2p$.
Then we have
\begin{multline} \label{eq:max12} 
v_p\left(\frac {(2n+m-1)!} {(2k-1)!
  \prod _{i=1} ^{2n-2k}i!^{c_i}\,c_i!}\prod _{i=1} ^{2n-2k}
{y^{c_i}\big(\tfrac {i-1} {2}\big)}\right)\\[1mm]
\ge
\#\left(\text{\em carries when adding $\left(2n+m-1-pc_p\right)_p$ and
  $(pc_p)_p$}\right)\kern2cm\\
+\#\left(\text{\em carries when adding $\left(2n-2k+m-pc_p\right)_p$ and
  $(2k-1)_p$}\right)\\
+\tfrac {1} {p-1}(2n-2k+m-pc_p)-\tfrac {1} {p-1}s_p(2n-2k+m-pc_p).
\end{multline}
\end{lemma}

\begin{remarknu} \label{rem:max12}
The only difference between \eqref{eq:max12} and \eqref{eq:max11} is a
``missing" factor of~2 in the denominator of the first term in the last line
of~\eqref{eq:max12}. Now, 
by Legendre's formula in Lemma~\ref{lem:4}, this last line 
equals the $p$-adic valuation of $(2n-2k+m-pc_p)!$
and is therefore non-negative (as opposed to the last line in
\eqref{eq:max11}; cf.\ Remark~\ref{rem:-1/2}).
\end{remarknu}

\begin{proof}[Proof of Lemma~{\em \ref{lem:17}}]
We use again the identity \eqref{eq:aux10}.
At this point, we observe that the additional conditions on the
$c_i$'s imply that the contribution of $c_i$ with $1\le i<2p$ and
$i\ne p$ in \eqref{eq:aux10} is non-negative. Indeed, this contribution is
\begin{multline} \label{eq:aux14}
-\frac {1}
      {p-1}\big(c_ii-c_is_p(i)+c_i-s_p(c_i)\big)+c_iv_p\Big(y\big(\tfrac
      {i-1} {2}\big)\Big) \\[1mm]
=
-\frac {1} {p-1}\big(ic_i-c_is_p(i)\big)+c_iv_p\Big(y\big(\tfrac {i-1} {2}\big)\Big).
\end{multline}
since $c_i<p$. If $i<p$, then we have $s_p(i)=i$ and hence the value
in \eqref{eq:aux14} is non-negative.
If, on the other hand, $p<i<2p$, then we have $s_p(i)=i-p+1$ and
$v_p\Big(y\big(\tfrac
{i-1} {2}\big)\Big)\ge1$, and therefore the value in \eqref{eq:aux14}
is again non-negative.

For the remaining terms in \eqref{eq:aux10}, we compute
\begin{multline}
\frac {1}
      {p-1}\big(2n-2k+m-s_p(2n+m-1)+s_p(2k-1)-pc_p+s_p(c_p)\big)
\\[1mm]
=
\frac {1}
      {p-1}(2n-2k+m-pc_p)-\frac {1} {p-1}s_p(2n-2k+m-pc_p)\\
+\frac {1} {p-1}\big(s_p(2n+m-1-pc_p)+s_p(pc_p)-s_p(2n+m-1)\big)\\
+\frac {1} {p-1}\big(s_p(2n-2k+m-pc_p)+s_p(2k-1)-s_p(2n+m-1-pc_p)\big).
\label{eq:aux16}
\end{multline}
By Lemma~\ref{lem:4-5}, the last line is the number of carries
that appears in the next-to-last line in \eqref{eq:max12}. 
For the same reason, the next-to-last line in \eqref{eq:aux16} is
the number of carries that appears in the third line from below in
\eqref{eq:max12}. This completes
the proof.
\end{proof}

The next lemma collects together several auxiliary inequalities that concern
the right-hand side of~\eqref{eq:max11}.
They will be used in subsequent arguments.

\begin{lemma} \label{lem:18}
Let $p$ be a prime with $p\equiv1$~{\em(mod~$4$)}
and $N$ and $\ell$ non-negative integers with $p^\ell\le N<p^{\ell+1}$. Then
\begin{alignat}2 \label{eq:N-1}
\fl{\tfrac {N} {p}}&\ge \ell,&&\quad \text{for }\ell\ge0,\\[1mm]
\label{eq:N-5}
\tfrac {1} {2(p-1)}N-\tfrac {1} {p-1}s_p(N)&> \tfrac {N} {2p}-\tfrac {3} {2},
&&\quad \text{for }\ell\ge0,\\[1mm]
\label{eq:N-2}
\tfrac {1} {2(p-1)}N-\tfrac {1} {p-1}s_p(N)&\ge \ell+\tfrac {3} {8},
&&\quad \text{for }\ell\ge2,\\[1mm]
\label{eq:N-3}
\tfrac {1} {2(p-1)}N-\tfrac {1} {p-1}s_p(N)&\ge \ell,
&&\quad \text{for }\ell\ge1\text{ and }N\ge p^\ell(p-1),\\[1mm]
\label{eq:N-4}
\tfrac {1} {2(p-1)}N-\tfrac {1} {p-1}s_p(N)&>0,
&&\quad \text{for }2p\le N<p^2.
\end{alignat}
\end{lemma}
\begin{proof}
The first inequality is obvious.

For the proof of \eqref{eq:N-5}, we write $N=n_0+n_1p+n_2p^2$ with
$0\le n_0,n_1<p$ and $n_2\ge0$, and then
estimate the difference of the
left-hand side and ${N} /(2p)$,
\begin{align*}
\tfrac {1} {2(p-1)}N-\tfrac {1} {p-1}s_p(N)-\tfrac {N} {2p}
&=\tfrac {N} {2p(p-1)}-\tfrac {1} {p-1}s_p(N)\\[1mm]
&=\tfrac {1} {2p(p-1)}\left(n_0+n_1p+n_2p^2-2pn_0-2pn_1-2ps_p(n_2)\right)\\[1mm]
&\ge\tfrac {1} {2p(p-1)}\left(-n_0(2p-1)-n_1p+n_2(p^2-2p)\right)\\[1mm]
&\ge-\tfrac {2p-1} {2p}
-\tfrac {1} {2}>-\tfrac {3} {2},
\end{align*}
as desired.

In order to show \eqref{eq:N-2},
let $N=n_0+n_1p+\dots+n_\ell p^\ell$ with $0\le n_i<p$ for all~$i$ and $n_\ell>0$.
Then we have
\begin{align}
\notag
  \tfrac {1} {2(p-1)}N-\tfrac {1} {p-1}s_p(N)
&=\tfrac {1} {2(p-1)}\big(n_0+n_1p+\dots+n_\ell p^\ell\big)
-\tfrac {1} {p-1}(n_0+n_1+\dots+n_\ell)\\[1mm]
\notag
&=\tfrac {1} {2(p-1)}\big(-n_0+n_1(p-2)+\dots+n_\ell (p^\ell-2)\big)\\[1mm]
&\ge\tfrac {1} {2(p-1)}\big(-(p-1)+(p^\ell-2)\big)\\[1mm]
&=\tfrac {1} {2(p-1)}(p^\ell-2)-\tfrac {1} {2}.
\label{eq:aux15}
\end{align}
Now, it is a simple exercise to show that
$$
p^\ell\ge 2(p-1)\ell+p^2-4p+4, \quad \text{for }\ell\ge2\text{ and }p\ge5.
$$
We use this in \eqref{eq:aux15} to get
$$
  \tfrac {1} {2(p-1)}N-\tfrac {1} {p-1}s_p(N)\ge 
\ell +\tfrac {1} {2(p-1)}(p^2-4p+2)-\tfrac {1} {2}
=\ell +\tfrac {1} {2}(p-4)-\tfrac {1} {2(p-1)}
\ge \ell+\tfrac {3} {8},
$$
since $p\ge 5$.

For the proof of the strengthening in \eqref{eq:N-3}, we observe that, instead of
the last line in~\eqref{eq:aux15}, we may improve that lower bound to
$$
  \tfrac {1} {2(p-1)}N-\tfrac {1} {p-1}s_p(N)
\ge\tfrac {1} {2(p-1)}\big(-(p-1)+(p-1)(p^\ell-2)\big)
=\tfrac {1} {2}(p^\ell-3).
$$
The fact that this is at least $\ell$ for $\ell\ge1$ and $p\ge5$ is
straightforward to verify.

Finally we turn to the proof of \eqref{eq:N-4}. We write $N=n_0+n_1p$
with $0\le n_0<p$ and $2\le n_1<p$, and compute
\begin{align*}
\tfrac {1} {2(p-1)}N-\tfrac {1} {p-1}s_p(N)
&=\tfrac {1} {2(p-1)}(n_0+n_1p)-\tfrac {1} {p-1}(n_0+n_1)\\[1mm]
&=\tfrac {1} {2(p-1)}\big(-n_0+n_1(p-2)\big)\\[1mm]
&\ge\tfrac {1} {2(p-1)}\big(-(p-1)+2(p-2)\big)\\[1mm]
&=\tfrac {1} {2(p-1)}\big(p-3\big)>0,
\end{align*}
since $p\ge5$.

This completes the proof of the lemma.
\end{proof}

The purpose of the lemma below is to ``convert" the lower bound from
Lemmas~\ref{lem:16} and~\ref{lem:17} into a form that is needed in the
proof of the upcoming Theorem~\ref{thm:12}.

\begin{lemma} \label{lem:13}
Let $p$ be a prime with $p\equiv1$~{\em(mod $4$)} and let 
$\big(y(n)\big)_{n\ge0}$ be an integer sequence with the property that
$v_p\big(y(n)\big)\ge e$ for $n\ge\cl{\frac {ep} {2}}$.
For all positive integers $n$ and $k,$ and non-negative integers $m$
and $c_i,$ $1\le i\le 2n-2k,$
with $\sum_{i=1}^{2n-2k}ic_i=2n-2k+m,$ $\sum_{i=1}^{2n-2k}c_i=m,$ and
$c_i=0$ for all even~$i,$ we have
\begin{equation} \label{eq:max6} 
\fl{\tfrac {2k} {p}}
+v_p\left(\frac {(2n+m-1)!} {(2k-1)!
  \prod _{i=1} ^{2n-2k}i!^{c_i}\,c_i!}\prod _{i=1} ^{2n-2k}
{y^{c_i}\big(\tfrac {i-1} {2}\big)}\right)\\
\ge1+\max_{c_p<i\le \fl{(2n+m-1)/p}}v_p(i).
\end{equation}
By convention, if $c_p=\fl{(2n+m-1)/p}$ the right-hand side is interpreted as
  zero {\em(}so that the inequality is trivially true{\em)}.
\end{lemma}

\begin{proof}
We proceed in a similar manner as in the proof of Lemma~\ref{lem:9}.

Choose $\al$, $\be$, and $\ga$ such that $2n+m-1=n_0p^\al+n_1$,
with $n_0$ not divisible by~$p$ and
 $n_1<p^{\be+1}$, 
$p^\be\le 2n+m-1-pc_p<p^{\be+1}$, and $p^\ga\le
2k-1<p^{\ga+1}$.
These conditions imply that the $p$-adic
representations can be schematically indicated as
\begin{alignat}5
\notag
&&&\underset{\downarrow}{\scriptstyle \al}\kern24pt
&\underset{\downarrow}{\scriptstyle \be}\kern-8pt
&&\underset{\downarrow}{\scriptstyle \ga}\kern17pt
\\
\notag
\left(2n+m-1\right)_p&={}\dots\,&&?\,0\,\dots\,0&&*\dots&\,\dots.\,&&&\dagger\\
\notag
\left(2n+m-1-pc_p\right)_p&=&&&&{\kern3pt}?\,\dots&\,\dots.\,&&&\dagger\\
\left(2k-1\right)_p&=&&&&&\kern0pt?\,\dots&&&*
\label{eq:c_p}
\end{alignat}
with the meaning that the shown $?$ in the first line
is the $\al$-th digit counted from right (the counting starting with~$0$),
that the left-most (non-zero) digit of the $p$-adic representation
of $2n+m-1-pc_p$
is the $\be$-th digit, and that the left-most (non-zero) digit of the $p$-adic representation
of $2k-1$
is the $\ga$-th digit. The symbol $*$ indicates an arbitrary digit
(between $0$ and~$p-1$) --- {\it not necessarily the same} in the first and
in the third line, the symbol $?$ indicates an arbitrary
{\it non-zero} digit (between $1$ and $p-1$) --- {\it not necessarily the same} in
lines~1--3,
while $\dagger$ indicates a digit which {\it is the same} in the
first and in the second line. The reader is advised to
consult~\eqref{eq:c_p} constantly while going through the subsequent arguments.

We first dispose of the simple case in which $\be=0$ (and hence
$\ga=0$): in that case, by considering the schematic representation in
\eqref{eq:c_p}, it becomes apparent that
$c_p=\fl{(2n+m-1)/p}$. According to the convention that we have made
in the statement of the lemma, there is nothing to prove.
Hence, from now on, we assume that $\be\ge1$.

\medskip
Now we distinguish two main cases, depending on the relative sizes
of~$n_1$ and\linebreak $2n+m-1-pc_p$.

\medskip
{\sc Case 1: $n_1<2n+m-1-pc_p$}. 
Here, by inspection of \eqref{eq:c_p}, we see that 
\begin{equation} \label{eq:max8} 
\max\limits_{c_p<i\le \fl{{(2n+m-1)} /{p}}}v_p(i)= \al-1,
\end{equation}
since we obtain the $p$-adic representation of $\fl{(2n+m-1)/p}$ 
by simply deleting the right-most digit (indicated by $\dagger$
in~\eqref{eq:c_p}) in the $p$-adic
representation of $2n+m-1$, and since $pc_p$ is the difference of the
first two lines in~\eqref{eq:c_p}.

On the other hand, by Lemma~\ref{lem:16}, 
the left-hand side of \eqref{eq:max6} is bounded below by
\begin{multline} \label{eq:carries-p} 
  \fl{\tfrac {2k} {p}}
+\#\left(\text{carries when adding $\left(2n+m-1-pc_p\right)_p$ and $(pc_p)_p$}\right)\\[1mm]
+\#\left(\text{carries when adding $\left(2n-2k+m-pc_p\right)_p$ and
  $(2k-1)_p$}\right)\\[1mm]
+\tfrac {1} {2(p-1)}(2n-2k+m-pc_p)
-\tfrac {1} {p-1}s_p(2n-2k+m-pc_p).
\end{multline}
Clearly, the number of carries when adding two numbers $A$ and $B$ in
their $p$-adic representations equals
the number of carries when performing the
subtraction of, say, $A$ from $A+B$.
In view of this observation, inspection of \eqref{eq:c_p} yields that the number of carries
in the first line of~\eqref{eq:carries-p} is at least
$\al-\be$. Hence, comparison of \eqref{eq:max6} with \eqref{eq:max8}
shows that what remains to demonstrate is that
\begin{multline} \label{eq:carries-p2} 
\fl{\tfrac {2k} {p}}
+\#\left(\text{carries when adding $\left(2n-2k+m-pc_p\right)_p$ and $(2k-1)_p$}\right)\\[1mm]
+\tfrac {1} {2(p-1)}(2n-2k+m-pc_p)
-\tfrac {1} {p-1}s_p(2n-2k+m-pc_p)\ge\be.
\end{multline}

In order to accomplish this, we need to discuss several subcases.

\medskip
If $\be>\ga+1$ and $\ga\ge0$, then
$2n-2k+m-pc_p\ge p^\be$ or $p^{\be-1}\le 2n-2k+m-pc_p< p^\be$.
In the former case we may use \eqref{eq:N-2} to conclude that the
left-hand side of \eqref{eq:carries-p2} is at least $\be+\frac {3}
{8}$, while in the latter case the number of carries in
\eqref{eq:carries-p2} must be at least~1 and $2n-2k+m-pc_p$ must be at
least $(p-1)p^{\be-1}$; consequently
\eqref{eq:N-3} leads to the conclusion
that the left-hand side of \eqref{eq:carries-p2}
is at least $1+(\be-1)=\be$, both confirming the inequality
in \eqref{eq:carries-p2}.

\smallskip
The case where $\be=1$ and $\ga=0$ needs a special treatment, to
which we will come at the end of this discussion of subcases.

\smallskip
If $\be=\ga\ge1$, then we use \eqref{eq:N-1} to see that $\fl{2k/p}$
is at least $\ga=\be$, again confirming \eqref{eq:carries-p2}.
We point out that for this conclusion we implicitly used the
observation of Remark~\ref{rem:-1/2}.

\smallskip
If $\be=\ga+1$ and $\ga\ge1$, then either $2n-2k+m-pc_p\ge p^\be$
or the number of carries in
\eqref{eq:carries-p2} is at least~1. In the former case, we may
use~\eqref{eq:N-2} to see that the last line on the left-hand side
of~\eqref{eq:carries-p2} is at least~$\be$, as desired. In the latter
case, the inequality~\eqref{eq:N-1} specialised to $N=2k>p^\ga$
implies that the left-hand side of
\eqref{eq:carries-p2} is at least $1+\ga=\be$, again confirming~\eqref{eq:carries-p2}.
Here also, this conclusion requires implicitly the
observation of Remark~\ref{rem:-1/2}.

\smallskip
Finally we discuss the remaining case where $\be=1$ and $\ga=0$.
Here, we have either
$2p\le 2n-2k+m-pc_p<p^2$, or
$p\le 2n-2k+m-pc_p< 2p$, or
$1\le 2n-2k+m-pc_p< p$.
In the first case, the inequality~\eqref{eq:N-4} implies that the
third line on the left-hand side of~\eqref{eq:carries-p2} is strictly positive. Since this
is a lower bound for the left-hand side of~\eqref{eq:max6}, an
integer, it is an effective lower bound of~$1=\be$, as desired.
For the second and third cases, we observe that, since by assumption
$$2n-2k+m-pc_p=\underset{i\ne p}{\sum_{i=1}^{2n-2k}}ic_i,$$
we must have $c_i=0$ for $i\ge 2p$, $c_i\le1$ for $p<i<2p$, and $c_i<p$
for $3\le i<p$. Thus, the conditions of Lemma~\ref{lem:17} are
satisfied. We may therefore bound the left-hand side
of~\eqref{eq:max6} by
\begin{multline} \label{eq:1-0}
\#\left(\text{carries when adding $\left(2n+m-1-pc_p\right)_p$ and $(pc_p)_p$}\right)\\[1mm]
+\#\left(\text{carries when adding $\left(2n-2k+m-pc_p\right)_p$ and
  $(2k-1)_p$}\right)\\[1mm]
+\tfrac {1} {p-1}(2n-2k+m-pc_p)-\tfrac {1} {p-1}s_p(2n-2k+m-pc_p).
\end{multline}
We have already found that the number of carries in the first line
is at least $\al-\be=\al-1$. On the other hand, under the
assumption on~$2n-2k+m-pc_p$ in the second case,
the expression in the third line equals~1. Together with~\eqref{eq:max8}, this confirms
again~\eqref{eq:max6}.
In the third case, there must be at least one carry when adding
$2n-2k+m-pc_p$ and $2k-1$ in their $p$-adic representations, meaning
that the second line in \eqref{eq:1-0} equals~1. This again confirms~\eqref{eq:max6}.  

\medskip
{\sc Case 2: $n_1\ge 2n+m-1-pc_p$}. 
Now, by inspection of \eqref{eq:c_p}, we see that 
\begin{equation*} 
\max\limits_{c_p<i\le \fl{{(2n+m-1)} /{p}}}v_p(i)\le \be-1.
\end{equation*}
A moment's thought will convince the reader that we are in exactly the
same situation as in Case~1: we must either prove
\eqref{eq:carries-p2} or,
in the case where $\be=1$ and $\ga=0$, show that
the sum of the second and third lines in~\eqref{eq:1-0} is at least~1.
Hence, the remaining steps are the same as the corresponding ones in Case~1, completing the proof.
\end{proof}

We are now in the position to prove the announced ``twisted"
periodicity of the matrix entries $R^{-1}_y(2n,2k)$ as given
by~\eqref{eq:Ry^{-1}} modulo prime
powers~$p^e$, when considered as a sequence in~$n$.

\begin{theorem} \label{thm:12}
Let $p$ be a prime with $p\equiv1$~{\em(mod $4$)} and let 
$\big(y(n)\big)_{n\ge0}$ be an integer sequence with the property that
$v_p\big(y(n)\big)\ge e$ for $n\ge\cl{\frac {ep} {2}}$.
For all positive integers $n,$ $k,$ and $e,$ 
we have
\begin{multline} \label{eq:R-1per}
p^{\fl{2k/p}}R^{-1}_y\left(2n+p^{e-1}(p-1),2k\right)\\
\equiv {y^{p^{e-1}}\big(\tfrac {p-1} {2}\big)}\cdot
(-1)^{(p-5)/4}
p^{\fl{2k/p}}R^{-1}_y\left(2n,2k\right)
\pmod{p^e},
\quad \text{for }n\ge e+1.
\end{multline}
In particular, if $y\big(\frac {p-1} {2}\big)$ is a quadratic residue modulo~$p$
and coprime to~$p,$ then the sequence
$\big(p^{\fl{2k/p}}R^{-1}_y\left(2n,2k\right)\big)_{n\ge e+1},$
when taken modulo any fixed $p$-power $p^e$ with $e\ge1,$
is purely periodic with {\em(}not necessarily minimal\/{\em)} 
period length $\frac {1} {4}p^{e-1}(p-1)^2$.
\end{theorem}

\begin{proof}
The last assertion follows from Euler's theorem, the fact that
$\varphi(p^e)=p^{e
  -1}(p-1)$, from our assumption that $y\big(\frac {p-1}
{2}\big)$ is a quadratic residue modulo~$p$ and coprime to~$p$, and the easily
proven property that, if $c\equiv d$~(mod~$p$), then
$c^{p^{e-1}}\equiv d^{p^{e-1}}$~(mod~$p^e$) for all positive integers~$e$.

\medskip
We turn to the proof of \eqref{eq:R-1per}.
Replacement of $k$ by $2k$ and of $n$ by $n-k$ in \eqref{eq:R^(-1)}
leads to
\begin{multline} 
R^{-1}_y(2n,2k)=
\sum_{m\ge0}
\underset{c_1=0}{\sum_{(c_i)\in\mathcal P^o_{2n-2k+m,m}}}
(-1)^m2^{n-k}\binom {2n+m-1} {2n-2k+m}\\[1mm]
\cdot
\frac {(2n-2k+m)!} {3!^{c_3}c_3!\,5!^{c_5}c_5!\cdots (2n-2k+1)!^{c_{2n-2k+1}}c_{2n-2k+1}!}
\prod _{i=1} ^{2n-2k+1}
      {y^{c_i}\big(\tfrac {i-1} {2}\big)}
      \pmod{p^e}.
\label{eq:TSumme}
\end{multline}
The earlier observation in \eqref{eq:m<n} also applies here --- with
$n$ replaced by $n-k$ --- so that $m\le n-k\le n$. In particular,
the sum over~$m$ in \eqref{eq:TSumme} is finite.

The fraction in the expression in \eqref{eq:TSumme} is an integer 
since, by Lemma~\ref{lem:3}, it is
the number of all partitions of~$\{1,2,\dots,2n-2k+m\}$ 
into $c_i$ blocks of size $i$, $i=3,5,\dots,2n-2k+1$.

Let $T(n,k,m,(c_i))$ denote the summand in
\eqref{eq:TSumme}; that is,
\begin{multline} \label{eq:Tshort}
  T(n,k,m,(c_i)):=
  (-1)^m2^{n-k}\frac {(2n+m-1)!} {(2k-1)!}\\
\times
\frac {1} {3!^{c_3}c_3!\,5!^{c_5}c_5!\cdots (2n-2k+1)!^{c_{2n-2k+1}}c_{2n-2k+1}!}
\prod _{i=1} ^{2n-2k+1}
      {y^{c_i}\big(\tfrac {i-1} {2}\big)}
\end{multline}
for a non-negative integer $m$ and
$(c_i)\in\mathcal P^{o}_{2n-2k+m,m}$ with $c_1=0$.
In view of the previous observation, $T(n,k,m,(c_i))$ is an integer
multiplied by a monomial in the $y(j)$'s.

\medskip
We want to prove the periodicity assertion in \eqref{eq:R-1per}.
We are going to prove that, for all positive integers~$a$, we have
\begin{equation} 
p^{\fl{2k/p}}T\big(n+\tfrac {1} {2}ap^{e-1}(p-1),k,m,(c_i))\equiv0\pmod{p^e},
\quad \text{for }n\ge e+1\text{ and }c_p<ap^{e-1},
\label{eq:Tcong1}
\end{equation}
and 
\begin{multline}
p^{\fl{2k/p}}T\big(n+\tfrac {1} {2}ap^{e-1}(p-1),k,m+ap^{e-1},(\tilde
c_i)\big)
\\[1mm]
\equiv 
{y^{ap^{e-1}}\big(\tfrac {i-1} {2}\big)}\cdot
(-1)^{a(p-5)/4}
p^{\fl{2k/p}}T\big(n,k,m,(c_i)\big)
\pmod{p^e},\quad \text{for }n<p^{e-1}(p-1),
\label{eq:Tcong2}
\end{multline}
with $\tilde c_p=c_p+ap^{e-1}$, and $\tilde c_i=c_i$ for all other~$i$.

\medskip
We claim that these two congruences together imply
\begin{multline} \label{eq:R-modpe} 
p^{\fl{2k/p}}R^{-1}_y\left(2n,2k\right)\\ \equiv
y^{ap^{e-1}}\big(\tfrac {p-1} {2}\big)\cdot
(-1)^{a(p-5)/4}
p^{\fl{2k/p}}R^{-1}_y\left(2n-ap^{e-1}(p-1),2k\right)
\pmod{p^e},
\end{multline}
where $a$ is maximal such that $n-\frac {1} {2}ap^{e-1}(p-1)\ge e+1$.
It is obvious that, in its turn, this congruence
implies~\eqref{eq:R-1per}, and thus the theorem.  

In order to prove the claim, we rewrite
the left-hand side of \eqref{eq:R-modpe} in
terms of our short notation~\eqref{eq:Tshort},
\begin{equation} \label{eq:R-T1} 
p^{\fl{2k/p}}R^{-1}_y\left(2n,2k\right)=
\sum_{m\ge0}
\underset{c_1=0}{\sum_{(c_i)\in\mathcal P^o_{2n-2k+m,m}}}
p^{\fl{2k/p}}T(n,k,m,(c_i)).
\end{equation}
We now consider the summands $p^{\fl{2k/p}}T(n,k,m,(c_i))$ on the right-hand side for
the various choices of $m$ and $(c_i)$.
Let $b$ be maximal such 
that $c_p-bp^{e-1}\ge0$.
If $b<a$, then we may use~\eqref{eq:Tcong1} with $a$
replaced by $b+1$ to conclude that the corresponding summand on the
right-hand side of~\eqref{eq:R-T1} vanishes modulo~$p^e$.
On the other hand, if $b\ge a$, then we
use~\eqref{eq:Tcong2} with $n$ replaced by $n-\frac {1}
{2}ap^{e-1}(p-1)$, $m$ replaced by $m-ap^{e-1}$ (this is indeed
non-negative since $m\ge c_p$),
and $c_p$ replaced by $c_p-ap^{e-1}$ to conclude that
\begin{multline*}
  p^{\fl{2k/p}}T(n,k,m,(c_i))\\[1mm]
  \equiv y^{ap^{e-1}}\big(\tfrac {p-1} {2}\big)\cdot
(-1)^{a(p-5)/4}
p^{\fl{2k/p}}T\left(n-\tfrac {1} {2}ap^{e-1}(p-1),k,m-ap^{e-1},(\hat
c_i)\right)
\pmod{p^e},\kern-3pt
\end{multline*}
with $\hat c_p=c_p-ap^{e-1}$ and $\hat c_i=c_i$ for all other~$i$.
Thus, we obtain
\begin{multline*}
p^{\fl{2k/p}}R^{-1}_y\left(2n,2k\right)\equiv
y^{ap^{e-1}}\big(\tfrac {p-1} {2}\big)
\cdot(-1)^{a(p-5)/4}
\\[1mm]
\times\sum_{m\ge0}
\underset{\hat c_1=0}
         {\sum_{(\hat c_i)\in\mathcal P^o_{2n-2k+m-ap^e,m-ap^{e-1}}}}
p^{\fl{2k/p}}T\left(n-\tfrac {1} {2}ap^{e-1}(p-1),k,m-ap^{e-1},(\hat c_i)\right)
\\\pmod{p^e},
\end{multline*}
with $a$ maximal such that $n-\frac {1} {2}ap^{e-1}(p-1)\ge e+1$.
Recalling~\eqref{eq:R-T1}, we are then directly led
to our claim~\eqref{eq:R-modpe}.

\medskip
Now we provide the proofs of the crucial congruences \eqref{eq:Tcong1}
and~\eqref{eq:Tcong2}.

\medskip
{\sc Proof of \eqref{eq:Tcong1}.}
Let $\fl{2k/p}=l$. If $l\ge e$, there is nothing to prove. Therefore,
we assume $l<e$ from now on.

If
$$
2\left(n+\tfrac {1} {2}ap^{e-1}(p-1)\right)-2k+m-pc_p\ge(2e-2l+1)p,
$$
then Lemma~\ref{lem:16} with $n$ replaced by $n+\frac {1}
{2}ap^{e-1}(p-1)$ and \eqref{eq:N-5} together imply that
$$v_p\big(T\big(n+\tfrac {1} {2}ap^{e-1}(p-1),k,m,(c_i)\big)
>\tfrac {(2e-2l+1)p} {2p}-\tfrac {3} {2}=e-l-1.
$$
In other words, the left-hand side --- being an integer --- is at
least $e-l$. In combination with $\fl{2k/p}=l$, this establishes
\eqref{eq:Tcong1} in this case.

On the other hand, if
$$
2\left(n+\tfrac {1} {2}ap^{e-1}(p-1)\right)-2k+m-pc_p<(2e-2l+1)p,
$$
then, since $m\ge c_p$ and $2k\le lp+(p-1)$, we infer
$$
2n+ap^{e-1}(p-1)-lp-(p-1)-(p-1)c_p<(2e-2l+1)p.
$$
Equivalently, we have 
$$
c_p>\tfrac {1} {p-1}\big(2n-2ep-p+lp\big)+ap^{e-1}-1.
$$
This entails
\begin{align*}
2\left(n+\tfrac {1} {2}ap^{e-1}(p-1)\right)+m-1&\ge
2n+ap^{e-1}(p-1)+c_p-1\\[1mm]
&>2n+ap^e+\tfrac {1} {p-1}\big(2n-2ep-p+lp\big)-2.
\end{align*}
Since, by assumption, we have $n\ge e+1$, it follows that
\begin{align*}
2\left(n+\tfrac {1} {2}ap^{e-1}(p-1)\right)+m-1
&>2(e+1)+ap^e+\tfrac {1} {p-1}\big(2(e+1)-2ep-p+lp\big)-2\\[1mm]
&\kern1cm =ap^e-1+\tfrac {1} {p-1}\big(1+lp\big).
\end{align*}
As the left-hand side is an integer, we conclude that 
\begin{equation} \label{eq:aux11} 
2\left(n+\tfrac {1} {2}ap^{e-1}(p-1)\right)+m-1
\ge ap^{e}.
\end{equation}

Now we use Lemma~\ref{lem:13} with $n$ replaced by $n+\frac {1}
{2}ap^{e-1}(p-1)$ to
see that\linebreak $p^{\fl{2k/p}}T\big(n+\frac {1} {2}ap^{e-1}(p-1),k,m,(c_i)\big)$ is divisible by 
$$
p^{1+\max\limits_{c_p< i\le \fl{(2n+ap^{e-1}(p-1)+m-1)/p}}v_p(i)}.
$$
As $c_p<ap^{e-1}$, and since, by \eqref{eq:aux11}, we have
$\fl{(2n+ap^{e-1}(p-1)+m-1)/p}\ge ap^{e-1}$, in the range 
$c_p<i\le \fl{(2n+ap^{e-1}(p-1)+m-1)/p}$ we will find $i=ap^{e-1}$, and
consequently $p^{\fl{2k/p}}T\big(n+\frac {1} {2}ap^{e-1}(p-1),k,m,(c_i)\big)$ 
is divisible by~$p^e$.

\medskip
{\sc Proof of \eqref{eq:Tcong2}.}
We substitute $n+\tfrac {1} {2}ap^{e-1}(p-1)$ for~$n$, $m+ap^{e-1}$ for~$m$,
and $c_p+ap^{e-1}$ for $c_p$ in the definition of $T(n,k,m,(c_i))$
in~\eqref{eq:Tshort}. 
After little manipulation, we get
\begin{multline} \label{eq:TExp1}
  T\left(n+\tfrac {1} {2}ap^{e-1}(p-1),k,m+ap^{e-1},(\tilde c_i)\right)\\[1.5mm]
  =y^{ap^{e-1}}\big(\tfrac {p-1} {2}\big)
  (-1)^{m+ap^{e-1}}2^{n+\frac {1} {2}ap^{e-1}(p-1)-k}
  \frac {(2n+m+ap^e-1)!} {(2k-1)!}\kern4.4cm\\
\times
\frac {1} {3!^{c_3}c_3!\,5!^{c_5}c_5!\cdots
  (2n-2k+1)!^{c_{2n-2k+1}}c_{2n-2k+1}!}\cdot\frac{c_p!}
      {p!^{ap^{e-1}}\,(c_p+ap^{e-1})!}
      \prod _{i=1} ^{2n-2k+1}
      {y^{c_i}\big(\tfrac {i-1} {2}\big)}.
\end{multline}

We have
\begin{equation} \label{eq:-1-Pot} 
(-1)^{p^{e-1}}=-1
\end{equation}
and
\begin{equation} \label{eq:2-Pot} 
2^{p^{e-1}(p-1)/2}\equiv (-1)^{(p-1)/4}\pmod{p^e},
\end{equation}
since for $p\equiv1$~(mod~$8$) the residue class of~2 is a
quadratic residue modulo~$p$, while for $p\equiv5$~(mod~$8$) it is not.
Furthermore, we have 
\begin{equation} \label{eq:pWilson} 
{(p-1)!}^{p^{e-1}}\equiv-1\pmod{p^e},
\end{equation}
as is seen by a straightforward induction on $e$ based on
Wilson's theorem. If we use \eqref{eq:-1-Pot}--\eqref{eq:pWilson} in
\eqref{eq:TExp1}, then we obtain
\begin{multline} 
  T\left(n+a\tfrac {p^{e-1}(p-1)} {2},k,m+ap^{e-1},(\tilde c_i)\right)\\
  \equiv y^{ap^{e-1}}\big(\tfrac {p-1} {2}\big)
(-1)^{m+a\frac {p-1} {4}}2^{n-k}  \frac {(2n+m+ap^e-1)!} {(2k-1)!}\kern6.2cm\\[1.5mm]
\times
\frac {1} {3!^{c_3}c_3!\,5!^{c_5}c_5!\cdots
  (2n-2k+1)!^{c_{2n-2k+1}}c_{2n-2k+1}!}\cdot\frac{c_p!}
      {p^{ap^{e-1}}\,(c_p+ap^{e-1})!}
      \prod _{i=1} ^{2n-2k+1}
            {y^{c_i}\big(\tfrac {i-1} {2}\big)}\\
            \pmod{p^e}.
      \label{eq:TExp2}
\end{multline}
We have
\begin{multline} 
\frac {(2n+m+ap^e-1)!} {(2n+m-1)!}\,p^{-ap^{e-1}}=
p^{-ap^{e-1}}(2n+m+ap^e-1)\cdots(2n+m+1)(2n+m)\\[1.5mm]
=\frac {\fl{(2n+m+ap^e-1)/p}!} {\fl{(2n+m-1)/p}!}
\cdot [(2n+m+ap^e-1)\cdots(2n+m+1)(2n+m)]_p,
\label{eq:[]_p}
\end{multline}
where $[a\cdot b\cdots z]_p$ denotes the product $a\cdot b\cdots z$ in
which all factors divisible by~$p$ are omitted.
Now we observe that $[b\cdot (b+1)\cdots(b+p^{e}-1)]_p$ forms a complete
set of representatives of the multiplicative group
$(\Z/p^e\Z)^\times$. Consequently, the product is congruent to $-1$
modulo~$p^e$. The term $[\,.\,]_p$ on the right-hand side of
\eqref{eq:[]_p} consists of $a$ such products. Therefore,
$$
\frac {(2n+m+ap^e-1)!} {(2n+m-1)!}\,p^{-ap^{e-1}}
\equiv(-1)^a\frac {\fl{(2n+m+ap^e-1)/p}!} {\fl{(2n+m-1)/p}!}
\pmod{p^e}.
$$
If we substitute this in \eqref{eq:TExp2}, then we get
\begin{align} 
\notag
  T&\left(n+a\tfrac {p^{e-1}(p-1)} {2},k,m+ap^{e-1},(\tilde c_i)\right)
  \equiv y^{ap^{e-1}}\big(\tfrac {p-1} {2}\big)
  (-1)^{m+a\frac {p-5} {4}}2^{n-k}
\frac {(2n+m-1)!} {(2k-1)!}\\[1.5mm]
\notag
&\kern2cm\times  \frac {\fl{(2n+m+ap^e-1)/p}!} {(c_p+ap^{e-1})!}
  \frac {c_p!} {\fl{(2n+m-1)/p}!}
\\
\notag
&\kern2cm\times
\frac {1} {3!^{c_3}c_3!\,5!^{c_5}c_5!\cdots
  (2n-2k+1)!^{c_{2n-2k+1}}c_{2n-2k+1}!}
      \prod _{i=1} ^{2n-2k+1}
      {y^{c_i}\big(\tfrac {i-1} {2}\big)}\\[1.5mm]
\notag
&\equiv y^{ap^{e
    -1}}\big(\tfrac {p-1} {2}\big)
(-1)^{m+a\frac {p-5} {4}}2^{n-k}
\frac {(2n+m-1)!} {(2k-1)!}
\prod _{i=c_p+1} ^{\fl{(2n+m-1)/p}}
\frac {i\cdot p^{-v_p(i)}+ap^{e-v_p(i)-1}} {i\cdot p^{-v_p(i)}}
\\
&\kern1cm\times
\frac {1} {3!^{c_3}c_3!\,5!^{c_5}c_5!\cdots
  (2n-2k+1)!^{c_{2n-2k+1}}c_{2n-2k+1}!}
      \prod _{i=1} ^{2n-2k+1}
            {y^{c_i}\big(\tfrac {i-1} {2}\big)}\pmod{p^e}.
      \label{eq:TExp3}
\end{align}
The reader should note that we wrote the first product over~$i$ in
this particular form in order to make sure that the expressions
$i\cdot p^{-v_p(i)}$ in the denominator are coprime to~$p$.

Similarly to the proof of \eqref{eq:cong3}, we would like to simplify the terms
$(i\cdot p^{-v_p(i)})+ap^{e-v_p(i)-1}$ to $i\cdot p^{-v_p(i)}$.
For, assuming the validity of this simplification, the first product
over~$i$ on the right-hand side of~\eqref{eq:TExp3} would simplify to~1,
and the remaining terms exactly equal
$y^{ap^{e
    -1}}\big(\tfrac {p-1} {2}\big)T\big(n,k,m,(c_i)\big)$.

We claim that this simplification is indeed allowed, {\it provided both
sides of~\eqref{eq:TExp3} are multiplied by $p^{\fl{2k/p}}$}.
(The reader should go back to~\eqref{eq:Tcong2} to see that this is
indeed what we need.)
For, by Lemma~\ref{lem:13}, we
know that the ``prefactor" of the first product over~$i$ on the
right-hand side of~\eqref{eq:TExp3} (that is, the right-hand side
of~\eqref{eq:TExp3} without that first product over~$i$), multiplied
by $p^{\fl{2k/p}}$, is an integer that is divisible by
$$
p^{1+\max\limits_{c_p< i\le \fl{(2n+m-1)/p}}v_p(i)}.
$$
Hence, instead of calculating modulo~$p^e$, we may
reduce the first product over~$i$ on the right-hand side
of~\eqref{eq:TExp3} modulo
$$
p^{e-1-\max\limits_{c_p< i\le \fl{(2n+m-1)/p}}v_p(i)}.
$$
(It should be observed here that the exponent in the last displayed expression is non-negative. Indeed,
as we observed at the beginning of this proof, we have $m\le n$. Furthermore,
by assumption, we have $n< p^{e-1}(p-1)$. Together, this
implies that $\fl{(2n+m-1)/p}\le 3p^{e-1}<p^e$.)
This is exactly what we need to perform the desired simplification and
the first product over~$i$ drops out. 

\medskip
This completes the proof of the theorem.
\end{proof}

\section{The sequence $(d(n))_{n\ge0}$ modulo prime powers $p^e$ with
  $p\equiv1$ (mod $4$)}
\label{sec:10}

We are now able to prove our third main result, this one concerning
the periodicity of
$d(n)$ modulo prime powers~$p^e$ with $p\equiv1$~(mod~4), announced in
Part~(2) of Theorem~\ref{thm:main}.

\begin{theorem} \label{thm:13}
Let $p$ be a prime number with $p\equiv1$~{\em (mod~$4$),} and let $e$ be some
positive integer.
Then the sequence $\big(d(n)\big)_{n\ge e+1}$ is purely periodic
modulo~$p^e$ with {\em(}not necessarily minimal\/{\em)} period
length~$\frac {1} {4}p^{e-1}(p-1)^2$.
\end{theorem}

\begin{proof}
We start by recalling \eqref{eq:d-v}, that is 
\begin{equation} \label{eq:d-v2}  
d(n)=\sum_{k=0}^{n}R^{-1}(2n,2k)v(k).
\end{equation}
By Theorem~\ref{thm:2A}, we know that $v(k)\equiv0$~(mod~$p^e$)
for $k\ge \cl{\frac {ep}2}$. Consequently, we may truncate the sum in
\eqref{eq:d-v2} when we consider both sides modulo~$p^e$.
In fact, Theorem~\ref{thm:2A} says more precisely that
$v(k)=p^{\fl{2k/p}}V(k,p)$, where $V(k,p)$ is an integer.
Altogether, this leads to
\begin{align} \notag
d(n)&\equiv\sum_{k=0}^{\fl{ep/2}}R^{-1}(2n,2k)p^{\fl{2k/p}}V(k,p)\pmod{p^e}\\[1mm]
&\equiv\sum_{k=1}^{\fl{ep/2}}R^{-1}(2n,2k)p^{\fl{2k/p}}V(k,p)
\pmod{p^e}, \quad \text{for }n\ge1.
\label{eq:d-v3}
\end{align}
We are indeed allowed to ignore the summand for $k=0$ since
$R^{-1}(2n,0)=0$ for $n\ge1$, cf.\ Proposition~\ref{prop:1}.
By Theorem~\ref{thm:12} with $y(k)=u(k)$ for all~$k$
(see in particular the last paragraph of the
statement; Theorem~\ref{thm:1A} provides the properties of $u(k)$
required by the theorem),
the sequence $\big(p^{\fl{2k/p}}R^{-1}(2n,2k)\big)_{n\ge e+1}$ is
purely periodic when taken modulo~$p^e$ with (not necessarily minimal)
period length~$\frac {1} {4}p^{e-1}(p-1)^2$. Since, by
\eqref{eq:d-v3}, the sequence $\big(d(n)\big)_{n\ge e+1}$, when taken
modulo~$p^e$, is a finite
linear combination of the sequences
$\big(p^{\fl{2k/p}}R^{-1}(2n,2k)\big)_{n\ge e+1}$, $k=1,2,\dots$,
it has the same periodicity behaviour.
\end{proof}

It should be observed that the above argument, combined
with~\eqref{eq:u-Pi1}, in fact proves a
refinement of the periodicity of the sequence $\big(d(n)\big)_{n\ge
  e+1}$ that generalises~\eqref{eq:Wakh}, namely
\begin{multline} \label{eq:d-u-p} 
d\left(n+\tfrac {p^{e-1}(p-1)} {2}\right)
\equiv
(-1)^{(p-5)/4}(3\cdot 7\cdot11\cdots(2p-3))^{2p^{e-1}}d(n)\pmod{p^e},\\
\text{for
}p\equiv1~(\text{mod 4})\text{ and }n\ge e+1.
\end{multline}

\section{Some conjectures and speculations}
\label{sec:Z}

In this final section, we report on some conjectures concerning congruence properties of
the sequences $\big(u(n)\big)_{n\ge0}$, $\big(v(n)\big)_{n\ge0}$,
and $\big(d(n)\big)_{n\ge0}$ that are suggested by extensive computer
experiments. If true, they would further strengthen the results of our paper.
Roughly speaking, the data suggest that it is possible to improve
period lengths and bounds for odd prime powers~$p^e$ by a factor of~2. By
contrast, it seems that, in our result in Section~\ref{sec:8} on
periodicity of $d(n)$ modulo powers of~2, the period length is
the exact one.

\medskip
Our first conjecture predicts that, for primes $p\equiv3$~(mod~4),
the vanishing of $u(n)$ modulo~$p^e$ occurs by a factor of~2 earlier
than proved in Theorem~\ref{thm:1}.

\begin{conjecture} \label{conj:1}

\medskip\noindent
{\em(1)}
If $p\equiv3$~{\em(mod~$4$)} and $e\ge1,$ 
we have $u(n)\equiv0$~{\em(mod~$p^{2e-1}$)}
for $n\ge \cl{\frac {ep^2-1} {2}}$.

\medskip\noindent
{\em(2)}
If $p\equiv3$~{\em(mod~$4$)} and $e\ge2,$ we have $u(n)\equiv0$~{\em(mod~$p^{2e}$)}
for $n\ge {\frac {ep^2+(e-2)p} {2}}$.

\medskip\noindent
{\em(2a)}
If $p\equiv3$~{\em(mod~$4$),} 
we have $u(n)\equiv0$~{\em(mod~$p^2$)}
for $n\ge {\frac {p^2-1} {2}}$.
\end{conjecture}

\begin{remark}
Computer experiments indicate that the above lower bounds can be improved
in two sporadic cases:

\medskip\noindent
{(1)}
If $p\equiv3$~{(mod~$4$)}, 
we have $u(n)\equiv0$~{(mod~$p$)}
for $n\ge {\frac {p^2-p} {2}}$.

\medskip\noindent
{(2)}
If $p\equiv3$~{(mod~$4$)}, we have $u(n)\equiv0$~{(mod~$p^6$)}
for $n\ge {\frac {3p^2-1} {2}}$.
%
\end{remark}

The next conjecture predicts an analogous strengthening of
Theorem~\ref{thm:2}, again for primes $p\equiv3$~(mod~4).

\begin{conjecture} \label{conj:2}
{\em(1)}
If $p\equiv3$~{\em(mod~$4$)} and $e$ is odd, 
we have $v(n)\equiv0$~{\em(mod~$p^e$)}
for $n\ge {\frac {(ep+2)(p+1)} {4}}$.

\medskip\noindent
{\em(2)}
If $p\equiv3$~{\em(mod~$4$)} and $e$ is even, 
we have $v(n)\equiv0$~{\em(mod~$p^e$)}
for $n\ge \cl{\frac {ep^2} {4}},$ ``with very few exceptions."
Based on data for $v(n)$ with $n\le 1500$, the only exceptions that we
found were

\begin{itemize} 
\item 
$p=7$ and $e=8,$ where
the correct lower bound is $102$ instead of $\cl{\frac {8\cdot 7^2}
  {4}}=98$;
\item 
$p=7$ and $e=24,$ where
the correct lower bound is $298$ instead of $\cl{\frac {24\cdot 7^2}
  {4}}=294$;
\item 
$p=7$ and $e=40,$ where
the correct lower bound is $494$ instead of $\cl{\frac {40\cdot 7^2}
  {4}}=490$;
\item 
$p=11$ and $e=36,$ where
the correct lower bound is $1095$ instead of $\cl{\frac {36\cdot 11^2}
  {4}}=1089$.
\end{itemize}

%
%
\end{conjecture}

For the sequence $\big(d(n)\big)_{n\ge0}$ and primes $p\equiv3$~(mod~4), it seems that
Theorem~\ref{thm:10} can be improved analogously.

\begin{conjecture} \label{conj:3}
If $p\equiv3$~{\em(mod~$4$),}
we have $d(n)\equiv0$~{\em(mod~$p^e$)}
for $n\ge \cl{\frac {ep^2} {4}}$.
\end{conjecture}

For primes $p\equiv1$~(mod~4), it appears that the periodicity of $d(n)$
modulo $p$-powers can be refined in the following way, which would 
improve Theorem~\ref{thm:13}.

\begin{conjecture} \label{conj:4}
{\em(1)}
If $p\equiv1$~{\em(mod~$4$),} the sequence $\big(d(n)\big)_{n\ge1},$
taken modulo~$p^e,$ is (eventually) periodic with {\em(}not necessarily minimal\/{\em)}
period length $\frac {1} {8}p^{e-1}(p-1)^2$.

\medskip\noindent
{\em(2)}
If $p\equiv1$~{\em(mod~$4$),} there exists a constant $C_{p,e}$ such that:

\begin{enumerate} 
\item [(i)]$d\big(n+\frac {p^{e-1}(p-1)} {4
}\big)\equiv C_{p,e}d(n)$~{\em(mod $p^e$)} for all $n\ge1$;
\vspace{1mm}
\item [(ii)]$C_{p,e}^{(p-1)/2}\equiv1$~{\em(mod $p^e$)}.
\end{enumerate}
\end{conjecture}

\begin{remark}
(1) From computer data, it seems that $\big(d(n)\big)_{n\ge1}$ is actually
{\it purely} periodic modulo~$p^e$ for a prime~$p$ with
$p\equiv1$~(mod~4).
However, that may be deceiving and just mean that one sees
counter-examples only if one goes to very high prime powers~$p^e$
(which however is difficult since the computation of $d(n)$ for large~$n$
quickly exceeds the capacity of computers).
In any case, with an arbitrary sequence $\big(y(n)\big)_{n\ge0}$
satisfying the conditions of Theorem~\ref{thm:12}, the
congruence~\eqref{eq:R-1per} is not true in general. Furthermore,
recall that we proved pure periodicity of $d(n)$ modulo~$p^e$
only starting from $n=e+1$.
Phrased differently, if pure
periodicity for $\big(d(n)\big)_{n\ge1}$ is true, then this would come
from very special properties of the sequences $\big(u(n)\big)_{n\ge0}$
and $\big(v(n)\big)_{n\ge0}$.

\medskip
(2) Item~(2) above is a strengthening
and generalisation of \cite[Conj.~18(2)]{WakhAA}.
\end{remark}

Concerning the matrix $\big(R(n,k)\big)_{n,k\ge0}$, we record the following ---
conjectural --- congruence properties. Item~(3) is a strengthening of
Theorem~\ref{thm:4} specialised to $y(k)=u(k)$ for all~$k$.

\begin{conjecture} \label{conj:5}
{\em(1)} For fixed~$k$ and any given prime power $p^e$
with $p\equiv1$~{\em(mod~4),} 
the sequence $(\mathbf R(2n+k,k))_{n\ge0},$ when considered
modulo~$p^e,$ is (purely) periodic.

\medskip
{\em(2)} For fixed~$k$ and any given prime power $p^e$
with $p\equiv3$~{\em(mod~4),} 
the sequence $(\mathbf R(2n+k,k))_{n\ge0},$ when considered
modulo~$p^e,$ is eventually~$0$.


\medskip
{\em(3)} For fixed~$a$ and any $2$-power $2^e,$ 
the sequence $\big(R^{-1}(2n+k,k)\big)_{k\ge0}$ is periodic modulo~$2^e$
with period length~$2^{e-3}$.
\end{conjecture}


An attentive reader may have observed earlier that we did not say
anything about the behaviour of $u(n)$ modulo powers of~2. The reason
is that we simply did not need to know more than $u(0)=1$, and that
$u(1)$ and $u(2)$ are odd, in order to prove periodicity of
$\big(d(n)\big)_{n\ge0}$ modulo powers of~2. However, data suggest 
a high divisibility of $u(n)$ by powers
of~2. Even much more seems to be true, namely
a ``hypergeometric" generalisation; see the conjecture below.
This high divisibility
might be necessary for a proof of Conjecture~\ref{conj:5}(3).

\begin{conjecture} \label{thm:5}
Let $(u(n))_{n\ge0}$ be defined by the recurrence \eqref{eq:Rek1-Pi}.
Then $\frac {u(n)} {(2n+1)!}\in\mathbb Z_2$. 
In other words, the rational number $\frac {u(n)} {(2n+1)!}$ can be
written with an odd denominator. In particular, we have $v_2\big(u(n)\big)
\ge 2n-\cl{\log_2(n)}$. More generally, it seems that any quotient
$$
\frac {
  {} _{2} F _{1} \!\left [ \begin{matrix} 
  \frac {3} {4}+a,\frac {3} {4}+b\\\frac {3} {2}+c\end{matrix} ; 
  {\displaystyle 4t}\right ]  }
{
  {} _{2} F _{1} \!\left [ \begin{matrix} 
  \frac {1} {4}+d,\frac {1} {4}+e\\\frac {1} {2}+f\end{matrix} ; 
  {\displaystyle 4t}\right ]  }
$$
with $a,b,c,d,e,f$ non-negative integers has coefficients in $\mathbb Z_2$.
\end{conjecture}

\section*{Acknowledgement}

We thank Tanay Wakhare for helpful correspondence.
The authors also thank the Mathematische Forschungsinstitut
Oberwolfach for the opportunity of a research fellowship
in August/September 2024, during which they succeeded to
improve the divisibility results
for primes $p\equiv3$~(mod~4) significantly.

\end{document}